\documentclass[reqno,12pt,a4paper]{amsart}
\usepackage{amsmath,amssymb,amsthm,graphicx,mathrsfs,url}
\usepackage[tmargin=3cm,bmargin=3.5cm,lmargin=2.5cm,rmargin=2.5cm,a4paper]{geometry}
\usepackage{enumitem}
\usepackage{mathtools, stackrel}
\usepackage[dvipsnames]{xcolor}
\usepackage[colorlinks=true,linkcolor=Blue,citecolor=Green]{hyperref}
\hypersetup{pdfstartview=XYZ}
\usepackage{tikz-cd}
\usetikzlibrary{tqft}
\usepackage[percent]{overpic} % percent option allows coordinates in %

\usetikzlibrary {patterns,patterns.meta}
\theoremstyle{plain}
\newtheorem{thm}{Theorem}
\newtheorem{lem}{Lemma}[section]
\newtheorem{prop}[lem]{Proposition}

\theoremstyle{definition}
\newtheorem{definition}[lem]{Definition}
\newtheorem{ex}[lem]{Example}

\theoremstyle{remark}
\newtheorem{rem}{Remark}[section]
 
\numberwithin{equation}{section}

\usepackage[normalem]{ulem} % For strikethrough
\newcommand{\C}{\mathbb{C}}
\newcommand{\R}{\mathbb{R}}

\newcommand{\Z}{\mathbb{Z}}
\newcommand{\D}{\mathcal{D}}
\newcommand{\N}{\mathbb{N}}

\newcommand{\HH}{\mathbb{H}}
\newcommand{\eps}{\varepsilon}

\newcommand{\norm}[1]{\left\Vert#1\right\Vert}

\newcommand*\Laplace{\mathop{}\!\mathbin\bigtriangleup}

\let\Im=\Imag

\DeclareMathOperator{\PSL}{PSL}

\let\Re=\Real

\DeclareMathOperator{\supp}{supp}

\def\Ddots{\mathinner{\mkern1mu\raise\p@
\vbox{\kern7\p@\hbox{.}}\mkern2mu
\raise4\p@\hbox{.}\mkern2mu\raise7\p@\hbox{.}\mkern1mu}}
\makeatother

\newcommand{\eklm}[1]{\left\langle #1 \right\rangle}

\renewcommand{\d}{\,d}

\newcommand{\F}{{\mathcal F}}

\newcommand{\M}{{\mathcal M}}

\renewcommand{\O}{{\mathcal O}}

\newcommand{\Xbf}{{\mathbf X}}

\renewcommand{\epsilon}{\vararepsilon}

\newcommand{\bdm}{\begin{displaymath}}
\newcommand{\edm}{\end{displaymath}}
\newcommand{\bq}{\begin{equation}}
\newcommand{\eq}{\end{equation}}
\newcommand{\bqn}{\begin{equation*}}
\newcommand{\eqn}{\end{equation*}}

\newcommand{\Cinft}{{C^{\infty}}}
\newcommand{\CT}{{C^{\infty}_c}}

\newcommand{\GL}{\mathrm{GL}}
\newcommand{\SO}{\mathrm{SO}}

\newcommand{\g}{{\mathfrak g}}

\newcommand{\aL}{{\mathfrak a}}
\newcommand{\nL}{{\mathfrak n}}
\renewcommand{\k}{{\mathfrak k}}

\newcommand{\p}{{\mathfrak p}}

\newcommand{\Ad}{\mathrm{Ad}}

\title[Patterson-Sullivan and Wigner distributions]{Patterson-Sullivan and Wigner distributions of convex-cocompact hyperbolic surfaces}

\author[B.~Delarue]{Benjamin Delarue}
\email{bdelarue@math.uni-paderborn.de}
\author[G.~Palmirotta]{Guendalina Palmirotta}
\email{gpalmi@math.uni-paderborn.de}
\address{Universit\"at Paderborn, Warburger Str.\ 100, 33098 Paderborn, Germany}

\date{\today}
\begin{document}

\maketitle

\begin{abstract}
We prove that the Patterson-Sullivan and Wigner distributions on the unit sphere bundle of a convex-cocompact hyperbolic surface are asymptotically identical. This generalizes results in the compact case by Anantharaman-Zelditch and Hansen-Hilgert-Schröder.\\

\end{abstract}

%\tableofcontents

%-------------------------------%%--------------------------------------%
\section{Introduction}
%----------------------------------------------
Let ${\Laplace}$ be the non-negative Laplace-Beltrami operator on a convex-cocompact hyperbolic surface $\mathbf{X}_\Gamma = \Gamma\backslash\HH^2$, where $\Gamma\subset\mathrm{PSL}(2,\R)$ is a convex-cocompact discrete group and $\HH^2$ the hyperbolic plane that we identify with the quotient $G/K=\mathrm{PSL}(2,\R)/\mathrm{SO}(2)$. 
By the classical works \cite{MazzeoMelrose,Guillope}, the $L^2$-resolvent
$$({\Laplace}-s(1-s))^{-1}: L^2(\mathbf{X}_\Gamma) \longrightarrow L^2(\mathbf{X}_\Gamma)$$
has a meromorphic extension from $\{s \in \C\,|\, \mathrm{Re}(s)>\frac{1}{2}\}$ to $\C$ as a family of continuous operators $$R_{{\Laplace}}(s):C^\infty_c(\mathbf{X}_\Gamma) \rightarrow C^\infty(\mathbf{X}_\Gamma).$$
The poles of $R_{{\Laplace}}(s)$ are called \emph{quantum resonances}. For each quantum resonance $s_0\in \C$, the image of the residue of $R_{{\Laplace}}(s)$ at $s=s_0$ is finite-dimensional and contains the non-zero space $\mathrm{Res}_{\Laplace}^1(s_0)\subset  C^\infty(\mathbf{X}_\Gamma)$ of \emph{quantum resonant states}. They are solutions $\phi$ of the eigenvalue problem  
$$
\Delta \phi=s_0(1-s_0)\phi
$$
with a particular asymptotic behavior towards the boundary at infinity of $\mathbf{X}_\Gamma$ (for details, see \cite[Eq.~(1.1)]{GHWa}). 

Given two quantum resonances $s_0,s_0' \in \C$  and two quantum resonant states $\phi\in \mathrm{Res}^1_{\Laplace}(s_0)$, $\phi'\in \mathrm{Res}^{1}_{\Laplace}(s_0')$,  one can associate to them a distribution $W_{\phi,\phi'}\in \D'(S\mathbf{X}_\Gamma)$ on the unit sphere bundle $S\mathbf{X}_\Gamma$ called the \emph{Wigner distribution} by the formula
    \begin{equation*}
         W_{\phi,\phi'}(u) := \langle \mathrm{Op}(u) \phi, \phi'\rangle_{L^2(\mathbf{X}_\Gamma)},\qquad u\in \CT(S\mathbf{X}_\Gamma),
    \end{equation*}
where the linear operator $\mathrm{Op}(u):C^\infty(\mathbf{X}_\Gamma)\to C^\infty_c(\mathbf{X}_\Gamma)$ is defined by a quantization $\mathrm{Op}$ introduced by Zelditch \cite{Zelditch19860} (see Section \ref{sect:Wigner} for details).   Wigner distributions are also known as \emph{microlocal lifts} or \emph{microlocal defect measures}. 

On the other hand, provided that $s_0,s_0'\not\in -\frac{1}{2}-\frac{1}{2}\N_0$, one can associate to $\phi,\phi'$ another distribution $\mathrm{PS}_{\phi,\phi'} \in \mathcal{D}'(S\mathbf{X}_\Gamma)$ called the  \emph{Patterson-Sullivan distribution} which is quasi-invariant under the classical evolution given by the geodesic flow $\varphi_t$ in the sense that 
\bq
\varphi_t^\ast \mathrm{PS}_{\phi,\phi'}=e^{(\bar s_0'-s_0)t}\mathrm{PS}_{\phi,\phi'}\quad \forall\, t\in \R,\label{eq:quasiinvariancePS}
\eq
where $\bar s_0'$ is the complex conjugate of $s'_0$.
Moreover, $\mathrm{PS}_{\phi,\phi'}$ is supported on the non-wandering set of $\varphi_t$ (see Section \ref{sect:classical_resonances}).  
The terminology comes from their analogy with the construction of boundary \emph{measures} associated to Laplace eigenfunctions studied by Patterson \cite{Patterson76, Patterson87} and Sullivan \cite{Sul79, Sul81}. The definition of the Patterson-Sullivan distribution $\mathrm{PS}_{\phi,\phi'}$ in the classical approach of  \cite{AZ07,AZintertwining, HHS12} makes explicit use of the Helgason boundary values  $T_{\phi},T_{\phi'} \in \D'(\partial_\infty \mathbb{H}^2)$ of $\phi$ and $\phi'$, which are distributions on the boundary $\partial_\infty \mathbb{H}^2\cong \mathbb{S}^1$ of $\HH^2$. By identifying distributions on $S\mathbf{X}_\Gamma$ with $\Gamma$-invariant distributions on the unit sphere bundle $S\HH^2$ of the hyperbolic plane, $\mathrm{PS}_{\phi,\phi'} \in \mathcal{D}'(S\mathbf{X}_\Gamma)\cong \mathcal{D}'(S\HH^2)^\Gamma$ is defined as the $\Gamma$-average of $\mathcal{R}'_{s_0,s_0'}(T_{\phi} \otimes \overline T_{\phi'})$, where $\mathcal{R}'_{s_0,s'_0}: \mathcal{D}'(\partial_\infty \mathbb{H}^2 \times \partial_\infty \mathbb{H}^2) \rightarrow \mathcal{D}'(S\HH^2)$ is the {dual of the} weighted Radon transform (see Section~\ref{sect:PS_Radon})
and $\overline{\phantom{a}}$ denotes complex conjugation. 
A more modern approach suggested in \cite{GHWb} uses the language of quantum-classical correspondence (see Proposition~\ref{prop:classicalquantum_corresp}): This correspondence assigns to $\phi,\phi'$ a unique pair $v_\phi,v_{\phi'}^\ast\in \mathcal{D}'(S\mathbf{X}_\Gamma)$ of so-called \emph{Ruelle resonant and co-resonant states}. These distributions are characterized by the property that their pushforwards along the sphere bundle projection $S\mathbf{X}_\Gamma\to \mathbf{X}_\Gamma$ are given by $\phi$ and $\phi'$, respectively, that they are quasi-invariant under the geodesic flow, and that their wavefront sets are contained in the dual stable and unstable subbundles of the geodesic flow, respectively. The latter implies that their distributional product is well-defined. This permits us to express the Patterson-Sullivan distribution by that product:
\bq
v_\phi\cdot \bar v_{\phi'}^\ast=\mathrm{PS}_{\phi,\phi'},\label{eq:vvetoile}
\eq
see Section~\ref{sect:PS_Resonant}. In the proof of the quantum-classical correspondence of \cite{GHWb} the construction of the inverse of the pushforward along the sphere bundle projection involves Helgason boundary values, so that the modern approach is technically not very different from the classical one. Furthermore, in the compact case, Anantharaman and Zelditch already expressed the Patterson-Sullivan distribution in \cite[Prop.~1.1]{AZintertwining} as a well-defined product of joint eigendistributions of the horocycle and geodesic flow, so that in fact the only missing piece to arrive at the modern approach was the interpretation of their construction as a quantum-classical correspondence. However, the passage from compact to non-compact convex-cocompact hyperbolic surfaces of the quantum-classical correspondence had only been achieved in \cite{GHWa}.

It is now a natural question how the two distribution families given by the Wigner and Patterson-Sullivan distributions are related.
Note that the Wigner distributions depend on a quantization, whose choice is not unique, while the Patterson-Sullivan distributions do not.
Furthermore, the quasi-invariance \eqref{eq:quasiinvariancePS} of the latter with respect to the geodesic flow is not shared by the Wigner distributions {and the Wigner distributions are not supported on the non-wandering set of the geodesic flow} (this can be seen numerically in \cite{WBKPS14}).
However, from the classical theory of quantum ergodicity (see e.g.\ \cite{zelditch1987,Zelditch19921})  it is well-known that the (lifted) \emph{quantum limits} obtained by considering the asymptotics of the Wigner distributions along unbounded sequences $s_j,s_j'$ $(j \in \N)$ of quantum resonances have good invariance properties and are independent of the choice of the quantization.
Therefore, the goal is to compare the Wigner and Patterson-Sullivan distributions \emph{asymptotically}. This has been successfully established in the compact case by Anantharaman and Zelditch \cite{AZ07,AZintertwining} and by Hansen, Hilgert and Schröder \cite{HHS12}, who generalized the results to compact locally symmetric spaces of higher rank.

From a scattering theory point of view, it is desirable to relate Wigner and Patterson-Sullivan distributions also on non-compact manifolds. In particular, Schottky surfaces are of interest, as they form a prototypical and  well-studied family of examples of scattering systems (see Example \ref{ex:Schottky}, \cite[Sec.~15.1]{Borthwickbook} {and Figure \ref{fig:resonances}}). 
To our knowledge, the theory of Patterson-Sullivan distributions on non-compact hyperbolic surfaces and their asymptotics is still largely unexplored, which was a main motivation for the present paper. 

In order to determine a reasonable asymptotic parameter, we need to find and parametrize an unbounded sequence of quantum resonances in the complex plane. For compact $\Xbf_\Gamma$ this is  trivial: 
The quantum resonances correspond to genuine eigenvalues of the Laplacian and the latter form an unbounded sequence of points on the real line, which means that it suffices to consider only resonances of the form $s_0=\frac{1}{2}+ir${, $s'_0=\frac{1}{2}-ir$}, $r>0$.
Then $r$, or equivalently $h\coloneq r^{-1}$, is the natural asymptotic parameter, see \cite[Sec.~3.2]{AZ07}.
For non-compact $\Xbf_\Gamma$, there are no quantum resonances on the \emph{critical line} $\frac{1}{2}+i\R$ except possibly $s_0=\frac{1}{2}$ \cite[Cor.~7.8]{Borthwickbook} and it is highly non-trivial to determine how close to the critical line one can find an infinite number of quantum resonances, see Section \ref{sec:parameter}.
However,  if one is willing to consider a wide enough vertical strip at the left of the critical line, then this is always possible:
A result of Guillopé-Zworski \cite{GuillopeZworski} (see also \cite[Thm.~12.4]{Borthwickbook}) implies that for all $C>\frac{3}{2}$ there is an unbounded sequence of quantum resonances in the strip $\{s\in \C\,|\,\frac{1}{2}-C\leq \Re s\leq\frac{1}{2} \}$. 

Another issue that is delicate in the non-compact setting and trivial in the compact case is the normalization of the quantum resonant states, which implies a normalization of the Wigner and Patterson-Sullivan distributions: In the compact case, one simply $L^2$-normalizes.
In our non-compact situation, the relevant quantum resonant states do not lie in $L^2(\Xbf_\Gamma)$ (because this would imply that $\Laplace$ has non-real $L^2$-eigenvalues). 
We focus in this paper on the asymptotic relation between the Wigner and Patterson-Sullivan distributions \emph{relative to each other} rather than relative to some preferred absolute scale. This works within a broad range of ``reasonable normalizations'' (see Remark \ref{rem:norm} and Section \ref{sec:normalization}). 
To define this normalization range and to state our main results below we use Sobolev spaces; see Section \ref{sec:Sobolev} for a brief summary of the corresponding notation.

We can now state our main result: For $j\in \N$, let $s_j,s_j' \in \C\setminus (-\frac{1}{2}-\frac{1}{2}\N_0)$ be quantum resonances of the form $s_j=q_j+ir_j$, {$s_j'=q_j' -ir_j$}, where $r_j\to +\infty$ as $j\to \infty$ and $\frac{1}{2}-C\leq q_j,q'_j\leq\frac{1}{2}$ for some $C>0$ (larger than the essential spectral gap, e.g.\ $C=2$, see Section \ref{sec:parameter}). 
Let $\phi_j\in \mathrm{Res}^1_{\Laplace}(s_j)$ and $\phi_j'\in \mathrm{Res}^1_{\Laplace}(s_j')$ be quantum resonant states.
\begin{thm} \label{thm:asymptotic_WPS}
    Let $k\in \N_0$ and suppose that the sequence $\{(\phi_j,\phi_j')\}_{j\in\N}$ is moderately $H^{-k}$-normalized (Definition \ref{def:moderatenormalization}).
    Then the Wigner and Patterson-Sullivan distributions are of Sobolev regularity $H^{-k}$, i.e., $W_{\phi_j,\phi'_j}, \mathrm{PS}_{\phi_j,\phi_j'}\in H^{-k}_{\mathrm{loc}}(S\mathbf{X}_\Gamma)$.    Furthermore, as $j\to \infty$, the distributions are asymptotically related as follows: For every compact set  $\mathscr{C}\subset S\mathbf{X}_\Gamma$ and every $N\in \N$ there are compact sets $\mathscr{K}_\mathscr{C}, \mathscr{K}_{\mathscr{C},N}\subset S\mathbf{X}_\Gamma$ such that
    \begin{align*}
        \big\Vert W_{\phi_j,\phi'_j}- c r_j^{-1/2} {\mathrm{PS}}_{\phi_j,\phi_j'}\big\Vert_{H^{-k}_{\mathscr{C}}(S\mathbf{X}_\Gamma)}&=\O\big( r_j^{-3/2}\big\Vert{\mathrm{PS}}_{\phi_j,\phi_j'}\big\Vert_{H^{-k}_{\mathscr{K}_\mathscr{C}}(S\mathbf{X}_\Gamma)}\big) +\O\big(r_j^{-N}\big),\\
        \Vert W_{\phi_j,\phi'_j}- c r_j^{-1/2} {\mathrm{PS}}_{\phi_j,\phi_j'}\big\Vert_{H^{-k}_{\mathscr{C}}(S\mathbf{X}_\Gamma)}&=\O\big( r_j^{-1}\big\Vert W_{\phi_j,\phi'_j}\Vert_{H^{-k}_{\mathscr{K}_{\mathscr{C},N}}(S\mathbf{X}_\Gamma)}\big) + \O\big(r_j^{-N}\big).
    \end{align*}
    Here the constant $c\in \C$ is explicitly given by $c\coloneq \frac{1}{\sqrt{\pi} }e^{-i\frac{\pi}{4}}$.\footnote{The value of this constant depends on {conventions  (cf.~\cite[Sec.~2.4.\ and Prop.~5.3]{AZintertwining}), which explains why our $c$ equals $2\pi$ times the constant  in} \cite[Thm.\ 5]{AZintertwining}.}
    
\end{thm}
This generalizes the result \cite[Thm.\ 5]{AZintertwining} of Anantharaman-Zelditch to the convex-cocompact case (see Section \ref{sec:cmptdiff} for a comparison),  proves the conjecture mentioned in \cite[end of p.~672]{SWB},  and gives a partial answer to \cite[Problem~6.22]{Hilgert23}.
In particular, Theorem \ref{thm:asymptotic_WPS} applies to  Schottky surfaces, which are examples of convex-cocompact hyperbolic surfaces \cite{Weich2015_ResonanceChains,BarkhofenFaureWeich2014_ResonanceChains}.

\begin{rem}\label{rem:norm}\begin{enumerate}[leftmargin=*]
    \item Theorem \ref{thm:asymptotic_WPS} is deduced in Section \ref{sect:asymptotic_equiv} from the more general (but also more technical)  Theorem \ref{thm:precise}, which is fully normalization-invariant. 
    The $\O(r_j^{-N})$-remainders in Theorem \ref{thm:asymptotic_WPS} are irrelevant in practice -- they  only account for the theoretical possibility that for some particular compact sets $\mathscr C$ the first remainder terms in the two lines might decay as $\O(r_j^{-\infty})$.    Then the presence of the $\O(r_j^{-N})$-remainders makes the statement of Theorem \ref{thm:asymptotic_WPS} trivial for such $\mathscr C$. On the other extreme end, our methods of proof show that, given any compact set $\mathscr K\subset S\mathbf{X}_\Gamma$  there is an $M_{\mathscr K}\in \N$ such that both  $W_{\phi_j,\phi'_j}$ and ${\mathrm{PS}}_{\phi_j,\phi_j'}$ are $\O_{H^{-k}_{\mathscr{K}}(S\mathbf{X}_\Gamma)}(r_j^{M_{\mathscr K}})$. 
    In summary, Theorem \ref{thm:asymptotic_WPS} is meaningful within the range of all normalizations making the Wigner and the Patterson-Sullivan distributions grow at most polynomially or decay at most inverse-polynomially.  
    \item Since the set $\mathscr{K}_{\mathscr{C}}$ in the first estimate in Theorem \ref{thm:asymptotic_WPS} does not depend on $N \in \N$, we can renormalize $\phi_j\mapsto (1+r_j^{-1/2}\Vert{\mathrm{PS}}_{\phi_j,\phi_j'}\Vert_{H^{-k}_{\mathscr{K}_{\mathscr{C}}}(S\mathbf{X}_\Gamma)})^{-1}\phi_j$ and replace $N$ by $N+M_{\mathscr{K}_{\mathscr{C}}}+1$ to get
    \begin{align*}
        \big\Vert W_{\phi_j,\phi'_j}- c r_j^{-1/2} {\mathrm{PS}}_{\phi_j,\phi_j'}\big\Vert_{H^{-k}_{\mathscr{C}}(S\mathbf{X}_\Gamma)}&=\O\big( r_j^{-1}\big),\\
        \Vert W_{\phi_j,\phi'_j}- c r_j^{-1/2} {\mathrm{PS}}_{\phi_j,\phi_j'}\big\Vert_{H^{-k}_{\mathscr{C}}(S\mathbf{X}_\Gamma)}&=\O\Biggr( \frac{r_j^{-1}\Vert W_{\phi_j,\phi'_j}\Vert_{H^{-k}_{\mathscr{K}_{\mathscr{C},N}}(S\mathbf{X}_\Gamma)}}{1+r_j^{-1/2}\Vert{\mathrm{PS}}_{\phi_j,\phi_j'}\Vert_{H^{-k}_{\mathscr{K}_{\mathscr{C}}}(S\mathbf{X}_\Gamma)}}\Biggl)+ \O\big(r_j^{-N}\big).\hspace*{-1em}
    \end{align*}
    This reformulation of Theorem \ref{thm:asymptotic_WPS} resembles Anantharaman-Zelditch's result \cite[Thm.\ 5]{AZintertwining} more closely. In fact, it is implicitly shown in \cite[Sec.~7]{AZintertwining} that in their compact case with $L^2$-normalized quantum resonant states  one has (with our notation) $\Vert W_{\phi_j,\phi'_j}\Vert_{H^{-k}(S\mathbf{X}_\Gamma)}=\O(1)$ as $j\to \infty$ for $k$ large enough. 
    Thus, in our case, the above normalization for a chosen compact set $\mathscr C\subset S\mathbf{X}_\Gamma$ and Sobolev order $k$ can be regarded as an abstract replacement for the unavailable $L^2(S\mathbf{X}_\Gamma)$-normalization under the assumption that an initial moderate $H^{-k}$-normalization as in Definition \ref{def:moderatenormalization} exists (note that the latter does not depend on $\mathscr{C}$).
    See Section \ref{sec:normalization} for more details and comments on the normalization issue.  
\end{enumerate}
\end{rem}

\begin{rem}\label{rem:newrem} Theorem \ref{thm:asymptotic_WPS} implies the following pointwise statement: 
Given a test function $u\in \CT(S\mathbf{X}_\Gamma)$ and a compact set $\mathscr C\subset S\mathbf{X}_\Gamma$ with $\supp u\subset \mathscr C$, Theorem \ref{thm:asymptotic_WPS} gives
\begin{align}
|W_{\phi_j,\phi'_j}(u)- c\, r_j^{-1/2} {\mathrm{PS}}_{\phi_j,\phi_j'}(u)|&=\O\big( r_j^{-3/2}\big\Vert{\mathrm{PS}}_{\phi_j,\phi_j'}\big\Vert_{H^{-k}_{\mathscr{K}_{\mathscr{C}}}(S\mathbf{X}_\Gamma)}\big) +\O\big(r_j^{-N}\big),\label{eq:firstlineRem}\\
|W_{\phi_j,\phi'_j}(u)- c\, r_j^{-1/2} {\mathrm{PS}}_{\phi_j,\phi_j'}(u)|&=\O\big( r_j^{-1}\big\Vert W_{\phi_j,\phi'_j}\Vert_{H^{-k}_{\mathscr{K}_{\mathscr{C},N}}(S\mathbf{X}_\Gamma)}\big) + \O\big(r_j^{-N}\big).\label{eq:secondlineRem}
  \end{align}
If $\supp u$ is disjoint from the non-wandering set of the geodesic flow, then ${\mathrm{PS}}_{\phi_j,\phi_j'}(u)=0$ for all $j$ because the Patterson-Sullivan distribution is supported on the non-wandering set, and the above estimates become
\begin{align*}
|W_{\phi_j,\phi'_j}(u)| &=\O\big( r_j^{-3/2}\big\Vert{\mathrm{PS}}_{\phi_j,\phi_j'}\big\Vert_{H^{-k}_{\mathscr{K}_{\mathscr{C}}}(S\mathbf{X}_\Gamma)}\big) +\O\big(r_j^{-N}\big),\\
|W_{\phi_j,\phi'_j}(u)| &=\O\big( r_j^{-1}\big\Vert W_{\phi_j,\phi'_j}\Vert_{H^{-k}_{\mathscr{K}_{\mathscr{C},N}}(S\mathbf{X}_\Gamma)}\big) + \O\big(r_j^{-N}\big).
  \end{align*}
These two statements can be regarded as evidence that, in a relative sense, the Wigner distributions are ``asymptotically supported on the non-wandering set'', provided that a  moderate $H^{-k}$-normalization is given. 
\end{rem}

Theorem \ref{thm:asymptotic_WPS} can be interpreted as an asymptotic relation between a quantum object on the left-hand side and a classical object on the right-hand side. Their asymptotic equivalence is then in accordance with the \emph{correspondence principle} from quantum physics, which says that in the high energy limit the quantization of a classical system should exhibit emergent features resembling those of the classical system. 

The asymptotic relation from Theorem \ref{thm:asymptotic_WPS} is useful in the context of quantum ergodicity, where the problem consists of determining which geodesic flow invariant probability measures arise as weak${}^*$-limit points of Wigner distributions. Note that in the setting of hyperbolic manifolds with funnels, the quantum (unique) ergodicity problem has already been solved, we refer to Dyatlov's survey paper \cite[Sec.~3.2.2.]{Dyatlov22} for further details on this topic.

The asymptotic equivalence between Wigner and Patterson-Sullivan distributions can certainly be generalized beyond the case of convex-cocompact hyperbolic surfaces. The latter provides a convenient setting in which all necessary technical ingredients of the proof such as the quantum-classical correspondence from \cite{GHWb} and the spectral estimates from \cite{GuillopeZworski} were readily available, which makes the proof of Theorem \ref{thm:asymptotic_WPS} relatively short and direct. Therefore we restrict in this paper to the convex-cocompact hyperbolic surfaces. We plan to carry out generalizations in separate future works.  Note that Hadfield \cite{Ha20} established a classical-quantum correspondence for open hyperbolic manifolds.

Furthermore, when the resonances are simple, the Patterson–Sullivan distributions coincide with the invariant Ruelle distributions of \cite{GHWb}. Extending this coincidence to resonances of higher (finite) rank requires a pairing formula; to our knowledge, such a formula has not yet been established in the convex-cocompact setting. Schütte and Weich \cite{SchutteWeich2023} showed that invariant Ruelle distributions distributions can be numerically visualized for convex-cocompact hyperbolic surfaces, in particular for two specific classes of rank-two Schottky surfaces, by approximating the distributions via weighted zeta functions.

Finally, let us mention that instead of considering asymptotics of quantum resonances, an alternative approach in the non-compact case is to study the continuous spectrum of the Laplacian and the associated Eisenstein functions, see \cite{GNeisenstein,DGeisenstein,Ingremeau_eisenstein}. 

\subsection{Differences to the compact case }\label{sec:cmptdiff}

 Here we give more details and introduce some required language related to the difference between our non-compact convex-cocompact case and the compact case studied in \cite{AZ07,AZintertwining,HHS12}.

\subsubsection{Unbounded sequences of quantum resonances}\label{sec:parameter} 
 Every non-compact convex-co\-com\-pact hyperbolic surface has an essential spectral gap \cite{BourgainDyatlov}, which means that there is an $\eps>0$ such that there are only finitely many quantum resonances $s_0$ with $\Re s_0> \frac{1}{2}-\eps$. The size of the essential spectral gap, given by the supremum $\eps_\mathrm{max}$ of all such $\eps$, is the subject of the famous \emph{Jakobson-Naud conjecture} \cite{JakobsonNaud}, which postulates that $\eps_\mathrm{max}=\frac{1-\delta}{2}$, where $\delta\in [0,1]$ is the Hausdorff dimension of the limit set $\Lambda_\Gamma\subset \mathbb{S}^1$ of $\Gamma$ (the set of all accumulation points of $\Gamma$-orbits in the compactified Poincaré disk $\mathbb D^2\cup \mathbb S^1$). Thus, if the conjecture is true, we can take any $C>\frac{1-\delta}{2}$ in Theorem \ref{thm:asymptotic_WPS}, see Figure \ref{fig:condition}. 

  \begin{figure}[h!]
        \centering
        \begin{tikzpicture}
        \fill [lightgray!20] (-1.2,3) rectangle (3.5,-3);
             \fill [blue!10] (1,3) rectangle (1.5,-3);
\draw[->] (-2.5,0)--(4.5,0) node[right]{$\Re(s)$};
\draw[->] (0,-3)--(0,3) node[left]{$\Im(s)$};
\draw[-,dashed,thick] (3.5,-3)--(3.5,3);
\draw[-,dashed,thick] (1.5,-3)--(1.5,3);
\draw[-,thick] (1.5,-0.125)--(1.5,0.125);
\draw[-,thick] (3.5,-0.125)--(3.5,0.125);
\draw[-,thick] (-1,-0.125)--(-1,0.125);
\draw (3.75,-0.35) node {$\frac{1}{2}$};
\draw (-1,-0.35) node {$-1$};
\draw (1.7,-0.35) node {$\frac{\delta}{2}$};
\draw (0.3,-0.3) node {$0$};
%\draw (1.3,-0.3) node {$1$};

\draw[blue,fill=blue] (-0.5,2.2) circle (1.5pt);
\draw[blue,fill=blue] (-1.2,2.45) circle (1.5pt);
\draw[blue,fill=blue] (-0.3,2) circle (1.5pt);
\draw[blue,fill=blue] (-0.9,1.1) circle (1.5pt);
\draw[blue,fill=blue] (-0.4,0.65) circle (1.5pt);
\draw[blue,fill=blue] (-0.6,0.15) circle (1.5pt);
\draw[blue,fill=blue] (-0.5,-2.2) circle (1.5pt);
\draw[blue,fill=blue] (-1.2,-2.45) circle (1.5pt);
\draw[blue,fill=blue] (-0.3,-2) circle (1.5pt);
\draw[blue,fill=blue] (-0.9,-1.1) circle (1.5pt);
\draw[blue,fill=blue] (-0.4,-0.65) circle (1.5pt);
\draw[blue,fill=blue] (-0.6,-0.15) circle (1.5pt);

\draw[blue,fill=blue] (0.5,2.8) circle (1.5pt);
\draw[blue,fill=blue] (1.2,2.85) circle (1.5pt);
\draw[blue,fill=blue] (1.3,2.5) circle (1.5pt);
\draw[blue,fill=blue] (0.9,2.1) circle (1.5pt);
\draw[blue,fill=blue] (0,2.15) circle (1.5pt);
\draw[blue,fill=blue] (0,-2.15) circle (1.5pt);
\draw[blue,fill=blue] (0.6,1.8) circle (1.5pt);
\draw[blue,fill=blue] (0.9,1.5) circle (1.5pt);
\draw[blue,fill=blue] (0.7,1.1) circle (1.5pt);
\draw[blue,fill=blue] (0.3,0.9) circle (1.5pt);
\draw[blue,fill=blue] (0.8,0.5) circle (1.5pt);
\draw[blue,fill=blue] (0,0.55) circle (1.5pt);
\draw[blue,fill=blue] (0,-0.55) circle (1.5pt);
\draw[blue,fill=blue] (0.5,0) circle (1.5pt);
\draw[blue,fill=blue] (1.5,1.2) circle (1.5pt);
\draw[blue,fill=blue] (1.3,0.85) circle (1.5pt);
\draw[blue,fill=blue] (1.7,0.75) circle (1.5pt);
\draw[blue,fill=blue] (0.5,0) circle (1.5pt);
\draw[blue,fill=blue] (1,0) circle (1.5pt);
\draw[blue,fill=blue] (0.5,-2.8) circle (1.5pt);
\draw[blue,fill=blue] (1.2,-2.85) circle (1.5pt);
\draw[blue,fill=blue] (1.3,-2.5) circle (1.5pt);
\draw[blue,fill=blue] (0.9,-2.1) circle (1.5pt);
\draw[blue,fill=blue] (0.6,-1.8) circle (1.5pt);
\draw[blue,fill=blue] (0.9,-1.5) circle (1.5pt);
\draw[blue,fill=blue] (0.7,-1.1) circle (1.5pt);
\draw[blue,fill=blue] (0.3,-0.9) circle (1.5pt);
\draw[blue,fill=blue] (0.8,-0.5) circle (1.5pt);
\draw[blue,fill=blue] (1.5,-1.2) circle (1.5pt);
\draw[blue,fill=blue] (1.3,-0.85) circle (1.5pt);
\draw[blue,fill=blue] (1.7,-0.75) circle (1.5pt);

\draw[blue,fill=blue] (-1.5,2.4) circle (1.5pt);
\draw[blue,fill=blue] (-2.2,2.74) circle (1.5pt);
\draw[blue,fill=blue] (-2.3,2.58) circle (1.5pt);
\draw[blue,fill=blue] (-1.9,2.12) circle (1.5pt);
\draw[blue,fill=blue] (-2.1,1.87) circle (1.5pt);
\draw[blue,fill=blue] (-2.3,1.53) circle (1.5pt);
\draw[blue,fill=blue] (-1.7,1.13) circle (1.5pt);
\draw[blue,fill=blue] (-1.3,0.95) circle (1.5pt);
\draw[blue,fill=blue] (-1.8,0.4) circle (1.5pt);
\draw[blue,fill=blue] (-1.5,0.2) circle (1.5pt);
\draw[blue,fill=blue] (-1.5,-2.4) circle (1.5pt);
\draw[blue,fill=blue] (-2.2,-2.74) circle (1.5pt);
\draw[blue,fill=blue] (-2.3,-2.58) circle (1.5pt);
\draw[blue,fill=blue] (-1.9,-2.12) circle (1.5pt);
\draw[blue,fill=blue] (-2.1,-1.87) circle (1.5pt);
\draw[blue,fill=blue] (-2.3,-1.53) circle (1.5pt);
\draw[blue,fill=blue] (-1.7,-1.13) circle (1.5pt);
\draw[blue,fill=blue] (-1.3,-0.95) circle (1.5pt);
\draw[blue,fill=blue] (-1.8,-0.4) circle (1.5pt);
\draw[blue,fill=blue] (-1.5,-0.2) circle (1.5pt);

\draw[blue,fill=blue] (2,0.3) circle (1.5pt);
\draw[blue,fill=blue] (2.8,0) circle (1.5pt);
\draw[blue,fill=blue] (2.5,1) circle (1.5pt);
\draw[blue,fill=blue] (2.1,1.3) circle (1.5pt);
\draw[blue,fill=blue] (2,-0.3) circle (1.5pt);
\draw[blue,fill=blue] (2.5,-1) circle (1.5pt);
\draw[blue,fill=blue] (2.1,-1.3) circle (1.5pt);

\node[anchor=north east, inner sep=0] at (6,3) {%
    \includegraphics[width=2.3cm]{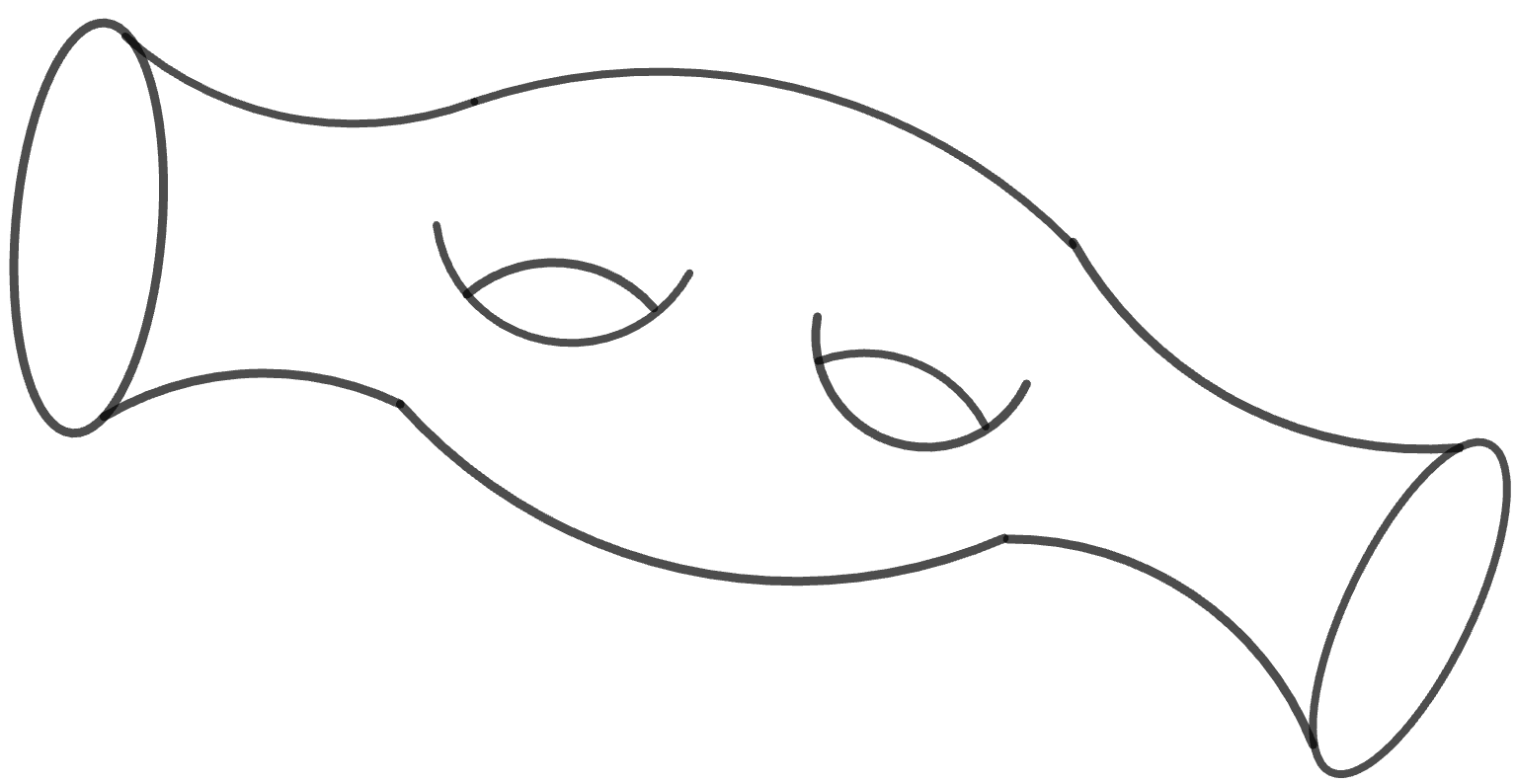} 
};
\end{tikzpicture}\hspace{2em}
\begin{tikzpicture}
\draw[->] (-2.5,0)--(1,0) node[right]{$\Re(s)$};
\draw[->] (0,-3)--(0,3) node[left]{$\Im(s)$};
\draw[-,thick] (-1,-0.125)--(-1,0.125);
\draw[-,thick] (-2,-0.125)--(-2,0.125);
\draw (0.3,-0.3) node {$0$};
\draw (-0.7,-0.3) node {$-1$};
\draw (-1.7,-0.3) node {$-2$};

\draw[blue,fill=blue] (0,1.5) circle (1.5pt);
\draw[blue,fill=blue] (0,2) circle (1.5pt);
\draw[blue,fill=blue] (0,2.5) circle (1.5pt);
\draw[blue,fill=blue] (0,3) circle (1.5pt);
\draw[blue,fill=blue] (0,1) circle (1.5pt);
\draw[blue,fill=blue] (0,0.5) circle (1.5pt);
\draw[blue,fill=blue] (0,0) circle (1.5pt);
\draw[blue,fill=blue] (0,-1.5) circle (1.5pt);
\draw[blue,fill=blue] (0,-2) circle (1.5pt);
\draw[blue,fill=blue] (0,-2.5) circle (1.5pt);
\draw[blue,fill=blue] (0,-3) circle (1.5pt);
\draw[blue,fill=blue] (0,-1) circle (1.5pt);
\draw[blue,fill=blue] (0,-0.5) circle (1.5pt);

\draw[blue,fill=blue] (-1,1.5) circle (1.5pt);
\draw[blue,fill=blue] (-1,2) circle (1.5pt);
\draw[blue,fill=blue] (-1,2.5) circle (1.5pt);
\draw[blue,fill=blue] (-1,3) circle (1.5pt);
\draw[blue,fill=blue] (-1,1) circle (1.5pt);
\draw[blue,fill=blue] (-1,0.5) circle (1.5pt);
\draw[blue,fill=blue] (-1,0) circle (1.5pt);
\draw[blue,fill=blue] (-1,-1.5) circle (1.5pt);
\draw[blue,fill=blue] (-1,-2) circle (1.5pt);
\draw[blue,fill=blue] (-1,-2.5) circle (1.5pt);
\draw[blue,fill=blue] (-1,-3) circle (1.5pt);
\draw[blue,fill=blue] (-1,-1) circle (1.5pt);
\draw[blue,fill=blue] (-1,-0.5) circle (1.5pt);

\draw[blue,fill=blue] (-2,1.5) circle (1.5pt);
\draw[blue,fill=blue] (-2,2) circle (1.5pt);
\draw[blue,fill=blue] (-2,2.5) circle (1.5pt);
\draw[blue,fill=blue] (-2,3) circle (1.5pt);
\draw[blue,fill=blue] (-2,1) circle (1.5pt);
\draw[blue,fill=blue] (-2,0.5) circle (1.5pt);
\draw[blue,fill=blue] (-2,0) circle (1.5pt);
\draw[blue,fill=blue] (-2,-1.5) circle (1.5pt);
\draw[blue,fill=blue] (-2,-2) circle (1.5pt);
\draw[blue,fill=blue] (-2,-2.5) circle (1.5pt);
\draw[blue,fill=blue] (-2,-3) circle (1.5pt);
\draw[blue,fill=blue] (-2,-1) circle (1.5pt);
\draw[blue,fill=blue] (-2,-0.5) circle (1.5pt);

% Include a real image in the top-right corner
\node[anchor=north east, inner sep=0] at (2.3,3) {%
    \includegraphics[width=2cm]{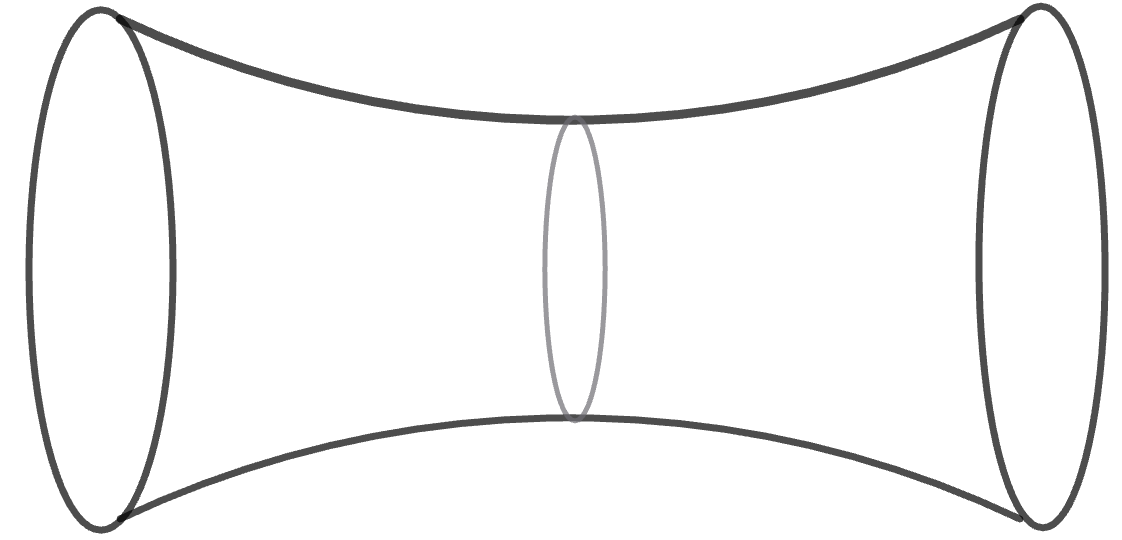} 
};

\end{tikzpicture}
        \caption{ \label{fig:condition} Left: Fictional plot of quantum resonances of a non-elementary convex-cocompact hyperbolic surface that might look like the one illustrated in the picture, for example. If the Jakobson-Naud conjecture holds for this surface, then any strip  of positive width (indicated in blue) at the left of the line $\Re s=\frac{\delta}2$  contains an unbounded sequence of resonances. Independently of the Jakobson-Naud conjecture, the large gray strip of width $>\frac{3}{2}$ at the left of the critical line $\Re s=\frac{1}{2}$ always contains an unbounded sequence of resonances. 
        Right: The elementary case of a hyperbolic cylinder. Here the quantum resonances lie on a lattice.}
    \end{figure}
  Note that if $\delta=0$, then $\Gamma$ is elementary, i.e., $\Gamma$ is trivial or $\Gamma\simeq\Z$ \cite[p.~33]{Borthwickbook}. In the trivial case there are no Patterson-Sullivan distributions (because the non-wandering set of $\varphi_t$, in which they are supported, is empty) and in the case $\Gamma\simeq\Z$ the surface $\Xbf_\Gamma$ is a hyperbolic cylinder whose  resonance spectrum is explicitly given by $\frac{2\pi i}{l}\Z-\N_0\subset \C$ for some $l>0$  \cite[Thm.~2]{ChristiansenZworski}. In general, one wants to choose the constant $C>0$ in Theorem \ref{thm:asymptotic_WPS} not much larger than the essential spectral gap because the quantum resonances with largest real part are the most relevant ones (for example because they dominate wave asymptotics \cite[p.~724]{Naud2014}). 

\begin{figure}[h!]
    \centering
    %\hspace{-4cm}
    \begin{overpic}[width=0.7\textwidth]{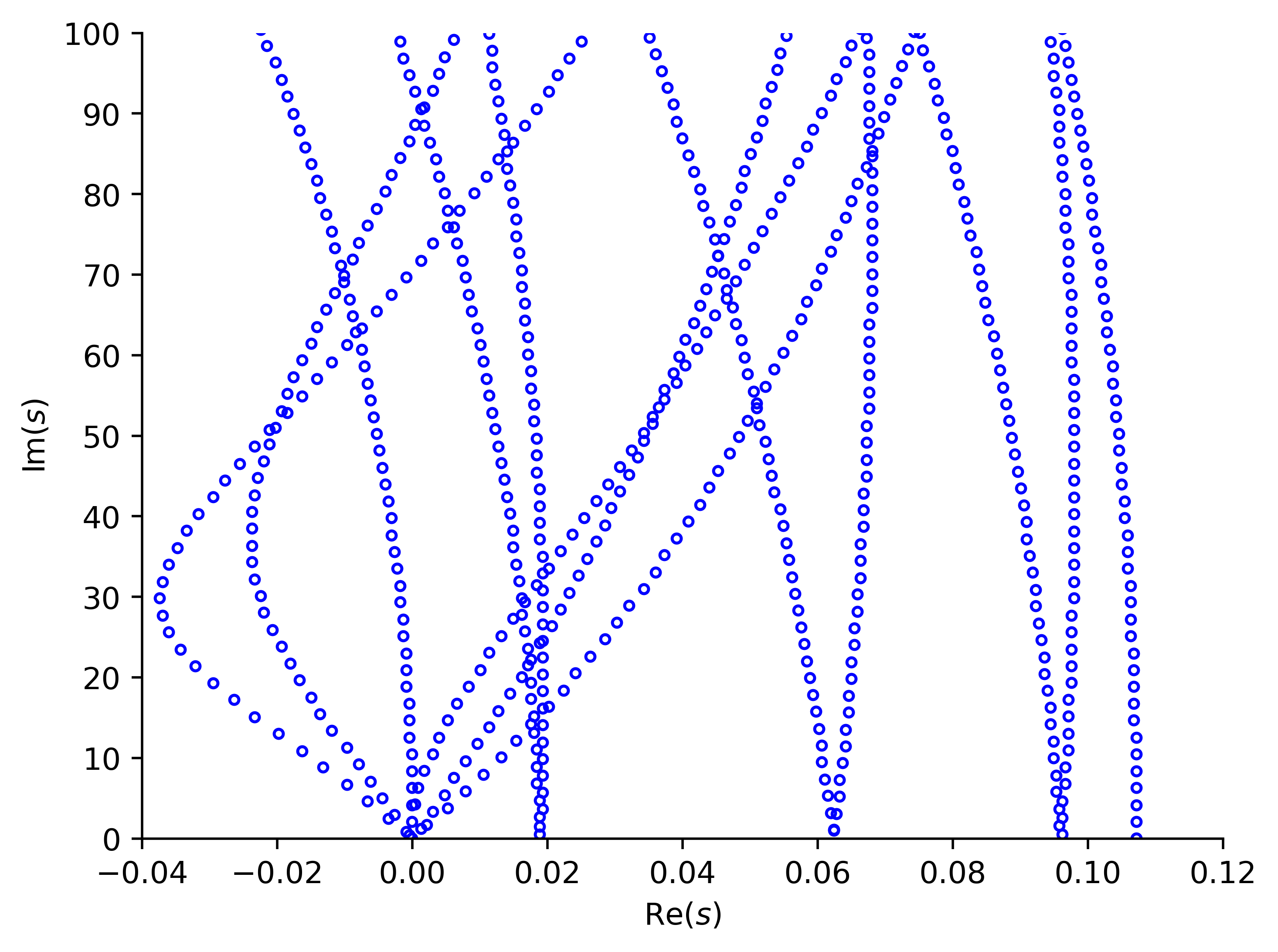}
        % place the small image at top-right corner (adjust coordinates)
        \put(90,55){%\put(75,54) corner in box %\put(98,30) middle
            \includegraphics[width=0.15\textwidth]{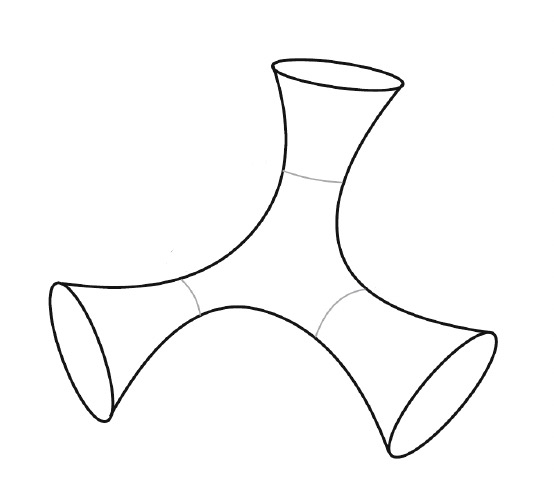}%
        } %width=0.13 %width=0.2
    \end{overpic}
    \caption{ \label{fig:resonances}Plotted are numerically calculated quantum resonances $s_j\in \C$ for a $3$-funneled Schottky surface, following the approach of \cite{Weich2015_ResonanceChains, Borthwickbook}.}
\end{figure}
  
While the Jakobson-Naud conjecture remains open, much weaker sufficient conditions than $C>\frac{3}{2}$ are known in many cases: For example, the results of Jakobson-Naud  \cite[Thm.~1.2]{JakobsonNaud} show that $C>\frac{1-\delta(1-2\delta)}{2}$ works whenever $\delta< \frac{1}{2}$, and that $C>\frac{3-2\delta}{4}$ works in case that $\delta>\frac{1}{2}$ and $\Gamma$ is a subgroup of an arithmetic group. For Schottky surfaces, precise results on a chain-like structure of the quantum resonance spectrum accompanied by numerical calculations \cite{Weich2015_ResonanceChains, BarkhofenFaureWeich2014_ResonanceChains} suggest an abundance of quantum resonances with positive real part in agreement with the Jakobson-Naud conjecture, c.f.\ Figure  \ref{fig:resonances}. 

%\begin{tikzpicture}[
%    tqft/cobordism/.style={draw, thick},
%    tqft/every boundary component/.style={draw, dashed}
%]
%    % Draw the pair of pants
%    \pic [tqft/pair of pants];
%\end{tikzpicture}

%\begin{tikzpicture}
%  \draw (0,0) ellipse (.5 and .25);
%  \draw (-1,-2) ellipse (.5 and .25);
%  \draw (1,-2) ellipse (.5 and .25);
%  \draw (-.5,0) to[out=-90,in=90] (-1.5,-2);
%  \draw (.5,0) to[out=-90,in=90] (1.5,-2);
%  \draw (-.5,-2) to[out=90,in=90] (.5,-2);
%\end{tikzpicture}

\subsubsection{Normalization}\label{sec:normalization}

In the case where $\Xbf_\Gamma$ is compact, one considers  sequences $\phi_j,\phi_j'$ of $L^2$-normalized quantum resonant states, which immediately gives that $W_{\phi_j,\phi'_j}(u)=\O(1)$ as $j\to \infty$ for any $u\in \CT(S\Xbf_\Gamma)$. This fixes an ``absolute scale'' with respect to which the growth of the Patterson-Sullivan distributions can be measured. Without such an absolute scale, one can wonder if it even makes sense to compare the Wigner and Patterson-Sullivan distributions asymptotically, as both of them have the same transformation behavior under rescaling: Substituting $(\phi_j,\phi'_j)\mapsto (z_j\phi_j,z_j'\phi'_j)$ with $z_j,z_j'\in \C$, we have $W_{z_j\phi_j,z_j'\phi'_j}=z_j\bar z_j' W_{\phi_j,\phi'_j}$ and $\mathrm{PS}_{z_j\phi_j,z_j'\phi'_j}=z_j\bar z_j' \mathrm{PS}_{\phi_j,\phi'_j}$, so in principle even the slightest difference in the asymptotic behavior of the two distributions can be arbitrarily exaggerated by renormalizing the quantum resonant states. Our methods of proof only allow to meaningfully compare the Wigner and Patterson-Sullivan distributions (even relative to each other) under a ``sanity condition'' which we describe in the following:

Fix $C>0$, and for $j\in \N$, let $s_j,s_j'\in \C\setminus (-\frac{1}{2}-\frac{1}{2}\N_0)$ be quantum resonances of the form $s_j=q_j+ir_j$, $s_j'=q_j' -ir_j$, where $r_j\to +\infty$ as $j\to \infty$ and $\frac{1}{2}-C\leq q_j,q'_j\leq\frac{1}{2}$. Furthermore, let $\phi_j\in \mathrm{Res}^1_{\Laplace}(s_j)$ and $\phi_j'\in \mathrm{Res}^1_{\Laplace}(s_j')$ be quantum resonant states. Then, by results of Dyatlov-Guillarmou \cite[Cor.~7.6]{DFG} there is a Sobolev order $-k_C<0$, depending on the constant $C$, such that the Helgason boundary values  $T_{\phi_j},T_{\phi'_j}\in \mathcal{D}'(\partial_\infty \mathbb{H}^2)$  lie in the Sobolev space $H^{-k_C}(\partial_\infty \mathbb{H}^2)$ for all $j\in \N$ (and thus in $H^{-k}(\partial_\infty \mathbb{H}^2)$ for all $k\geq k_C$). 
\begin{definition}\label{def:moderatenormalization}
Let $k\in \N_0$, $k\geq k_C$.  We call the sequence $\{(\phi_j,\phi_j')\}_{j\in \N}$ \emph{moderately $H^{-k}$-normalized} if there is a constant $N>0$ such that
 \bq
 \Vert T_{\phi_j}\otimes \overline{T}_{\phi'_j}\Vert_{H^{-k}(\partial_\infty \mathbb{H}^2\times \partial_\infty \mathbb{H}^2)} =\O(r_j^N)\qquad\text{as }j\to \infty.\label{eq:TTestim}
 \eq
\end{definition}
In the compact case, pairs of $L^2$-normalized quantum resonant states are moderately $H^{-k}$-normalized for some large enough $k\in \N_0$  (as would be pairs with polynomially growing $L^2$-norm). This is proved in \cite[Thm.~3.13]{HHS12} using representation theoretic methods and in \cite[Eq.~(3.14)]{AZ07} using a regularity result of Otal \cite{Otal} (see also \cite[Sec.~7.1]{AZintertwining} for more details).  

As already mentioned in Remark \ref{rem:norm}, in the non-compact case one might wonder if there is some natural choice of normalization of the quantum resonant states replacing the global $L^2$-normalization that intrinsically allows to control appropriate seminorms (e.g. Sobolev norms) of the boundary distributions $T_{\phi_j}\otimes \overline{T}_{\phi'_j}$ and the Wigner and Patterson-Sullivan distributions. We intend to address this separate question in future works.

\subsection{Structure of the paper} 

In Section~\ref{sect:preliminaries} we begin with an introduction to our geometric setting for convex-cocompact hyperbolic surfaces. %This entails first instigating two common geometric models for the hyperbolic plane: the Poincaré half-plane and disk models. Subsequently, we collate the requisite geometric notions and facts for our setting, which are necessary for the construction of the Wigner and Patterson-Sullivan distributions.

Section~\ref{sect:classical_quantum} provides the definitions of the Pollicott-Ruelle resonances and the quantum resonances, as well as their classical-quantum correspondence, following the construction of \cite{DG,GHWa,Guillope, MazzeoMelrose}, in our framework.

The final Section~\ref{sect:relation_W_PS} is devoted to the asymptotic relation between the Wigner and Patterson-Sullivan distributions. We first give, in Subsection~\ref{sect:Wigner}, a detailed construction of the Wigner distributions via %the geometric pseudo-differential calculus and describe it then by introducing 
the (weighted) Radon transform and the intertwining operator in our specific context. %In particular, as in \cite{HHS12} (and as also observed in \cite[Appendix~8]{AZintertwining}), we need to introduce a special smooth function in order to carry out the localisation step away from the diagonal of the boundaries. 
Secondly, in Subsection~\ref{sect:PS}, two descriptions of Patterson-Sullivan distributions are provided. The first description is in terms of resonant and co-resonant states, following the approach outlined in \cite{GHWb}, while the second description is in terms of the (weighted) Radon transform. As demonstrated in \cite{GHWb}, both descriptions are equivalent. 
In the last Subsection~\ref{sect:asymptotic_equiv}, we prove our main result (Theorem~\ref{thm:asymptotic_WPS}) by combining the asymptotic results on both types of distributions.

\subsection{Acknowledgments}
The authors are grateful to Tobias Weich for suggesting this project and valuable feedback, to Yannick Guedes-Bonthonneau for fruitful discussions, to Joachim Hilgert for valuable advice and for explaining many details from the work \cite{HHS12} to us, and to Job Kuit for helpful remarks on some representation theoretic  aspects regarding Helgason boundary values. 
Finally, we warmly thank an anonymous referee whose suggestions led to corrections and  substantial improvements in this paper. This work is funded by the Deutsche Forschungsgemeinschaft (DFG, German Research Foundation) – Project-ID 491392403 – TRR 358.
B.D.\ has received further funding from the DFG through the Priority Program (SPP) 2026 “Geometry at Infinity”. This work was completed while B.D.\ was a member of the DFG's Heisenberg Programme (Project no.~DE 3864/3-1). 

%-------------------------------%%--------------------------------------%
\section{Preliminaries and background} 
\label{sect:preliminaries}
%---------------------------------------------------------------------%
In this section, we recall a number of geometric preliminaries. Let $G\coloneq\PSL(2,\R)=\{Q\in \GL(2,\R)\,|\, \det Q =1\}/\{\pm \mathrm{Id}\}$ and $$K\coloneq\SO(2)=\bigg\{
    \left[k_\vartheta:=\begin{pmatrix}
        \cos\vartheta & \sin\vartheta \\ -\sin\vartheta & \cos\vartheta
    \end{pmatrix}\right]\,\Big|\,\vartheta\in \R\bigg\}\subset G.$$
%---------------------------------------------------------------------%
\subsection{Poincaré half-plane and disk models}\label{sec:Poincare}
%---------------------------------------------------------------------%
Let us briefly introduce two common geometric models for the hyperbolic plane $\mathbb H^2$ and the $G$-action on it. The \emph{Poincaré half-plane} is  $
\mathcal{H}^2\coloneq\{z\in \C \,|\,\Im z>0\}$ 
equipped with the Riemannian metric
$
   (\Im z)^{-2}\,g_{\R^2},
$
where $g_{\R^2}$ is the Euclidean metric on $\R^2=\C$. The \emph{Poincaré disk} is the open unit disk 
$    \mathbb D^2\coloneq\{z\in \C : |z|<1\}
$
equipped with the Riemannian metric
$
     4(1-|z|^2)^{-2}\,g_{\R^2}.
$
The group $G$ acts isometrically and transitively on $\mathcal{H}^2$ by the Möbius transformations
\[
 \left[\begin{pmatrix}
       a & b \\ c & d
   \end{pmatrix}\right]
    \cdot z= \frac{a z + b}{cz+d}.
\]
The \emph{Cayley transform}
$ \mathcal C: \mathcal{H}^2 \to  \mathbb{D}^2, z\mapsto \frac{z-i}{z+i}$
is an isometry. We transfer the $G$-action to $\mathbb D^2$ by conjugation with $\mathcal C$. The stabilizers of the imaginary element $i\in \mathcal H^2$ and the origin $0=\mathcal C(i)\in \mathbb D^2$ are equal to $K$, respectively, which leads to isomorphisms of Riemannian $G$-manifolds
\[
\mathbb H^2\equiv G/K\cong\mathcal H^2\cong\mathbb D^2.
\]

%---------------------------------------------------------------------%
\subsection{Structure of the Lie group $G=\PSL(2,\R)$}\label{sec:setupnot}
%---------------------------------------------------------------------%
$G$ is a non-compact connected simple Lie group with trivial center.  
On its Lie algebra $\g\coloneq\mathfrak{sl}(2,\R)=\{Y\in \mathfrak{gl}(2,\R)\,|\, \mathrm{tr}\, Y =0 \}$, the Killing form $\mathfrak{B}$ is given by
$\mathfrak{B}(Y,Y') =4\,\mathrm{tr}(YY'),$ for $Y, Y'\in \g.$
Let 
\begin{equation*}
    \theta:
\begin{cases}
    G\to G,& [Q]\mapsto [(Q^{-1})^T]\\
    \g\to \g,& Y\mapsto -Y^T
\end{cases}
\end{equation*}
be the Cartan involutions (denoted by the same symbol). We equip $\g$ (hence also its dual $\g^\ast$ and by linear extension the complexifications of $\g$ and $\g^\ast$) with the inner product
\begin{equation} \label{eq:innerprod}
    \eklm{Y,Y'}\coloneq-\frac{1}{2}\mathfrak{B}(Y,\theta Y') = 2\,\mathrm{tr}(Y(Y')^T),\;\;Y, Y'\in \g.
\end{equation}
Let us introduce the matrices in $\g$
\begin{equation*}
\begin{split}
    X\coloneq\begin{pmatrix}\frac{1}{2} & 0 \\ 0 & -\frac{1}{2}\end{pmatrix},\;\;
    V&\coloneq\begin{pmatrix}0 & \frac{1}{2} \\ -\frac{1}{2} & 0\end{pmatrix} = \frac{1}{2}(U_+-U_-),\;\;
    X_\perp\coloneq\begin{pmatrix}0 & \frac{1}{2} \\ \frac{1}{2} & 0 \end{pmatrix}= \frac{1}{2}(U_++U_-),\\
    U_+&\coloneq\begin{pmatrix}0 & 1 \\ 0 & 0\end{pmatrix} \;\; \text{ and }\;\;  U_- \coloneq\begin{pmatrix}0 & 0 \\ 1 & 0\end{pmatrix}.
\end{split}
\end{equation*}
Then $\{X,X_\perp,V\}$ is an orthonormal basis of $\g$ and the above matrices have the Lie brackets
\bq \begin{split} \label{eq:Liebrackets}
    [X,V]&=X_\perp, \qquad\quad [X,X_\perp]=V,\qquad [V,X_\perp]=X.\\
    [X,U_\pm]&=\pm  U_\pm,\qquad [U_+,U_-]=2X.
    \end{split}
\eq
We have the Cartan decomposition $G=K{\exp}({\p})$ with ${\p}\coloneq\mathrm{span}_\R\{X,X_\perp\}$, 
the Iwasawa decomposition $G=K A N_{+}$ with
\bqn
    A\coloneq\bigg\{\left[\begin{pmatrix}e^t & 0 \\ 0 & e^{-t}
    \end{pmatrix}\right]\,\Big|\,t\in \R\bigg\}, \qquad N_{+}\coloneq\bigg\{\left[\begin{pmatrix}
        1 & x \\ 0 & 1
    \end{pmatrix}\right]\,\Big|\,x\in \R\bigg\},
\eqn
as well as the associated infinitesimal Iwasawa decomposition $\g={\k}\oplus {\aL}\oplus {\nL}_{+}$, where ${\k}\coloneq\R V,{\aL}\coloneq\R X$ and ${\nL}_{+}\coloneq\R U_+$.
The Iwasawa decomposition defines analytic functions
\bq \label{eq:Iwasawafcnbreve}
    \kappa:G\to K,\qquad {\mathscr{A}}:G\to{\aL},\qquad         {\mathscr{N}}:G\to{\nL}_{+}
\eq
such that $g=\kappa(g){\exp}({\mathscr{A}}(g)){\exp}({\mathscr{N}}(g))$ for all $g\in G$. Explicitly, if
$ g=\left[\begin{pmatrix}
        a & b\\
        c & d
    \end{pmatrix}\right] \in~G,$ $a,b,c,d \in \R$,
and we write $r\coloneq\sqrt{a^2+c^2}>0$ (recalling that $ad-bc=1$), then
\bq \label{eq:Iwasawaexplicit}
    \kappa(g) =
    \left[\begin{pmatrix}
        \frac{a}{r} & -\frac{c}{r}\\
        \frac{c}{r} & \frac{a}{r}
    \end{pmatrix}\right],
    \qquad \mathscr{A}(g)=
    \begin{pmatrix}
        \log r & 0\\
        0 & -\log r
    \end{pmatrix},
    \qquad \mathscr{N}(g)=
    \begin{pmatrix}
        0 & \frac{ab+cd}{r}\\
        0 & 0
    \end{pmatrix}.
\eq
Furthermore, we define $ {\nL}_{-}\coloneq\theta {\nL}_{+}=\R U_-$
so that we have the Bruhat decomposition of the Lie algebra
\bq  \label{eq:gBruhat}
    \g =  \nL_{+}\oplus\aL \oplus\nL_{-}.
\eq
%The groups $ %\label{eq:parabolicspsl2r}
%    P^\pm \coloneq A N_{\pm}
%$
%are two standard minimal parabolic subgroups of $G$ consisting of all upper ($+$) or lower ($-$) triangular matrices, respectively. 
The centralizer $Z_{K}(\aL)=\{k\in K\,|\,\Ad(k)\aL=0\}$ is trivial and the normalizer $N_{K}(\aL)=\{k\in K\,|\,\Ad(k)\aL\subset \aL\}$ is the group  generated by the identity $e=[I]$ and the element 
\bq \label{eq:brevew0}
    w_0\coloneq
    \left[\begin{pmatrix}
        0 & 1 \\ -1 & 0
    \end{pmatrix}\right]
    \in N_{K}(\aL).
\eq
Thus the Weyl group is given by
$
    W=\frac{N_{K}(\aL)}{Z_{K}(\aL)}=\{e,w_0\}.
$
Denote by $\alpha\in  \aL^\ast$ the positive restricted root of $\g$ with ${\nL}_{+}=\g_{\alpha}=\{Y\in \g\,|\,[X,Y]=\alpha(X)Y\}$, {and $\Sigma^+$ its set of positive restricted roots.}
Then by \eqref{eq:Liebrackets} we have
\bq \label{eq:brevealpha0}
    \alpha(X)=1.
\eq
Finally, we define the special element
\bq
\label{eq:rho}
    \varrho\coloneq\frac{1}{2}(\dim\, \g_{\alpha}) \alpha 
    =\frac{1}{2}\alpha\in \aL^\ast
\eq
Note that in view of \eqref{eq:Iwasawaexplicit} and \eqref{eq:brevealpha0} we have the explicit formula
\bq
\alpha(\mathscr{A}(g))=\log(a^2+c^2),\qquad g=\left[\begin{pmatrix}
        a & b\\
        c & d
    \end{pmatrix}\right]\in G. \label{eq:alphaA}
\eq

%---------------------------------------------------------------------%

%---------------------------------------------------------------------%
\subsubsection{Left invariant vector fields and tangent bundle of quotient manifolds}
%---------------------------------------------------------------------%

Every Lie algebra element $Y\in \g$ can be considered as a left invariant vector field on $G$ denoted again by $Y$, defined by
\[
    Y_g(f)\coloneq\frac{d}{dt}\Big|_{t=0}f(g\exp(tY)),\qquad f\in \Cinft(G),g\in G.
\]
This fixes a trivialization
\bq
    TG=G\times \g\label{eq:identifTG}
\eq
of the tangent bundle of the Lie group $G$. Thus smooth vector fields on $G$ are identified with smooth functions $G\to \g$. Furthermore, the identification \eqref{eq:identifTG} has the important property that for every $g\in G$ the differential $d(\cdot g):TG\to TG$ of the right multiplication map $G\to G$, $g'\mapsto g'g$, corresponds to the map
\begin{align}\begin{split}
    G\times \g &\to G\times \g,\\
    (g',Y)&\mapsto (g'g,\Ad(g^{-1})Y).\label{eq:Addiff}\end{split}
\end{align}
Here we denote by $\Ad(g):\g\to \g$ the adjoint action of a Lie group element $g$, which in our setting is simply given by matrix conjugation:
\[
    \Ad(g)Y=QYQ^{-1}\qquad \forall g=[Q]\in G,\; Y\in \g.
\]

%---------------------------------------------------------------------%
\subsection{The hyperbolic plane as a symmetric space}\label{sec:symmspace}
%---------------------------------------------------------------------%

Motivated by Section \ref{sec:Poincare} we regard the hyperbolic plane as the quotient $     \mathbb H^2= G/K=\PSL(2,\R)/\SO(2)$
which we equip with the $G$-action induced by left multiplication. The Cartan decomposition $G=K\exp(\p)$ and the observation \eqref{eq:Addiff} provide an identification 
\bq
    T\mathbb H^2=G\times_{\Ad(K)} \p,\label{eq:TGK}
\eq
where the right-hand side is the associated vector bundle defined by the representation of $K$ on $\p$ given by the restricted adjoint action. It is defined by
$
G\times_{\Ad(K)} \p=(G\times \p)/\sim$, where for $g,g' \in G,$ $Y,Y'\in \p$ the equivalence $(g,Y)\sim (g',Y')$ means that there exists $k\in K$ such that $ g'=gk,Y'=\Ad(k^{-1})Y$. 
The bundle projection is given by $[g,Y]\mapsto gK$. With respect to the identification \eqref{eq:TGK} the $G$-invariant Riemannian metrics on $\mathbb H^2$ are in one-to-one correspondence with the $\Ad(K)$-invariant inner products on $\p$. Of the latter, we choose the restriction of the inner product \eqref{eq:innerprod} to $\p$, which fixes a Riemannian metric of constant curvature $-1$ on $\mathbb H^2$. %Similarly, the smooth vector fields on $\mathbb H^2$ are in one-to-one correspondence with the smooth functions $\mathcal X:G\to \s$ having the equivariance property $\mathcal X(gk)=\Ad(k^{-1})\mathcal X(g)$ for all $g\in G,k\in K$. For each $X\in \s$ we can define a function $\mathcal X_X:G\to \s$ with that equivariance property using the Iwasawa decomposition \eqref{eq:Iwasawafcnbreve}:
%\[
%    \mathcal X_X(g)\coloneq\Ad(\kappa(g^{-1}))X.
%\]
%The vector fields $\mathcal X_{X_1},\mathcal X_{X_2}$ with respect to any basis $X_1,X_2$ of $\s$ provide a global frame of $T\mathbb H^2$ hence a trivialization
%\bq
%    T\mathbb H^2\cong \mathbb H^2\times \s.\label{eq:TGKtrivial}
%\eq
By taking differentials, the $G$-action on $\mathbb H^2$ lifts to the tangent bundle $T\mathbb H^2$. With respect to the identification \eqref{eq:TGK} this lifted action is simply given by $g\cdot [g',Y]=[gg',Y]$ for $g\in G$, $[g',Y]\in G\times_{\Ad(K)} \p$. %, while with respect to the trivialization \eqref{eq:TGKtrivial} it corresponds to an action on $\mathbb H^2\times \s$ given by the somewhat clumsy formula
%\bq
%g\cdot (g'K,X)=\Big(gg'K,\Ad\big(\kappa({g'}^{-1}g^{-1})^{-1}\kappa({g'}^{-1})\big)X\Big).\label{eq:Gactiontrivial}
%\eq
Finally, fixing an orientation of the vector space $\p$ provides an orientation of the manifold $\mathbb H^2= G/K$ in view of \eqref{eq:TGK}.
\subsubsection{The unit sphere bundle}\label{sec:unitspherebundle}
Let $S\mathbb H^2\subset T\mathbb H^2$ be the Riemannian unit sphere bundle given by all tangent vectors of length $1$. Denote by $\pi:S\mathbb H^2\to\mathbb H^2$ the bundle projection.  Let $\mathbb{S}^1_\p\subset \p$ be the circle formed by all Lie algebra elements $Y\in \p$ with $\norm{Y}=1$. Then the identification \eqref{eq:TGK} restricts to an identification
\bq
 S\mathbb H^2=G\times_{\Ad(K)} \mathbb{S}^1_\p,\label{eq:identSH}
\eq
the associated bundle on the right-hand side being defined similarly as in \eqref{eq:TGK}. By a straightforward computation, one verifies that the $G$-action on  $S\mathbb H^2$ is free and transitive. Choosing the base point $[e,X]\in G\times_{\Ad(K)} \mathbb{S}^1_\p$, we obtain a $G$-equivariant diffeomorphism
\bq
S\mathbb H^2\cong G, \label{eq:GSH2diffeo}
\eq
where $G$ acts on itself by left multiplication. Thus, as a $G$-manifold, the unit sphere bundle $S\mathbb H^2$ is just the Lie group $G$. This allows for an efficient algebraic description of many geometric and dynamical objects of our interest.
\begin{lem}\label{lem:fiberdiffeo}For each point $x=gK\in \mathbb H^2$ the image of the fiber  $S_x\mathbb H^2$ under the $G$-equivariant diffeomorphism $S\mathbb H^2\cong G$ fixed in \eqref{eq:GSH2diffeo} is given by the set $gK\subset G$.
\end{lem}
\begin{proof}This is an immediate consequence of the identification $S\mathbb H^2=G\times_{\Ad(K)} \mathbb{S}^1_\p$ from \eqref{eq:identSH} and the fact that the $K$-action on $\mathbb{S}^1_\p$ is transitive.
\end{proof}

%---------------------------------------------------------------------%
\subsubsection{The geodesic flow}\label{sec:geodesicflow}
%---------------------------------------------------------------------%

In view of the identification $TG=G\times \g$ fixed in \eqref{eq:identifTG} and the identification $G=S\mathbb H^2$ from \eqref{eq:GSH2diffeo} the Bruhat decomposition \eqref{eq:gBruhat} becomes a splitting of the tangent bundle of $S\mathbb H^2$ into flow, stable (+) and unstable ($-$) bundles:
\bq
    T (S\mathbb H^2)=E_0\oplus E_+\oplus E_-,\qquad 
     E_0=G\times \aL,\qquad E_\pm= G\times \nL_{\pm}.\label{eq:splitting}
\eq
We define the dual splitting 
$T^*(S\mathbb H^2)=E_0^\ast\oplus E_+^\ast\oplus E_-^\ast$ by the fiber-wise relations
$$
    E^\ast_0(E_+\oplus E_-)=0,\qquad E^\ast_\pm(E_0\oplus E_\pm)=0.\label{eq:dualsplitting}
$$
The left invariant vector field $X$ is the geodesic vector field on $S\mathbb H^2=G$. Equivalently, the geodesic flow $\varphi_t:S\mathbb H^2\to S\mathbb H^2$ is given by the right-multiplication
\[
\varphi_t(g)=g \exp(tX).
\]
Recalling \eqref{eq:Addiff}, this means that the derivative $d\varphi_t:G\times \g\to G\times \g$ is given by
\begin{align*}
   d\varphi_t(g,Y)&=(g \exp(tX),\Ad(\exp(-tX))Y)
   =(g \exp(tX),e^{-t\;\mathrm{ad}_{X}}Y),
\end{align*}
where the endomorphism $\mathrm{ad}_{X}:\g\to \g$ is given by the Lie bracket $\mathrm{ad}_{X}(Y)=[X,Y]$. Since $\aL=\R X$ and $\nL_{\pm}=\R U_\pm$, we see from the Lie brackets \eqref{eq:Liebrackets} that $d\varphi_t$ leaves each of the subbundles $E_0,E_\pm$ occurring in \eqref{eq:splitting} invariant and acts on them according to
\bq
d\varphi_t(g,Y)=\begin{cases}
    (g \exp(tX),Y), & Y \in \aL,\\
    (g \exp(tX),e^{\mp t}Y), & Y \in \nL_{\pm}.\label{eq:hyperbolicaction}
\end{cases}
\eq

\subsubsection{Measures and the natural pushforward}\label{sec:measures}
%---------------------------------------------------------------------%

Let $d\mu_L$ be the Liouville measure on $S\mathbb H^2$, where we identify the latter with the contact submanifold $S^\ast \mathbb H^2\subset T^\ast \mathbb H^2$ using the Riemannian metric. Then $d \mu_L$ agrees with the Riemannian measure defined by the Sasaki metric on $S\mathbb H^2$. Furthermore, let $d x$ be the Riemannian measure on $\mathbb H^2$. Then each fiber $S_x\mathbb H^2$ of $S\mathbb H^2$ carries a smooth measure $dS_x$ characterized by the property that
\bq
\int_{S\mathbb H^2}f(x,\xi)\d\mu_L (x,\xi) = \int_{\mathbb H^2} \int_{S_x\mathbb H^2}f(x,\xi)\d S_x(\xi)\d x \qquad \forall\,f\in C_c(G).\label{eq:measuresgeometric}
\eq
Explicitly, $dS_x$ is the Sasaki-Riemannian measure on $S_x\mathbb H^2$. %Fiber-wise integration defines a natural pushforward
%\begin{align*}
%    \pi_\ast: \CT(S\mathbb H^2)&\to \CT(\mathbb H^2)\\
%    f&\mapsto \Big(x\mapsto \int_{S_x\mathbb H^2}f(x,\xi)\d S_x(\xi)\Big)
%\end{align*}
%which allows us to rewrite \eqref{eq:measuresgeometric} for all $f\in \CT(S\mathbb H^2)$ as
%\bq
%\int_{S\mathbb H^2}f(x,\xi)\d (x,\xi) = \int_{\mathbb H^2} \pi_\ast f(x)\d x.\label{eq:pshfwdgeometric}
%\eq

All of this has an analogous algebraic description: Identifying $G=S\mathbb H^2$ as $G$-spaces as explained above, we can choose a Haar measure $dg$ on $G$ such that $dg=d\mu_L$. Let $dk$ be the Haar measure on $K$ with $\mathrm{vol}\, K=2\pi$. Passing to the coset notation $x=gK$ for points in $G/K=\mathbb H^2$, the measures $d(gK)=dx$ and $dk$ satisfy for all $f\in C_c(G)$
\begin{align}\label{eq:measureGK}
 \int_G f(g)\d g&=  \int_{G/K}\int_K f(gk) \d k \d (gK),
\end{align}
c.f.\ \cite[Thm.\ 8.36]{knapp88}. Combining this with Lemma \ref{lem:fiberdiffeo}, we see that \eqref{eq:measureGK} is the algebraic version of the geometric integration formula \eqref{eq:measuresgeometric}. Note that the measure $d\mu_L=dg$ is invariant under the geodesic flow $\varphi_t$.

We choose the Haar measure $da$ on the abelian group $A$ in such a way that it corresponds to $(2\pi)^{-1/2}$ times the Lebesgue measure under the diffeomorphism $\R\to A$, $t\mapsto \exp(tX)$.  Then, in accordance with \cite[Eq.~(2.15)]{HHS12}, we choose the Haar measure $dn$ of the abelian group $N_+$ in such a way that we have for all $f\in C_c(\mathbb H^2)$:
\bq
\int_{\mathbb H^2} f(x) \d x = \int_A \int_{N_+} f(an\cdot o) \d n \d a, \label{eq:dxdadn}
\eq
where $o=K\in G/K=\mathbb H^2$ is the canonical base point. One then computes that $dn$ corresponds to $(\sqrt{2}\pi)^{-1}$ times the Lebesgue measure under the diffeomorphism $\R\to N_+$, $t\mapsto \exp(tU_+/\sqrt{2})$, where $\frac{1}{\sqrt{2}}U_+\in \nL_+$ has unit length.

% and 
 
%The Iwasawa projection $\kappa:G\to K$ from \eqref{eq:Iwasawafcnbreve} allows us to define a diffeomorphism 
%\bq
%    \begin{split}
 %       F:G&\longrightarrow G/K\times K\\
 %       g&\longmapsto (gK,\kappa(g))
%    \end{split}\label{eq:F}
%\eq
%with inverse
%\bq
 %   \begin{split}
 %       F^{-1}:G/K\times K\longrightarrow G\\
 %       (gK,k)\longmapsto g\kappa(g^{-1}k).
 %   \end{split}\label{eq:Finverse}
%\eq
%This diffeomorphism fulfills for each $f\in C_c(G)$ the transformation formula
%\bq\label{eq:intstatement}
%        \int_{G}f(g)\d g =\int_{G/K}\int_{K}f(F^{-1}(gK,k))e^{\alpha(\mathscr A(g^{-1}k))}\d k\d(gK).
%\eq
%A proof of this formula is given for example in \cite[Proof of Prop.\ 4.3]{GHWa}, recalling from \eqref{eq:rho} that $\varrho=\frac{1}{2}\alpha$. The significance of Formula \eqref{eq:intstatement} will be explained in Section \ref{sec:bdry}.

\subsubsection{Boundary at infinity and Poisson transform}\label{sec:bdry}

As before we identify $G=S\mathbb H^2$. The Iwasawa projection $\kappa:G\to K$ identifies the circle $K=\SO(2)\cong \mathbb{S}^1$ with the \emph{(geodesic/visual) boundary at infinity} $\partial_\infty \mathbb H^2$ of the Riemannian symmetric space $\mathbb H^2$  given by the equivalence classes of all geodesic rays $r:[0,\infty)\to \mathbb H^2$ emanating from a common base point $r(0)$, where two rays $r,r'$ are equivalent if $\sup_{t\geq 0}|r(t)-r'(t)|<\infty$. Indeed, the Iwasawa decomposition and the explicit formula \eqref{eq:hyperbolicaction} imply that, in the unit tangent bundle, the geodesic rays $\{\varphi_t(g)\}_{t\geq 0}$, $\{\varphi_t(g')\}_{t\geq 0}$ starting at two arbitrary points $g,g'\in G=S\mathbb H^2$ in positive time stay at bounded distance from each other if, and only if, the same holds for the geodesic rays $\{\varphi_t(\kappa(g))\}_{t\geq 0}$, $\{\varphi_t(\kappa(g'))\}_{t\geq 0}$ starting at $\kappa(g)$ and $\kappa(g')$, respectively. However, by Lemma \ref{lem:fiberdiffeo} the elements $\kappa(g),\kappa(g')\in K\subset G=S\mathbb H^2$ are tangent vectors in the fiber $S_{o}\mathbb H^2$ over the canonical base point $o=K\in G/K=\mathbb H^2$ and an explicit computation reveals that their positive rays $\{\varphi_t(\kappa(g))\}_{t\geq 0}$, $\{\varphi_t(\kappa(g'))\}_{t\geq 0}$ remain at a bounded distance from each other if, and only if, $\kappa(g)=\kappa(g')$. In total, this gives us identifications
$$
K=S_{o}\mathbb H^2=\partial_\infty \mathbb H^2. %\label{eq:KSx0H2}
$$
Moreover, the $G$-action on $\partial_\infty \mathbb H^2=K$ extending the isometric $G$-action on $\mathbb H^2$ in the geodesic compactification $\mathbb H^2\cup \partial_\infty \mathbb H^2$ is given by
\bq
g\cdot k\coloneq\kappa(gk).\label{eq:GactionK}
\eq
%Note that the diffeomorphism $F:G\to G/K\times K$ from \eqref{eq:F} is $G$-equivariant when $G$ acts by left multiplication on itself and and $G/K$ and $K$ carries the $G$-action \eqref{eq:GactionK}. % Therefore, the transformation formula \eqref{eq:intstatement} implies that the Jacobian function of the diffeomorphism $g:K\to K$ given by the action of an element $g\in G$ is given by
%\bq
%|\det dg|_k=e^{\alpha(\mathscr A(g^{-1}k))}.\label{eq:Jacobian}
%\eq
We define the initial and end point maps $B_\pm: G=S\mathbb H^2\to \partial_\infty \mathbb H^2=K$ by
\begin{equation} \label{eq:defB}
B_+(g)\coloneq\kappa(g),\qquad B_-(g)\coloneq\kappa(gw_0),
\end{equation}
where $w_0\in N_K(\aL)\subset K$ has been introduced in \eqref{eq:brevew0}. For $g\in S\mathbb H^2$, the boundary points $B_\pm(g)\in \partial_\infty \mathbb H^2$ are represented by the geodesic rays departing from $g$ in forward ($+$) and backward ($-$) time, respectively.  See \cite[(2.1) on p.~1237]{GHWa} for an explicit computation of $B_\pm$ in the Poincaré disk model. By the Iwasawa decomposition, the maps
\bq
F_\pm: S\mathbb H^2\to \mathbb H^2\times\partial_\infty \mathbb H^2,\qquad g\mapsto (gK,B_\pm(g))  \label{eq:Fpm}
\eq
are diffeomorphisms. They provide two useful trivializations of the sphere bundle $S\mathbb H^2$.

There is also an important diffeomorphism
\begin{align}\begin{split}
%\Psi:G &\stackrel{\cong}{\longrightarrow} A\times(\partial_\infty \mathbb H^2)^{(2)} \nonumber \\ %\label{eq:Psi}\\
%\nonumber    g&\longmapsto (\exp(\mathscr{A}(g)),B_+(g),B_-(g)),\\
\psi:G/A &\stackrel{\cong}{\longrightarrow} (\partial_\infty \mathbb H^2)^{(2)}\label{eq:psi}\\
 gA&\longmapsto (B_+(g),B_-(g)),\end{split}
\end{align}
where $(\partial_\infty \mathbb H^2)^{(2)}\coloneq(\partial_\infty \mathbb H^2\times \partial_\infty \mathbb H^2)\setminus\mathrm{diagonal}$, see \cite[Prop.~2.10]{HHS12}.

Let now $\D'(\partial_\infty \mathbb H^2)$ be the topological dual of the space $C_c^{\infty}(\partial_\infty \mathbb H^2)=C^{\infty}(\partial_\infty \mathbb H^2)$ of test functions on the oriented manifold $\partial_\infty \mathbb H^2=K$ equipped with the volume form $dk$. Then $C^\infty(\partial_\infty \mathbb H^2)$  is embedded densely into the space of distributions $\D'(\partial_\infty \mathbb H^2)$ by the sesquilinear integration pairing with respect to $dk$.

\begin{definition}\label{def:Poisson} For $x=gK\in \mathbb H^2=G/K$ and $b=k\in \partial_\infty \mathbb H^2=K$, write
\begin{equation} \label{eq:horocycle_bracket}
    \eklm{x,b}\coloneq-2\varrho(\mathscr A(g^{-1}k))\in \R,
\end{equation}
 using the element $\varrho$ from \eqref{eq:rho}. In terms of matrices, using \eqref{eq:alphaA}, one finds the explicit formula
\bq 
\eklm{x,b}=-\log(s^2+u^2),\qquad x=gK,\;b=k,\; g^{-1}k=\left[\begin{pmatrix}
        s & t\\
        u & v
    \end{pmatrix}\right]. \label{eq:bracketexplicit}
\eq
The \emph{Poisson kernel} is the function $p\in \Cinft(\mathbb H^2\times \partial_\infty \mathbb H^2)$ given by
    \bqn
        p(x,b)\coloneq e^{\eklm{x,b}}.
    \eqn

    The \emph{Poisson-Helgason transform} with parameter $\lambda \in \C$ is the map 
    $\mathcal P_\lambda: \D'(\partial_\infty \mathbb H^2)\to \Cinft(\mathbb H^2)$ given by
   \bqn
        \mathcal P_\lambda (\omega)(x)\coloneq \omega\big(p(x,\cdot)^{1+\lambda}\big),\;\; x \in \mathbb H^2.
     \eqn
\end{definition}
In particular, $\mathcal P_\lambda$ acts on $\Cinft(\partial_\infty \mathbb H^2)\subset \D'(\partial_\infty \mathbb H^2)$ according to
\bq
\mathcal P_\lambda(f)(gK)=\int_{K}p(g{K},k)^{1+\lambda}f(k)\d k\label{eq:Poissononfunctions}.
\eq
{\begin{rem}[Conventions]\label{rem:conventions}
The definition of the bracket $\eklm{x,b}$ given in \eqref{eq:horocycle_bracket} differs from the usual definition of the \textit{horocycle bracket} (as in \cite{HHS12} or \cite[p.~81]{GASS}) by a factor of two. This is because we want our Poisson kernel to be equal to the ``geometer's Poisson kernel'', as used for example in \cite{DFG,GHWa}, which is given by
\bq
p_\mathrm{geom}(x,b):=e^{-\beta(x,b)},
\eq
where $\beta(x,b):=\lim_{t\to +\infty}(d_{\mathbb H^2}(\gamma(t),x)-t)$ is the Busemann function associated to the Riemannian distance $d_{\mathbb H^2}$ of $\mathbb H^2$ (here $\gamma$ is the unique geodesic starting at $x$ with limit $b$). Explicitly, in the Poincaré disk model (see Section \ref{sec:Poincare}), one has
    \bq 
        p_\mathrm{geom}(x,b)=\frac{1-|x|^2}{|x-b|^2}, \quad x \in \mathbb{D}^2,\; b \in \partial_\infty \mathbb{D}^2=\mathbb{S}^1.\label{eq:geometersPoissonkernel}
    \eq
The standard metric on $\mathbb H^2$ with curvature $-1$, as fixed in Section \ref{sec:Poincare}, corresponds to a choice of inner product on $\g$ for which 
\bq
\norm{\varrho}=\frac{1}{2}.\label{eq:normrho}
\eq
This has the effect that in order to achieve the equality $-\beta(x,b)=\eklm{x,b}$, which is equivalent to $p=p_\mathrm{geom}$, one needs to define $\eklm{x,b}$ as in \eqref{eq:horocycle_bracket}.
\end{rem}}
In order to describe the image of the Poisson transformation, we need to introduce spaces of smooth functions with moderate growth since we restrict our attention to spaces of distributions.
For $f\in \Cinft(\mathbb H^2)$ and $r\in [0,\infty)$, define
\[
\norm{f}_r\coloneq\sup_{x\in \mathbb H^2}|f(x)e^{-r d_{\mathbb{H}^2}(o,x)}|\in [0,\infty],
\]
where $d_{\mathbb{H}^2}(\cdot,\cdot)$ is the Riemannian distance on $\mathbb H^2$ and $o=K\in G/K=\mathbb H^2$ is the canonical base point of $\mathbb{H}^2$. Then the space of \emph{smooth functions of moderate growth} on $\mathbb H^2$ is defined as
\[
C^\infty_\mathrm{mod}(\mathbb H^2)\coloneq\bigcup_{r\in [0,\infty)}\{f\in \Cinft(\mathbb H^2)\,|\,\norm{f}_r < \infty\},
\]
equipped with the direct limit topology with respect to the norms $\norm{f}_r$. 

Let ${\Laplace}_{\mathbb H^2}\coloneq d^\ast d: \Cinft(\mathbb H^2)\to \Cinft(\mathbb H^2)$ be the non-negative Laplacian. Given $\mu\in \C$, let
\begin{equation} \label{eq:modspace}
    \mathcal{E}_\mu\coloneq\{f\in C^\infty_\mathrm{mod}(\mathbb H^2)\,|\, ({\Laplace}_{\mathbb H^2}  -\mu)f =0 \}
\end{equation}
be the  $\mu$-eigenspace of ${\Laplace}_{\mathbb H^2}$ in $C^\infty_\mathrm{mod}(\mathbb H^2)$. 

The main properties of the Poisson-Helgason transform that we shall need are summarized in the following result proved by Helgason \cite{Helg74} in the context of hyperfunctions and by van den Ban-Schlichtkrull \cite[Cor.~11.3 and Thm.~12.2]{vdBS87} and Oshima-Sekiguchi \cite[Thm.~3.15]{OS80} in the distribution context and more general settings. We follow the presentation of \cite[Lem. 2.1]{GHWa} and \cite[Sec.~6.3]{DFG}.
\begin{prop} \label{prop:Poisson}
    For every $\lambda\in \C$, the image of the Poisson transform $\mathcal P_\lambda: \D'(\partial_\infty \mathbb H^2)\to \Cinft(\mathbb H^2)$ is contained in the eigenspace $\mathcal{E}_{-\lambda(1+\lambda)}$ and $\mathcal P_\lambda$ is $G$-equivariant in the sense that
    $$
    g^\ast \mathcal P_\lambda(\omega) = \mathcal P_\lambda(|\mathrm{det}\, dg|^{-\lambda} g^\ast \omega), \;\; \forall \omega \in \mathcal{D}'(\partial_\infty \mathbb{H}^2),
    $$
    where $g^\ast \omega$ is the pullback of the distribution $\omega$ along the diffeomorphism given by the action \eqref{eq:GactionK} of the group element $g \in G$  and $|\mathrm{det}\, dg|$ is the Jacobian function of that diffeomorphism. 
Moreover, if $\lambda\not\in -\N=\{-1,-2,\ldots\}$, then 
$$\mathcal P_\lambda: \D'(\partial_\infty \mathbb H^2)\to \mathcal{E}_{-\lambda(1+\lambda)}$$ is an isomorphism of topological vector spaces.
\end{prop}
For later use, we define two positive  functions $\Phi_\pm\in \Cinft(S\mathbb H^2)$ by
\bq
\Phi_\pm\coloneq p\circ F_\pm,\label{eq:defPhipm}
\eq
with the Poisson kernel $p$ from Definition \ref{def:Poisson} and the maps $F_\pm$ from \eqref{eq:Fpm}.

\subsection{Convex-cocompact hyperbolic surfaces}\label{sec:convcocomp}

Let $\Gamma \subset G=\mathrm{PSL}(2,\mathbb{R})$ be a discrete subgroup. Its \emph{limit set} $\Lambda_\Gamma$ is the set of all accumulation points of the $\Gamma$-orbit $\Gamma\cdot o$ (or equivalently of any other $\Gamma$-orbit in $\mathbb H^2$) in the geodesic compactification $\mathbb H^2\cup\partial_\infty \mathbb H^2$. Since $\Gamma$ acts properly discontinuously on $\mathbb H^2$, we have $\Lambda_\Gamma\subset \partial_\infty \mathbb H^2$. Let $\mathrm{Conv}(\Lambda_\Gamma)\subset \mathbb H^2$ be the (geodesic) convex hull of $\Lambda_\Gamma$, a $\Gamma$-invariant set. We call the group $\Gamma$ \emph{convex-cocompact} if it is torsion-free and acts cocompactly on $\mathrm{Conv}(\Lambda_\Gamma)$. The quotient 
\[
\mathbf X_\Gamma\coloneq \Gamma \backslash \mathbb H^2
\]
is a $2$-dimensional oriented complete Riemannian manifold called a \emph{convex-cocompact hyperbolic surface}. It is either compact or of infinite volume.
 In the case of infinite volume, $\mathbf{X}_\Gamma$ is a non-compact hyperbolic surface with no cusps, i.e., only funnel ends.

\begin{ex}[Schottky surfaces]\label{ex:Schottky}
    Schottky surfaces are a special class of convex-co\-com\-pact hyperbolic surfaces defined as quotients of $\HH^2$ by \emph{Schottky groups}. Such a group is constructed from a collection of Euclidean disks $D_1, \dots, D_{2r}$ ($r \in \N$) in $\C$ with centers on the real axis and mutually disjoint closures,  in the following way: For every $1 \leq i \leq r$ there exists a unique element $S_i \in \mathrm{PSL}(2,\R)$ that maps the boundary $\partial D_i$ to the boundary of $\partial D_{i+r}$ and the interior of $D_i$ to the exterior of $D_{i+r}$. The Schottky group is the free discrete subgroup
    $$\Gamma=\langle S_1, \dots, S_r \rangle \subset \mathrm{PSL}(2,\mathbb{R})$$
    generated by $S_1, \dots, S_r$.
    %It is known (see e.g. \cite[Sec.~15.1]{Borthwickbook}) that any convex-cocompact Fuchsian group is a classical Schottky group.
\end{ex}

We write $S\mathbf X_\Gamma\subset T\mathbf X_\Gamma$ for the Riemannian unit sphere bundle of $\mathbf X_\Gamma$. %Note that in the literature it is often referred to as a \emph{phase space}. 
It can be canonically identified with the quotient manifold $\Gamma \backslash S\mathbb H^2 = \Gamma \backslash G$, recalling the identification $G=S\mathbb H^2$ via the diffeomorphism \eqref{eq:GSH2diffeo}. We denote by
\bq\begin{split} \label{eq:pi_Gamma}
    \pi_\Gamma: G=S\mathbb H^2 &\to S\mathbf X_\Gamma=\Gamma \backslash G\\
    \widetilde \pi_\Gamma: G/K=\mathbb H^2 &\to \mathbf X_\Gamma=\Gamma \backslash G/K
\end{split}
\eq
the orbit projections. For a set $F\subset S\mathbf X_\Gamma$, we use the notation $\widetilde{F}\coloneq\pi^{-1}_\Gamma(F)\subset S\mathbb H^2=G$. 

Functions and distributions on $\mathbf X_\Gamma$ and $S\mathbf X_\Gamma$ can be identified with $\Gamma$-invariant functions and distributions on $\mathbb H^2=G/K$ and $S\mathbb H^2=G$ via pullback along the surjective submersions $\widetilde\pi_\Gamma$ and $\pi_\Gamma$, respectively. Explicitly, we will use the isomorphisms
\bq\begin{split} \label{eq:isomorphism_pi_Gamma}
{\pi}_\Gamma^\ast: \D'(S\mathbf X_\Gamma)&\stackrel{\cong}{\longrightarrow} \D'(S\mathbb H^2)^\Gamma\\
\widetilde\pi_\Gamma^\ast: \Cinft( \mathbf X_\Gamma)&\stackrel{\cong}{\longrightarrow} \Cinft(\mathbb H^2)^\Gamma,\end{split}
\eq
where the upper index $\Gamma$ indicates the subspace of $\Gamma$-invariant functions and distributions, respectively. 
The Riemannian measure on $\mathbf X_\Gamma$ and the Sasaki-Riemannian measure on $S\mathbf X_\Gamma$ are the pushforwards $(\pi_\Gamma)_\ast d x$ and $(\widetilde\pi_\Gamma)_\ast d g$ of the measure $dx=d(gK)$ on $\mathbb H^2=G/K$ and the Haar/Liouville measure $d\mu_L=dg$ on $S\mathbb H^2=G$ introduced in Section \ref{sec:measures}. 

The geodesic flow on $S\mathbf X_\Gamma$ will again be denoted by $\varphi_t$. From Section \ref{sec:geodesicflow} we get the explicit formula
\begin{equation*}
\varphi_t(\Gamma  g)=\Gamma g \exp(tX)\label{eq:geodflowXGamma}
\end{equation*}
for all $\Gamma  g\in \Gamma \backslash G=S\mathbf X_\Gamma$. Every Lie algebra element $Y\in \g$ generates a smooth vector field on $S\mathbf X_\Gamma$ denoted again by $Y$, given by the pushforward of the left invariant vector field $Y$ along $\widetilde\pi_\Gamma$. In particular, the Lie algebra element $X$ generates the geodesic flow $\varphi_t$ this way. We denote by $$\pi:S\mathbf{X}_\Gamma \rightarrow \mathbf{X}_\Gamma$$ the sphere bundle projection. As the fibers of this bundle are compact, the pullback $\pi^\ast:C^{\infty}_c(\mathbf{X}_\Gamma) \rightarrow C^{\infty}_c(S\mathbf{X}_\Gamma)$ dualizes to a pushforward map
\bq
\pi_\ast: \D'(S\mathbf{X}_\Gamma)\to \D'(\mathbf{X}_\Gamma)\label{eq:pushfwd}
\eq
which acts on functions in $C^{\infty}_c(S\mathbf{X}_\Gamma)\subset \D'(S\mathbf{X}_\Gamma)$ (embedded using the $L^2$-pairing with respect to the Sasaki-Riemannian measure) by integration in the fibers of $S\mathbf{X}_\Gamma$ with respect to the fiber-wise Riemannian measures.  

\subsubsection{Smooth fundamental domain cutoffs} Let  $\mathscr C\subset S\mathbf X_\Gamma$ be a compact set. Recall that $\widetilde {\mathscr C}=\pi^{-1}_\Gamma({\mathscr C})\subset S\mathbb H^2$ denotes its $\Gamma$-invariant lift.
\begin{definition}\label{def:smoothfundcutoff}
    A \emph{smooth fundamental domain cutoff near $\widetilde {\mathscr C}$} is a function $\chi\in \CT(S\mathbb H^2)$ such that
    \begin{equation*}
    \sum_{\gamma\in \Gamma}\chi(\gamma x) = 1 \quad \forall\,x\in \widetilde {\mathscr C}.\label{eq:funddomaincutoff}
    \end{equation*}
\end{definition}
\begin{lem}\label{lem:existencecutoff}
   There exists a smooth fundamental domain cutoff near $\widetilde {\mathscr C}$.
\end{lem}
\begin{proof}Let $\mathcal F\subset S\mathbb H^2$ be a fundamental domain for the $\Gamma$-action. 
Since ${\mathscr C}$ is compact, $\mathcal F\cap \widetilde {\mathscr C}$ is compact. We can therefore choose functions $u\in \Cinft(S\mathbb H^2)^\Gamma$, $f\in \Cinft(S\mathbb H^2)$ such that $u=1$ on $\widetilde {\mathscr C}$, $f=1$ on $\mathcal F$, $f\geq 0$, and $\supp u\cap \supp f$ is compact. Then
\bq
\chi(x)\coloneq\frac{u(x)f(x)}{\sum_{\gamma\in \Gamma}f(\gamma x)},\qquad x\in S\mathbb H^2\label{eq:chidef}
\eq
defines a smooth fundamental domain cutoff near $\widetilde {\mathscr C}$.
\end{proof}
\begin{lem}\label{lem:indepchi}
    Let $v\in \D'(S\mathbb H^2)^\Gamma$, $u\in \CT(S\mathbf X_\Gamma)$, and $\chi_1,\chi_2\in \CT(S\mathbb H^2)$ two smooth fundamental domain cutoffs near $\widetilde{\supp u}$. Then
    \[
    v(\chi_1 {\pi}_\Gamma^\ast u)=v(\chi_2 {\pi}_\Gamma^\ast u).
    \]
\end{lem}
\begin{proof}We argue as in \cite[proof of Lem. 3.5]{AZ07}: Since $\sum_{\gamma\in \Gamma}\gamma^\ast\chi_j=1$ on $\supp \pi_\Gamma^\ast u$ for $j=1,2$, we find that 
    \begin{align*}
        v(\chi_1 {\pi}_\Gamma^\ast u)&=v\bigg(\bigg(\sum_{\gamma\in \Gamma}\gamma^\ast\chi_2\bigg)\chi_1 {\pi}_\Gamma^\ast u\bigg)
        =v\bigg(\chi_2\bigg(\sum_{\gamma\in \Gamma}\gamma^\ast\chi_1\bigg) {\pi}_\Gamma^\ast u\bigg)
        =v(\chi_2 {\pi}_\Gamma^\ast u),
    \end{align*}
    where in the second step we used the $\Gamma$-invariance of $v$ and ${\pi}_\Gamma^\ast u$.
\end{proof}

Conversely, it will sometimes be convenient to express the evaluation of a $\Gamma$-invariant distribution on an arbitrary test function as an evaluation at the product of the lift of a function on $S\mathbf X_\Gamma$ and a smooth fundamental domain cutoff:
\begin{lem}\label{lem:newlemma}
Let $f\in \CT(S\mathbb H^2)$. Then there is a function $u_f\in \CT(S\mathbf X_\Gamma)$ and a smooth fundamental domain cutoff $\chi\in \CT(S\mathbb H^2)$ near $\widetilde{\supp u_f}$ such that for all $v\in \D'(S\mathbb H^2)^\Gamma$ one has
\[
v(f)=v(\chi \pi_\Gamma^\ast u_f).
\]
Moreover, the map
\begin{align*}
    \CT(S\mathbb H^2) &\to \CT(S\mathbf X_\Gamma)\\
    f&\mapsto u_f
\end{align*}
is continuous.
\end{lem}
\begin{proof}
    Let $\mathcal F\subset S\mathbb H^2$ be a fundamental domain for the $\Gamma$-action, $\mathcal C\subset S\mathbb H^2$ a compact set, and $f\in \CT(S\mathbb H^2)$ with $\supp f\subset \mathcal C$. Then the set $\Gamma_{\mathcal C}:=\{\gamma \in \Gamma\,|\,\F\cap \gamma^{-1}\mathcal C\neq \emptyset\}\subset \Gamma$ is finite and the support of the $\Gamma$-invariant function
    \[
    \tilde u_f:=\sum_{\gamma\in \Gamma}\gamma^\ast f\in \Cinft(S\mathbb H^2)^\Gamma
    \]
    intersects $\F$ only in the compact set $\widetilde{\mathcal K}_{\mathcal C}:=\bigcup_{\gamma\in \Gamma_{\mathcal C}}\F\cap \gamma^{-1}\mathcal C$. Therefore, the unique function $u_f\in \Cinft(S\mathbf X_\Gamma)$ with $\pi_\Gamma^\ast u_f=\tilde u_f$ satisfies $\supp u_f\subset {\mathcal K}_{\mathcal C}$, where ${\mathcal K}_{\mathcal C}:=\pi_\Gamma(\widetilde{\mathcal K}_{\mathcal C})\subset S\mathbf X_\Gamma$ is compact, in particular $u_f\in \CT(S\mathbf X_\Gamma)$. Moreover, if $D$ is any differential operator on $S\mathbf X_\Gamma$, it lifts to a $\Gamma$-invariant differential operator $\widetilde D$ on $S\mathbb H^2$ and we have $\norm{Du_f}_\infty\leq \Vert\widetilde D f\Vert_\infty$, which shows that the map $f\mapsto u_f$ is continuous with respect to the standard LF topologies on $\CT(S\mathbb H^2)$ and $\CT(S\mathbf X_\Gamma)$, respectively. 

    Finally, let $v\in \D'(S\mathbb H^2)^\Gamma$ and $\chi\in  \CT(S\mathbb H^2)$ a smooth fundamental domain cutoff near $\widetilde{\supp u_f}=\supp \tilde u_f$, which exists by Lemma \ref{lem:existencecutoff}. Then
    \[
    v(f)=v\bigg(\sum_{\gamma\in \Gamma}((\gamma^{-1})^\ast \chi) f\bigg)=\sum_{\gamma\in \Gamma}v\big((\gamma^{-1})^\ast \chi) f\big)=\sum_{\gamma\in \Gamma}v(\chi \gamma^\ast f)=u\bigg(\chi\sum_{\gamma\in \Gamma}\gamma^\ast f\bigg)=v(\chi \tilde u_f),
    \]
    where the third equality is due to the $\Gamma$-invariance of $v$. 
\end{proof}

%---------------------------------------------------------------------%
\section{Classical and quantum resonances}
\label{sect:classical_quantum}
%---------------------------------------------------------------------%
In this section, we introduce the concept of the Pollicott-Ruelle resonances and the quantum resonances for convex-cocompact quotients as it was established in \cite{DG,GHWa,Guillope, MazzeoMelrose}. In all of the following, $\mathbf{X}_\Gamma=\Gamma\backslash \mathbb{H}^2$ denotes an oriented convex-cocompact hyperbolic surface {(see Section \ref{sec:convcocomp})}. 

%---------------------------------------------------------------------%
\subsection{Classical resonances}
\label{sect:classical_resonances}
%---------------------------------------------------------------------%

 As before, let $X$ be the vector field generating the geodesic flow $\varphi_t$ on the unit tangent bundle $S\mathbf{X}_\Gamma$ of the convex-cocompact hyperbolic surface $\mathbf{X}_\Gamma$.
Dyatlov and Guillarmou \cite[Prop.~6.2]{DG} showed that the  
$L^2$-resolvent 
$$(-X-\lambda)^{-1}: L^2(S\mathbf{X}_\Gamma) \longrightarrow L^2(S\mathbf{X}_\Gamma)$$
of the operator $-X$, defined for $\mathrm{Re}(\lambda) \gg 0$ by the integral formula
$$
(-X-\lambda)^{-1}f=-\int_0^\infty e^{-\lambda t}f\circ \varphi_{-t}\; \d t
$$
has a meromorphic continuation to $\C$ as a family of continuous operators 
$$R_X(\lambda):C^\infty_c(S\mathbf{X}_\Gamma) \rightarrow \D'(S\mathbf{X}_\Gamma).$$ 
The poles of $R_X(\lambda)$ are called \emph{classical resonances} or \emph{Pollicott-Ruelle resonances}. 
%\begin{rem} \label{rem:correlation}
%    The Pollicott-Ruelle resonances are closely related to the decay of correlations. More precisely, the resolvent operators
%    $$R_X(\lambda; f_1,f_2)\coloneq - \int_0^\infty e^{-\lambda t} C_X(t;f_1,f_2) \; \mathrm{d}t= \langle R_X(\lambda) f_1,f_2 \rangle_{L^2}, \;\; f_1,f_2 \in C^\infty_c(S\mathbf{X}_\Gamma)$$
%    can be interpreted as the Laplace transform of the \emph{correlation functions}
%    $C_X(t;f_1,f_2)\coloneq\int_{S\mathbf{X}_\Gamma} \mathcal{L}_t f_1 . f_2 \; \mathrm{d}\mu_L,$
%    where $\mathcal{L}_t: C^\infty_c(S\mathbf{X}_\Gamma) \rightarrow C^\infty_c(S\mathbf{X}_\Gamma), f \mapsto f \circ \varphi_{-t}$ is the linear \emph{transfer} operator of the geodesic flow with its generator $-X$, and $\mu_L$ is the Liouville measure which is $\varphi_t$-invariant.
%\end{rem}
%Moreover, $R_X(\lambda)$ has finite rank polar part each pole $\lambda$ where the polar part is of the form
%$$-\sum_{j=1}^{J(\lambda_0} \frac{(-X-\lambda_0)^{j-1}\Pi_{\lambda_0}^X}{(\lambda-\lambda_0)^j}, \;\; J(\lambda_0) \in \N,$$
%where $J(\lambda_0)$ are the Jordan blocks.
Moreover, by \cite[Thm.~2,~p.~3092]{DG}, the following holds. 
Given a classical resonance $\lambda_0\in \C$, the residue of $R_X(\lambda)$  at $ \lambda=\lambda_0$ and its formal $L^2$-adjoint are finite rank operators
\bqn
\begin{split}
\Pi_{\lambda_0}^X&\coloneq \mathrm{Res}_{\lambda=\lambda_0}(R_X(\lambda)): C^\infty_c(S\mathbf{X}_\Gamma) \rightarrow \D'(S\mathbf{X}_\Gamma)\\
\Pi_{\lambda_0}^{X^\ast}&\coloneq(\Pi_{\lambda_0}^X)^\ast: C^\infty_c(S\mathbf{X}_\Gamma) \rightarrow \D'(S\mathbf{X}_\Gamma)\end{split}
\eqn
that commute with $X$ and satisfy \bq
\begin{split}
\mathrm{Ran}(\Pi_{\lambda_0}^X) &\subset \ker (-X-\lambda_0)^{J(\lambda_0)},\\
\mathrm{Ran}(\Pi_{\lambda_0}^{X^\ast}) &\subset \ker (X-\overline\lambda_0)^{J(\lambda_0)}\label{eq:RanRan}\end{split}
\eq for some power $J(\lambda_0)\in \N$ that we choose minimal.

\begin{definition}[(Generalized) classical (co-)resonant states] \label{defn:Ruelle_resonant_states}Let $\lambda_0\in \C$ be a classical resonance. 
    The spaces of \emph{generalized classical resonant states} and \emph{generalized classical co-resonant states} of order $j\geq 1$ are
    \begin{equation*}\begin{split}
        \mathrm{Res}^j_X(\lambda_0)&\coloneq \mathrm{Ran}(\Pi_{\lambda_0}^X)\cap \ker (-X-\lambda_0)^j\subset \D'(S\mathbf{X}_\Gamma),\\
        \mathrm{Res}^j_{X^\ast}(\lambda_0)&\coloneq \mathrm{Ran}(\Pi_{\lambda_0}^{X^\ast})\cap \ker (X-\overline\lambda_0)^j\subset \D'(S\mathbf{X}_\Gamma).\end{split}
    \end{equation*}
    We call $\mathrm{Res}^1_X(\lambda_0)$ and $\mathrm{Res}^1_{X^\ast}(\lambda_0)$ the spaces of \emph{classical resonant states} and \emph{classical co-resonant states}, respectively, and denote by
    \begin{equation} \label{eq:widetildeRes}\begin{split}
    \widetilde{\mathrm{Res}}^{1}_{X}(\lambda_0)&\coloneq\pi_\Gamma^{*}(\mathrm{Res}^{1}_{X}(\lambda_0))\in \mathcal D'(S\mathbb H^2)^\Gamma,\\
    \widetilde{\mathrm{Res}}^{1}_{X^\ast}(\lambda_0)&\coloneq\pi_\Gamma^{*}(\mathrm{Res}^{1}_{X^\ast}(\lambda_0))\in \mathcal D'(S\mathbb H^2)^\Gamma
    \end{split}
    \end{equation}
    the spaces of $\Gamma$-invariant lifts.
\end{definition}
Define the \emph{incoming and outgoing tails} of the geodesic flow $\varphi_t$ by
\[
\Upsilon_\pm\coloneq\{\xi \in S\mathbf{X}_\Gamma \,|\, \{\varphi_{\mp t}(\xi):t\in [0,\infty)\}\text{ is bounded} \},
\]
as well as the \emph{trapped} or \emph{non-wandering set}
\[
\Upsilon\coloneq\Upsilon_+\cap \Upsilon_-.
\]
For a distribution $v\in \D'(S\mathbf{X}_\Gamma)$ we denote by $\mathrm{WF}(v)\subset T^\ast (S\mathbf{X}_\Gamma)$ its wavefront set.
\begin{lem}[{\cite[Thm.~2,~p.~3092]{DG}}]\label{lem:ResResetoile}For every classical resonance $\lambda_0\in \C$ and every $j\in \{1,\ldots,J(\lambda_0)\}$ one has
\bqn
\begin{split}
    \mathrm{Res}^j_X(\lambda_0)&=\{v \in \D'(S\mathbf{X}_\Gamma)\,|\, \mathrm{supp}(v) \subset \Upsilon_+, \mathrm{WF}(v) \subset E^*_+, (-X-\lambda_0)^jv=0\},\\
    \mathrm{Res}^j_{X^\ast}(\lambda_0)&=\{v^\ast \in \D'(S\mathbf{X}_\Gamma)\,|\, \mathrm{supp}(v^\ast) \subset \Upsilon_-, \mathrm{WF}(v^\ast) \subset E^*_-, (X-\overline\lambda_0)^jv^\ast=0\}.
\end{split}
\eqn
\end{lem}
From the explicit local presentation of the singular part of the resolvent $R(\lambda)$ in \cite[Thm.~2]{DG} and the general fact that every Jordan chain of a matrix contains an eigenvector it follows that have for every classical resonance $\lambda_0\in \C$ we have
\begin{equation*}
\mathrm{Res}^1_{X}(\lambda_0)\neq \{0\},\qquad \mathrm{Res}^1_{X^\ast}(\lambda_0)\neq \{0\}.\label{eq:classicalresonancecharacterization}
\end{equation*}
Since the dual stable and unstable bundles $E_+^\ast$, $E_-^\ast$ are transverse to each other, a generalized resonance state $v$ and a generalized coresonance state $v^\ast$ of the same classical resonance $\lambda_0$ satisfy the Hörmander criterion \cite[Thm.~8.2.14]{Hor90}, so that the distributional product $
v\cdot \bar v^\ast\in \D'(S\mathbf{X}_\Gamma)$ 
is well-defined. Since $\supp v\subset \Upsilon_+$ and $\supp v^\ast\subset \Upsilon_-$, we have $\supp v\cdot \bar v^\ast\subset \Upsilon$. Moreover, if $v$ is a classical resonant state {of the resonance $\lambda_0$ and $v^\ast$ a classical co-resonant state of the resonance $\lambda_0'$, then we have
\bq \label{eq:distproductXzero}
X(v\cdot \bar v^*)=(Xv)\cdot \bar v^* + v \cdot X\bar v^*=(\lambda_0'-\lambda_0)(v\cdot \bar v^*),
\eq
which implies \eqref{eq:quasiinvariancePS} (where $\lambda_0=s_0$, $\lambda_0'=\bar s_0'$). In particular, if $\lambda_0=\lambda_0'$, then \eqref{eq:distproductXzero} shows that} the distribution $v\cdot \bar v^\ast$ is invariant under the geodesic flow. 

There is an important subspace of (generalized) classical (co-)resonant states that are invariant under the stable and unstable horocycle flows, i.e., killed by the vector fields $U_\pm$. They are referred to as the (generalized) \emph{first band} (co-)resonant states since by \cite[Prop.~1.3]{GHWa} the classical resonance spectrum has a band structure generated algebraically by the operators $U_\pm$ (outside the real axis the bands consist of lines parallel to the imaginary axis) and the (generalized) first band (co-)resonant states are those associated to classical resonances $\lambda_0$ lying in the first band. 

%---------------------------------------------------------------------%
\subsection{Quantum resonances}\label{sec:quantumresonances}
%---------------------------------------------------------------------%

Here we mainly follow \cite{GHWa}. Recall that ${\Laplace}$ is the non-negative Laplace-Beltrami operator on $\mathbf{X}_\Gamma = \Gamma\backslash\HH^2$.

\subsubsection{Quantum resonant states and their boundary values}
 For each quantum resonance $s_0\in \C$ the residue of $R_{\Laplace}(s)$ at $s=s_0$ is a finite rank operator
$$
\Pi_{s_0}^{\Laplace}\coloneq\mathrm{Res}_{s=s_0}(R_{\Laplace}(s)): C^\infty_c(\mathbf{X}_\Gamma) \rightarrow C^\infty(\mathbf{X}_\Gamma)
$$
that commutes with ${\Laplace}$ and satisfies $\mathrm{Ran}(\Pi_{s_0}^{\Laplace})\subset \ker ({\Laplace}-s_0(1-s_0))^{J(s_0)}$ for some $J(s_0)\in \N$ that we choose minimal.

\begin{definition}[(Generalized) quantum resonant states] \label{defn:Quantum_resonant_states}
   Let $s_0\in \C$ be a quantum resonance. 
        The space of \emph{generalized quantum resonant states} of order $j\geq 1$ is 
         \begin{equation*}
            \mathrm{Res}^{j}_{\Laplace}(s_0)\coloneq\mathrm{Ran}(\Pi_{s_0}^{\Laplace})\cap \ker({\Laplace}-s_0(1-s_0))^{j}\subset C^\infty(\mathbf{X}_\Gamma).
        \end{equation*}
        We call $\mathrm{Res}^{1}_{\Laplace}(s_0)$ the space of \emph{quantum resonant states}  and denote by
    $ \widetilde{\mathrm{Res}}^{1}_{\Laplace}(s_0)\in \Cinft(\mathbb H^2)^\Gamma$ 
    the space of $\Gamma$-invariant lifts.
     %     \item The \emph{multiplicity} of the quantum resonance $s_0$ is the rank of the operator $\Pi_{s_0}^{\Laplace}$.
\end{definition}
It is important to observe that the $\Gamma$-invariant lifts of quantum resonant states have moderate growth: For each quantum resonance $s_0\in \C$ the precise asymptotic estimate of generalized quantum resonant states established in \cite[(1.1)]{GHWa} implies that one has
\bq
\widetilde{\mathrm{Res}}^{1}_{\Laplace}(s_0)\subset\mathcal{E}_{s_0(1-s_0)},\label{eq:rescontainedinE}
\eq
where $\mathcal{E}_{s_0(1-s_0)}$ was defined in \eqref{eq:modspace}. {If $\Re s_0\geq 0$, then by \cite[(1.1)]{GHWa} the quantum resonant states in $\mathrm{Res}^{1}_{\Laplace}(s_0)$ are actually bounded; that is, in terms of $\Gamma$-invariant lifts:
\bq
\widetilde \varphi\in\widetilde{\mathrm{Res}}^{1}_{\Laplace}(s_0),\;\Re s_0\geq 0\implies \norm{\widetilde \varphi}_\infty<\infty.\label{eq:resstatebounded}
\eq}
Furthermore, from the explicit local presentation of the singular part of the resolvent $R_{\Laplace}(s)$ in \cite[Thm.~4.2]{GHWa} and the general fact that every Jordan chain of a matrix contains an eigenvector it follows that for every quantum resonance $s_0\in \C$ we have
\bq
\mathrm{Res}^1_{{\Laplace}}(s_0)\neq \{0\}.\label{eq:quantumresonancecharacterization}
\eq
Given a quantum resonance $s_0$ with $s_0\not\in  -\frac{1}{2}-\frac{1}{2}\N_0$ and a quantum resonant state $\phi\in \mathrm{Res}^{1}_{\Laplace}(s_0)$, Proposition \ref{prop:Poisson} and the fact that $\widetilde{\pi}^*_\Gamma(\phi)$ has moderate growth allow us to define the Helgason boundary value
\bq
T_{\phi}\coloneq(\mathcal P_{s_0-1}^{-1}\circ \widetilde{\pi}^*_\Gamma)(\phi)\in \mathcal D'(\partial_\infty \mathbb H^2).\label{eq:Helgbval}
\eq 

%---------------------------------------------------------------------%
\subsection{Classical-quantum correspondence}
Here we recall the classical-quantum correspondence results from \cite{GHWa}. {This reference does not treat classical co-resonant states. However, this is easily added in Section \ref{sec:cores} below.} Note that this correspondence has been generalized by \cite{Ha20} to open hyperbolic manifolds.  Recall that the (generalized) quantum resonant states are smooth functions on $\mathbf{X}_\Gamma$ while the (generalized) classical resonant states are distributions on $S\mathbf{X}_\Gamma$, and that the sphere bundle projection $\pi:S\mathbf{X}_\Gamma\to \mathbf{X}_\Gamma$ defines a natural pushforward $\pi_\ast:\D'(S\mathbf{X}_\Gamma)\to \D'(\mathbf{X}_\Gamma)$, see \eqref{eq:pushfwd}.

\begin{prop}[{\cite[Thms. 3.3 and 4.7]{GHWa}}] \label{prop:classicalquantum_corresp}For every $\lambda_0\in \C\setminus (-\frac{1}{2}-\frac{1}{2}\N_0)$ the following holds:
$\lambda_0$ is a classical first band resonance if, and only if, $\lambda_0+1$ is a quantum resonance. In this case, the natural pushforward $\pi_\ast:\D'(S\mathbf{X}_\Gamma)\to \D'(\mathbf{X}_\Gamma)$ restricts for each $j\in \N$ to a linear isomorphism of finite-dimensional vector spaces
    \[
    \pi_\ast: \mathrm{Res}^j_{X}(\lambda_0)\cap \ker U_-\stackrel{\cong}{\longrightarrow}\mathrm{Res}^{j}_{\Laplace}(\lambda_0+1).
    \]
\end{prop}
As an immediate consequence of this result and \eqref{eq:quantumresonancecharacterization} we get that a classical resonance $\lambda_0\in \C\setminus (-\frac{1}{2}-\frac{1}{2}\N_0)$ is a first band resonance if, and only if, {$\mathrm{Res}^1_{X}(\lambda_0)\cap \ker U_-\neq \{0\}$.  Given a quantum resonance $s_0$ with $s_0\not\in  -\frac{1}{2}-\frac{1}{2}\N_0$, the result of Guillarmou-Hilgert-Weich \cite[Thms.~3.3 and 4.5]{GHWa} comes with an explicit description of the inverse 
\bqn
\mathbf I_-\coloneq\pi_\ast^{-1}:\mathrm{Res}^{1}_{\Laplace}(s_0)\to \mathrm{Res}^1_{X}(s_0-1)\cap \ker U_-
\eqn
of the isomorphism from Proposition \ref{prop:classicalquantum_corresp} for $j=1$:
For $\lambda\in \C$, define the operator 
\begin{align*}
    \mathcal{Q}_{\lambda,-}:\D'(\partial_\infty \mathbb{H}^2) &\rightarrow  \widetilde{\mathrm{Res}}^{1}_X(\lambda)\cap \mathrm{ker}\, U_-
\end{align*} given by
$\mathcal{Q}_{\lambda,-}(u)\coloneq\Phi_-^{\lambda}B_-^*(u),$
for $u \in \D'(\partial_\infty \mathbb{H}^2)$, where $\Phi_-$ was defined in \eqref{eq:defPhipm} and $B^*_-: \mathcal{D}'(\partial_\infty \mathbb{H}^2) \to \mathcal{D}'(S \mathbb{H}^2)$ is the pullback along the initial point map defined in  \eqref{eq:defB}. Then one has
\begin{equation} \label{eq:I_minus}
    \mathbf I_-(\phi) =  \big((\pi^*_\Gamma)^{-1} \circ \mathcal{Q}_{s_0-1,-}\big) (T_{\phi}),
\end{equation}
where $T_{\phi}\in \D'(\partial_\infty \mathbb{H}^2)$ is the Helgason boundary value of $\phi \in \mathrm{Res}^{1}_{\Laplace}(s_0)$ introduced in \eqref{eq:Helgbval} and $\pi^*_\Gamma$ is the isomorphism from \eqref{eq:isomorphism_pi_Gamma}.

\subsubsection{Quantum-classical correspondence for co-resonant states}\label{sec:cores}
Consider the antipodal involution $\iota:S\mathbf X_\Gamma\to S\mathbf X_\Gamma$, $\xi\mapsto -\xi$. By pullback, it acts on smooth functions and distributions and satisfies $\iota^\ast X = -X \iota^\ast$, which implies that $\iota(\Upsilon_+)=\Upsilon_-$. Moreover, the adjoint of the derivative of $\iota$ interchanges $E_+^\ast$ and $E_-^\ast$ and we have $\iota^\ast U_+ = U_- \iota^\ast$. In view of Lemma~\ref{lem:ResResetoile} this implies that for all $j\in \N$ and $\lambda_0\in \C$ we have $\iota^\ast \mathrm{Res}^j_{X}(\lambda_0) = \mathrm{Res}^j_{X^\ast}(\overline\lambda_0)$ and
\bq
\iota^\ast (\mathrm{Res}^j_{X}(\lambda_0)\cap \ker U_-) = \mathrm{Res}^j_{X^\ast}(\overline\lambda_0)\cap \ker U_+.\label{eq:ResjResj}
\eq
On the other hand, since the natural pushforward $\pi_\ast:\D'(S\mathbf{X}_\Gamma)\to \D'(\mathbf{X}_\Gamma)$ is given on smooth functions by integration over the fiber and this determines it uniquely, one has
\bq
\pi_\ast\circ \iota^\ast = \pi_\ast. \label{eq:pushfowardiota}
\eq
We also have $\iota^\ast\Phi_+=\Phi_-$ and $\iota^\ast\circ B^*_+=B^*_-$, where $\Phi_\pm$ was defined in \eqref{eq:defPhipm} and $B^*_\pm: \mathcal{D}'(\partial_\infty \mathbb{H}^2) \to \mathcal{D}'(S \mathbb{H}^2)$ is the pullback along the initial/end point map defined in  \eqref{eq:defB}. Combining this with \eqref{eq:ResjResj} and \eqref{eq:pushfowardiota}, Proposition \ref{prop:classicalquantum_corresp} implies the following quantum-classical corresponence for co-resonant states:
\begin{prop}[Dual version of {Proposition \ref{prop:classicalquantum_corresp}}] \label{prop:dual}
For every classical first band resonance $\lambda_0\in \C\setminus (-\frac{1}{2}-\frac{1}{2}\N_0)$ the natural pushforward $\pi_\ast:\D'(S\mathbf{X}_\Gamma)\to \D'(\mathbf{X}_\Gamma)$ restricts for each $j\in \N$ to a linear isomorphism of finite-dimensional vector spaces 
    \[
    \pi_\ast: \mathrm{Res}^j_{X^\ast}(\lambda_0)\cap \ker U_+\stackrel{\cong}{\longrightarrow}\mathrm{Res}^{j}_{\Laplace}(\overline\lambda_0+1).
    \]
    For $j=1$, the inverse $\mathbf I_+\coloneq\pi_\ast^{-1}:\mathrm{Res}^{1}_{\Laplace}(s_0)\to \mathrm{Res}^1_{X^\ast}(\bar s_0-1)\cap \ker U_+$ is given by 
\begin{equation} \label{eq:I_plus}
    \mathbf I_+(\phi) =  \big((\pi^*_\Gamma)^{-1} \circ \mathcal{Q}_{\bar s_0-1,+}\big) (\overline T_{\phi}),
\end{equation}
where we put $\mathcal{Q}_{\lambda,+}(u)\coloneq\Phi_+^{\lambda}B_+^*(u)$ for $u \in \D'(\partial_\infty \mathbb{H}^2)$ and $\lambda \in \C$,  $T_{\phi}\in \D'(\partial_\infty \mathbb{H}^2)$ is the Helgason boundary value of $\phi \in \mathrm{Res}^{1}_{\Laplace}(s_0)$ introduced in \eqref{eq:Helgbval}, and $\pi^*_\Gamma$ is as in \eqref{eq:isomorphism_pi_Gamma}.
\end{prop}
Finally, let us point out an important issue that is often a source of confusion:
\begin{rem}[Sesquilinear vs.\ bilinear pairings]
    In microlocal analysis there is always the question whether to define pairings in a bilinear or a sesquilinear way (e.g.\ using an $L^2$-product).   
    Since the Wigner distributions are traditionally defined using an $L^2$-pairing, we use the sesquilinear convention for pairings throughout this paper. This is the reason why complex conjugates appear in Equations \eqref{eq:vvetoile} and  \eqref{eq:RanRan} and in Proposition \ref{prop:dual}.
\end{rem}
}

%-------------------------------%%--------------------------------------%
%\section{Patterson-Sullivan and Wigner distributions}
\section{Wigner and Patterson-Sullivan distributions}
\label{sect:relation_W_PS}
%-------------------------------%%--------------------------------------%
In this section, we describe the Wigner and Patterson-Sullivan distributions on $S\mathbf{X}_\Gamma$ and study their asymptotics, to ultimately show that they are equivalent (Section \ref{sect:asymptotic_equiv}).
%-------------------------------%%--------------------------------------%

\subsection{Wigner distributions}
\label{sect:Wigner}
%---------------------------------------------------------------------

Wigner distributions are a way to study  \emph{microlocal} properties of quantum resonant states $\phi\in C^\infty(\mathbf{X}_\Gamma)$ such as their oscillations in different directions in the phase space $T^\ast \mathbf X_\Gamma$ expressed by their distribution in the sphere bundle $S\mathbf X_\Gamma$ (which we identify with the co-sphere bundle using the metric). While these are intrinsic properties of the resonant states, the Wigner distributions associated to them are defined with respect to a choice of \emph{quantization map}, that is, a pseudo-differential operator calculus.  In order to define Wigner distributions, we quantize compactly supported smooth functions $u\in \CT(S \mathbf X_\Gamma)$ on the sphere bundle of the convex-cocompact hyperbolic surface $\mathbf{X}_\Gamma$. Such functions are very tame \emph{symbols} in the microlocal jargon\footnote{Note that, depending on the formalism, one regards the symbol functions $u$ as being smoothly extended to functions in $\CT(T^\ast \mathbf X_\Gamma)$ by identifying $S \mathbf X_\Gamma$ with the co-sphere bundle in $T^\ast \mathbf X_\Gamma$ using the metric and multiplying $u$ by a cutoff function.} and for their quantization, there are no essential differences between the most common pseudodifferential calculi (see \cite[Sec.~6]{HHS12} for a brief overview). In our setting, the equivariant pseudodifferential calculus developed by Zelditch in \cite{Zelditch19860} suggests itself because it has been specially designed for the hyperbolic plane $\mathbb H^2$ and its quotients by subgroups of  $G=\PSL(2,\R)$, providing an efficient framework for computations. 
%The semiclassical version of this calculus, where the quantization map depends on an auxiliary parameter $h>0$ (c.f.\ \cite[Sec.~6]{HHS12}) 
It provides a quantization map
$$
\mathrm{Op}:
\begin{cases}
    C^\infty_c(T^\ast \mathbf X_\Gamma) &\rightarrow\quad \mathcal{B}(\mathcal D'(\mathbf{X}_\Gamma),\CT(\mathbf{X}_\Gamma))\\
    \;\;\;\;\;\; u &\mapsto\quad \mathrm{Op}(u)
\end{cases}
$$
associating to every compactly supported smooth symbol function $u\in \CT(T^\ast \mathbf X_\Gamma)$ a continuous operator $\mathrm{Op}(u):\mathcal D'(\mathbf{X}_\Gamma)\to \CT(\mathbf{X}_\Gamma)$. Note that the potential non-compactness of  $\mathbf{X}_\Gamma$ is unproblematic here since we are only quantizing symbols with compact support. Recalling that $\Cinft(\mathbf{X}_\Gamma)$ is embedded into $\mathcal D'(\mathbf{X}_\Gamma)$ using the sesquilinear pairing $\langle \cdot, \cdot \rangle_{L^2(\mathbf{X}_\Gamma)}$, each linear operator of the form $\mathrm{Op}(u)$ restricts to an operator
\[
\mathrm{Op}(u):\Cinft(\mathbf{X}_\Gamma)\to \CT(\mathbf{X}_\Gamma).
\] 
Now, very conveniently, this calculus has been adapted in \cite[Sec.~3.2]{AZ07} to obtain a direct quantization $\mathrm{Op}(u)$ of symbol functions $u\in \CT(S \mathbf X_\Gamma)$ on the (co-)sphere bundle without extending them to functions in $C^\infty_c(T^\ast \mathbf X_\Gamma)$ using a cutoff. We will only need to consider the action of $\mathrm{Op}(u)$ on a quantum resonant state $\phi\in \mathrm{Res}^{1}_{\Laplace}(s_0)$ for a quantum resonance $s_0\in \C$ with $s_0\not\in  -\frac{1}{2}-\frac{1}{2}\N_0$. By \cite[(3.15)]{AZ07}, this action can be defined as follows: Given $u\in \CT(S \mathbf X_\Gamma)$, define 
\bq \label{eq:tildea}
\tilde{u}\coloneq u \circ \pi_\Gamma \in \Cinft(S\mathbb H^2),
\eq
where we recall that %$F_+:S\mathbb H^2\stackrel{\cong}{\to} \mathbb H^2\times\partial_\infty \mathbb H^2$ is the end point trivialization from \eqref{eq:Fpm} and 
$\pi_\Gamma:S\mathbb H^2\to S \mathbf X_\Gamma$ is the canonical projection from \eqref{eq:pi_Gamma}. Then
\bq\begin{split}
    \mathrm{Op}(u) \phi(x) 
    &\coloneq \big((\widetilde{\pi}_\Gamma^{-1})^*\circ \mathcal P_{s_0-1}\big)\big( \tilde u\circ F_+^{-1}(x,\cdot)  T_{\phi}\big)\\
    &\;= \big((\widetilde{\pi}^{-1}_\Gamma)^*\circ\mathcal P_{s_0-1}\circ m_{\tilde u \circ F_+^{-1}(x,\cdot)} \circ\mathcal P_{s_0-1}^{-1}\circ \widetilde{\pi}^*_\Gamma\big)(\phi),\end{split}\label{eq:explicitOpa}
\eq
 where for $f\in \Cinft(\partial_\infty \mathbb H^2)$ we write $m_f:\D'(\partial_\infty \mathbb H^2)\to \D'(\partial_\infty \mathbb H^2)$ for  multiplication by $f$ and we use the isomorphism $\widetilde\pi_\Gamma^\ast: \Cinft( \mathbf X_\Gamma)\stackrel{\cong}{\longrightarrow} \Cinft(\mathbb H^2)^\Gamma$ introduced in \eqref{eq:isomorphism_pi_Gamma}.\footnote{The reason why \cite[(3.15)]{AZ07} looks simpler than Equation \eqref{eq:explicitOpa} at first glance is that we do not make any implicit identifications. Thus \eqref{eq:explicitOpa} unwinds the second identification explained in \cite[Sec.~2.1]{AZ07} and the identifications between elements on $\Gamma$-quotients and their $\Gamma$-invariant lifts made in \cite{AZ07}.} Note that $\tilde{u} \circ F_+^{-1} \in \Cinft(\HH^2\times\partial_\infty \mathbb \HH^2),$ where $F_+:S\mathbb H^2\stackrel{\cong}{\to} \mathbb H^2\times\partial_\infty \mathbb H^2$ is the end point trivialization from \eqref{eq:Fpm}.

\begin{definition}[{Wigner distributions, c.f. \cite[Eq.~(1.1), Sec.~3.2.]{AZ07}, \cite[Sec.~6]{HHS12}, \cite[Def.~A.1.]{GHWb}}] \label{defn:Wigner}
    Let  $s_0,s_0' \in \C$ be quantum resonances    and $\phi\in \mathrm{Res}^1_{\Laplace}(s_0)$, $\phi'\in \mathrm{Res}^1_{\Laplace}(s_0')$. 
    The \emph{Wigner distribution} $W_{\phi,\phi'}\in \D'(S\mathbf{X}_\Gamma)$ associated to $\phi,\phi'$ is defined by
    \begin{equation*}
         W_{\phi,\phi'}(u) \coloneq \langle \mathrm{Op}(u) \phi, \phi'\rangle_{L^2(\mathbf{X}_\Gamma)},\qquad u\in \CT(S\mathbf{X}_\Gamma).
    \end{equation*}
\end{definition}
Here the $L^2$-pairing is well-defined thanks to the compact support of the smooth function $\mathrm{Op}(u) \phi$.  

If $s_0,s_0' \in \C\setminus (-\frac{1}{2}-\frac{1}{2}\N_0)$, then by \eqref{eq:explicitOpa} and recalling the Definition \ref{def:Poisson} of the Poisson transform, we have the following explicit expression of $W_{\phi,\phi'}$ in terms of Helgason boundary values (cf.~\cite[proof of Lem.~4.1]{AZ07}, \cite[(6.54)]{Sch10}): For $\phi\in \mathrm{Res}^1_{\Laplace}(s_0)$,  $\phi'\in \mathrm{Res}^1_{\Laplace}(s_0')$, and $u\in \CT(S\mathbf{X}_\Gamma)$,
\bq
W_{\phi,\phi'}(u)=\widetilde W_{\phi,\phi'}(\chi \tilde{u}),\label{eq:Wexplicit00}
\eq
where $\chi\in C^\infty_c(S\mathbb{H}^2)$ is a smooth fundamental domain cutoff near $\pi_\Gamma^{-1}(\supp u)$ as in Definition \ref{def:smoothfundcutoff} and the $\Gamma$-invariant distribution $\widetilde W_{\phi,\phi'}\in \D'(S\mathbb H^2)^\Gamma$ is defined by
\begin{equation} \label{eq:explicit_W}
   \widetilde W_{\phi,\phi'}(f)
    := \int_{\partial_\infty \mathbb{H}^2\times \partial_\infty \mathbb{H}^2} \Big( \int_{\mathbb{H}^2} (f  \circ F^{-1}_+)(x,b) e^{s_0\eklm{x,b}+\bar s_0'\eklm{x,b'}} \, dx\Big) T_{\phi}(b) \overline T_{\phi'}(b'),
\end{equation}
with $T_{\phi},T_{\phi'}\in \D'(\partial_\infty \mathbb{H}^2)$  the Helgason boundary values defined in \eqref{eq:Helgbval}, and the integral $\int_{\partial_\infty \mathbb{H}^2\times \partial_\infty \mathbb{H}^2}$ understood in the distributional sense (i.e.,\ applying the distribution $T_{\phi}\otimes \overline T_{\phi'}$). Note that the choice of $\chi$ in \eqref{eq:Wexplicit00} is irrelevant by Lemma \ref{lem:indepchi}.

\begin{rem}[Scaling conventions] \label{rem:scaling}
As already addressed in Remark \ref{rem:conventions}, in the subject of harmonic analysis there is the notorious problem of different  conventions. This affects us as follows: In the works \cite{GHWa} 
and \cite{AZ07} the conventions are compatible and agree with ours. However, in \cite{GHWb} the inner product used on $\g$ and hence the Riemannian metric on $\mathbb H^2$ differs from ours by a factor of $2$, with the effect that if we denote the (non-negative) Laplacian used in \cite{GHWb} by ${\Laplace}_{\mathbb H^2}^{\mathrm{GHW21}}$, then we have ${\Laplace}_{\mathbb H^2}^{\mathrm{GHW21}}=2{\Laplace}_{\mathbb H^2}$, where the right-hand side is our Laplacian. If for $s\in \C$ we put $\mu\coloneq s\varrho\in \aL_\C ^\ast$, one finds
\begin{eqnarray*}
    \mathrm{Eig}_{{\Laplace}_{\mathbb H^2}^{\mathrm{GHW21}}}(\mu)\cap C^\infty_\mathrm{mod}(\mathbb H^2) 
    &=& \Big\{f \in C_\mathrm{mod}^\infty(\mathbb H^2)\,|\, \Big(2{\Laplace}_{\mathbb H^2}-\frac{1}{2}+\frac{s^2}{2}\Big)(f)=0\Big\} \\
    &=&\mathcal E_{(\frac{1}{2}+\frac{s}{2})(\frac{1}{2}-\frac{s}{2})}
    =\mathcal E_{s_0(1-s_0)},
\end{eqnarray*}
where for $s_0$ we get the two solutions $s_0=\frac{1\pm s}{2}$. Taking the solution with the $+$ sign, our conventions are compatible with those in \cite{GHWb}.
\end{rem}

\subsubsection{Asymptotic parameter, Radon transform, and intertwining operator}

To systematically study the distribution $\widetilde W_{\phi,\phi'}\in \D'(S\mathbb H^2)^\Gamma$ defined in \eqref{eq:explicit_W}, identify $G=S \mathbb H^2$ by \eqref{eq:GSH2diffeo}, let $f\in \CT(G)$  and let $J_{s_0,s'_0}(f)\in \Cinft(\partial_\infty \mathbb{H}^2\times \partial_\infty \mathbb{H}^2)$ be defined by the inner integral over $\mathbb{H}^2$:
\bq
J_{s_0,s'_0}(f)(b,b') \coloneq\int_{\mathbb{H}^2} (f \circ F^{-1}_+)(x,b) e^{s_0\eklm{x,b}+\bar s_0'\eklm{x,b'}} \, dx.\label{eq:Jintegral}
\eq
Before we start rewriting it, we note that by writing $s_0=q+i r$, $s_0'=q'+i r'$ and assuming that $r\neq r'$,  the above integral acquires the shape of an oscillatory integral 
\[
J_{s_0,s'_0}(f)(b,b') =\int_{\mathbb{H}^2}  e^{i\frac{r-r'}{2}\Psi_{b,b',r,r'}(x)}  f_{b,b',q,q'}(x)\, dx
\]
with the asymptotic parameter $\frac{r-r'}{2}$, the phase function $\Psi_{b,b',r,r'}\in\Cinft(\mathbb H^2)$ given by
\begin{equation*}
    \Psi_{b,b',r,r'}(x)=\frac{2r}{r-r'}\eklm{x,b}-\frac{2r'}{r-r'}\eklm{x,b'},
\end{equation*}
and the amplitude $f_{b,b',q,q'}\in \CT(\mathbb H^2)$ given by
\begin{equation*}
    f_{b,b',q,q'}(x)\coloneq(f\circ F^{-1}_+)(x,b)e^{q\eklm{x,b}+q'\eklm{x,b'}}.
\end{equation*}
In order to compute an asymptotic expansion of  $J_{s_0,s'_0}(f)(b,b')$ using the Theorem of Stationary Phase \cite[Thm.~7.7.5]{Hor90}, one needs to find the set of critical points of the phase function $\Psi_{b,b',r,r'}$, i.e., the vanishing locus of the differential $d\Psi_{b,b',r,r'}$. This set has been computed in \cite[Lem. 5.4]{HHS12}, which in our rank one situation says the following (the parameters $\nu,\nu'$ in \cite[(5.1)]{HHS12} are given in our notation by $\nu=\frac{4r}{r-r'}\varrho$ and $\nu'=-\frac{4r'}{r-r'}\varrho$):
\begin{lem}[{\cite[Lem. 5.4]{HHS12}}]If $r\neq r'$, then one has $d\Psi_{b,b',r,r'}=0$ if, and only if, $r=-r'$, $(b,b')=(B_+(g),B_-(g))$, and $x\in gAK$ for some $g\in G$. 
\end{lem}
In view of this result, we will assume from now on that $r=-r'>0$, so that the asymptotic parameter is just $r$ and the phase function reduces to the $r$-independent function\label{nuHHS}\footnote{In the notation of \cite[(5.1)]{HHS12} and  \cite[Sec.~5]{HHS12}, this means that $\nu=\nu'=2\varrho$ and $h=\frac{1}{r}$.} $\Psi_{b,b'}(x)\coloneq\eklm{x,b}+\eklm{x,b'}$ which satisfies
\bq
d\Psi_{b,b'}(x)=0 \iff \exists\;g\in G: (b,b')=(B_+(g),B_-(g)),\; x\in gAK.\label{eq:dPsi0}
\eq
Note that geometrically the right-hand side of \eqref{eq:dPsi0} means that the point $x$ lies on a geodesic with endpoints $b$ and $b'$ in $\partial_\infty \mathbb{H}^2$.  
To cut out the vanishing locus of $d\Psi_{b,b'}$ we use another lemma from \cite{HHS12}, which is based on the geometric fact that the endpoints of geodesics in $\mathbb{H}^2$ passing through a given compact set are contained in a compact set disjoint from the diagonal in $\partial_\infty \mathbb{H}^2\times \partial_\infty \mathbb{H}^2$ (see Figure \ref{fig:supports}).

\begin{lem}[{Compare \cite[Lem. 5.7]{HHS12}\footnote{Here we corrected a small error in \cite[Lem. 5.7]{HHS12}: in the second line (and only in that line) of \cite[Lem. 5.7]{HHS12}, $1-\beta$ needs to be replaced by $\beta$.}}] \label{lem:beta}
Let $\mathscr C\subset G$ be a compact set. Then there is a function $\beta_{\mathscr C}\in \CT((\partial_\infty \mathbb H^2)^{(2)})\subset\CT(\partial_\infty \mathbb H^2\times\partial_\infty \mathbb H^2)$ such that for all $g\in G$ we have:
\bq
gAK\cap {\mathscr C}K\neq \emptyset \implies \beta_{\mathscr C}(B_+(g),B_-(g))=1.\label{eq:gAKcapCK}
\eq  
Moreover, if $\tilde \beta_{\mathscr C}\in \Cinft(G)^A$ is the $A$-invariant lift of $\beta_{\mathscr C}\circ \psi\in \CT(G/A)$ and ${\mathscr C}_A\subset A$ is a compact set, then $K{\mathscr C}_AN_+\cap \supp \tilde \beta_{\mathscr C}$ is compact.
\end{lem}
\begin{figure}
\newcommand{\hgline}[2]{
\pgfmathsetmacro{\thetaone}{#1}
\pgfmathsetmacro{\thetatwo}{#2}
\pgfmathsetmacro{\theta}{(\thetaone+\thetatwo)/2}
\pgfmathsetmacro{\phi}{abs(\thetaone-\thetatwo)/2}
\pgfmathsetmacro{\close}{less(abs(\phi-90),0.0001)}
\ifdim \close pt = 1pt
    \draw[blue] (\theta+180:1) -- (\theta:1);
\else
    \pgfmathsetmacro{\R}{3*tan(\phi)}
    \pgfmathsetmacro{\distance}{sqrt(9+\R^2)}
    \draw[blue] (\theta:\distance) circle (\R);
\fi
}
    \centering
    \begin{tikzpicture}\begin{scope}
        \draw[pattern={Lines[angle=45,distance={3pt/sqrt(2)}]},pattern color=lightgray] %[fill=gray!10!white] 
        (0,0) circle (3);
        \clip (0,0) circle (3);
        \hgline{-20}{-58}
        \draw[fill=black!50!white] %[pattern={Lines[angle=45,distance={3pt/sqrt(2)}]},pattern color=black]
        (1,-0.8) -- (0.7,-0.6)..controls(0.5,-1)..(1.3,-1.4)
         -- (1.7,-1.3) -- cycle;
        %\path [fill=blue, fill opacity = 0.3] (1.06,-1.06) circle (2.5);
        %\path [fill=white, fill opacity = 1] (1.265,-1.265) circle (2.2); 
       \end{scope}
   \node [blue] at (3.1,-1.1) {$b$};
   \draw [color=blue,fill=blue] (2.8,-1.05) circle (1.5pt);
   \node [blue] at (1.7,-2.9) {$b'$};
    \draw [color=blue,fill=blue] (1.6,-2.56) circle (1.5pt);
    \node [black] at (0,1.5) {$\mathbb H^2=G/K$};
    \node [black] at (-4.1,-0.5) {$\partial_\infty\mathbb H^2=\mathbb{S}^1$}; 
  \node [black] at (0.8,-0.3) {${\mathscr C}K$}; 
\end{tikzpicture}
    \caption{The projection ${\mathscr C}K\in \mathbb H^2=G/K$ of a compact set ${\mathscr C}\subset G$ and a geodesic in $\mathbb H^2$ passing through ${\mathscr C}K$ with its endpoints $b,b'\in \partial_\infty\mathbb H^2=\mathbb{S}^1$. This particular geodesic realizes the minimal distance in $\mathbb{S}^1$ of endpoints of geodesics passing through ${\mathscr C}K$ -- for any other such geodesic, its endpoints will be at least as far apart as $b$ and $b'$. In particular, the pairs of such endpoints cannot be arbitrarily close to the diagonal in $\partial_\infty\mathbb H^2\times \partial_\infty\mathbb H^2$.  } 
    \label{fig:supports}
\end{figure}
From now on, we fix some arbitrary $C>0$ and consider only parameters $s_0,s_0'\in \C\setminus (-\frac{1}{2}-\frac{1}{2}\N_0)$ of the form 
\bq
s_0=q+i r,\qquad s_0'=q'-i r,\qquad \frac{1}{2}-C\leq q, q'\leq \frac{1}{2},\;r\in \R_{>0}.\label{eq:resonanesform}
\eq
Using Lemma \ref{lem:beta} in case that $\supp f\subset \mathscr C$, for some compact set $\mathscr C\subset G$, by inserting $1=\beta_{\mathscr C}+(1-\beta_{\mathscr C})$ in front of the integral in \eqref{eq:Jintegral}, the description \eqref{eq:dPsi0} of the stationary points and the non-stationary phase principle {\cite[Thm.\ 7.7.1]{Hor90}} give us
\begin{multline}
    \label{eq:intJwithremainder}
J_{s_0,s'_0}(f)(b,b') =\underbrace{\beta_{\mathscr C}(b,b')\int_{\mathbb{H}^2} (f\circ F^{-1}_+)(x,b) e^{s_0\eklm{x,b}+\bar s_0'\eklm{x,b'}} \, dx}_{=: J^\mathrm{prin}_{s_0,s'_0}(f)(b,b')}\\+\; {\mathrm{Rem}_{s_0,s_0'}(f)(b,b'),}
\end{multline}
where the principal term $J^\mathrm{prin}_{s_0,s'_0}(f)$ is now supported in a compact subset of  $(\partial_\infty \mathbb H^2)^{(2)}\subset\partial_\infty \mathbb H^2\times\partial_\infty \mathbb H^2$ independent of $s_0,s_0'$  and the remainder $\mathrm{Rem}_{s_0,s_0'}(f)\in \Cinft(\partial_\infty \mathbb H^2\times\partial_\infty \mathbb H^2)$ satisfies
\bq
\mathrm{Rem}_{s_0,s_0'}(f)=\O_{\Cinft(\partial_\infty \mathbb H^2\times\partial_\infty \mathbb H^2)}(r^{-\infty})\quad \text{as }r\to +\infty.\label{eq:remainder}
\eq
By this notation we mean that for every continuous seminorm $\mathsf p$ on the Fréchet space $\Cinft(\partial_\infty \mathbb H^2\times\partial_\infty \mathbb H^2)$, for example any Sobolev norm, one has $\mathsf p(\mathrm{Rem}_{s_0,s_0'}(f))=\O(r^{-\infty})$.

By the general integration formula \eqref{eq:dxdadn} and the $G$-invariance of the measure $dx$ on $\mathbb{H}^2$, we can rewrite $J^\mathrm{prin}_{s_0,s'_0}(f)(b,b')$ as
\bq
    \label{eq:intw_AN}
J^\mathrm{prin}_{s_0,s'_0}(f)(b,b') =\beta_{\mathscr C}(b,b')\int_A\int_{N_+}  (f\circ F^{-1}_+)(gan\cdot o,b) e^{s_0\eklm{gan\cdot o,b}+\bar s_0'\eklm{gan\cdot o,b'}} \, d n \d a,
\eq
where $o=K\in G/K=\mathbb{H}^2$ is the canonical base point and $g\in G$ is arbitrary. 

As in \cite{AZ07} and \cite[Sec.~5.2]{HHS12}, we proceed by writing the integrations over $A$ and $N_+$ separately as two operators called the \emph{(weighted) Radon transform} and the \emph{intertwining operator} with respect to the parameters $s_0,s_0'$, respectively. To begin, we compose $J_{s_0,s'_0}$ with the diffeomorphism $\psi:G/A \stackrel{\cong}{\to} (\partial_\infty \mathbb H^2)^{(2)}\subset \partial_\infty \mathbb{H}^2\times \partial_\infty \mathbb{H}^2$ introduced in \eqref{eq:psi}. Given $(b,b')\in (\partial_\infty \mathbb H^2)^{(2)}$ and $g\in G$ such that $\psi(gA)=(b,b')$, write $n_g\coloneq\exp(\mathscr{N}(g))$, $a_g\coloneq\exp(\mathscr{A}(g))$, so that $g=\kappa(g)a_gn_g$. Then 
$$(b,b')=\psi(gA)\equiv(B_+(g),B_-(g))=(\kappa(g),\kappa(gw_0))=(g\cdot e,g w_0\cdot e).$$
We now compute using \eqref{eq:horocycle_bracket} for all $a\in A$, $n\in N_+$:
\begin{equation*}
\begin{split} \label{eq:brackets_b}
        \eklm{gan\cdot o,b}
        =\eklm{gan\cdot o, \kappa(g)}
        =-2\varrho(\mathscr A(n^{-1}a^{-1}n_g^{-1}a_g^{-1}))
        =2\varrho(\mathscr A(ga)), 
\end{split}
\end{equation*}
since $A$ normalizes $N_+$ and $\mathscr A(kg'n)=\mathscr A(g')$ for all $k\in K$, $g'\in G$, $n\in N_+$. The analogous computation for $\eklm{gan\cdot o,b'}$ is more involved due to the appearance of $w_0$. It has been carried out in \cite[p.~629]{HHS12} with the result (taking into account Remark \ref{rem:conventions})
\begin{equation*}
\label{eq:brackets_b'}
        \eklm{gan\cdot o,b'}    =-2\varrho(\mathscr A(n^{-1}w_0))+2\varrho(\mathscr A(gaw_0)).
\end{equation*}
Recalling the definition of the $G$-equivariant diffeomorphism $F_+:G=S\mathbb H^2\stackrel{\cong}{\to} \mathbb H^2\times\partial_\infty \mathbb H^2$ from \eqref{eq:Fpm}, we can easily decompose the function $f \circ F^{-1}_+\in \CT(\mathbb{H}^2 \times \partial_\infty \mathbb{H}^2)$ by preserving the canonical base points as follows:
\bq \label{eq:tilda_a_ref}
  (f\circ F^{-1}_+)(gan \cdot o, b)
  = (f\circ F^{-1}_+)(gan \cdot o, g \cdot e)
  = (f\circ F^{-1}_+)(gan \cdot o, gan \cdot e)
  =  f(gan).
\eq
If $\mathscr C\subset G$ is compact and $f\in \CT(G)$ is supported in $\mathscr C$, then by Lemma \ref{lem:beta} the smooth function $G/A\to \C$ given by
\bq
\widetilde {\mathcal{R}}_{s_0,s'_0}(f)(gA):=\int_{A}e^{s_02\varrho(\mathscr{A}(ga))+\bar s'_02\varrho(\mathscr{A}(ga \omega_0))} f(ga) \; da\label{eq:Rtilde}
\eq
is supported in the compact set $\psi^{-1}(\mathrm{supp}\, \beta_{\mathscr CK})\subset G/A$, where  $\beta_{\mathscr CK}$ is a chosen function as in Lemma \ref{lem:beta} (cf.~\cite[p.~91]{GAGA} and Figure~\ref{fig:supports}). In particular, the pre-composition with  $\psi^{-1}:(\partial_\infty \mathbb{H}^2)^{(2)}\to G/A$ extends smoothly by zero to $\partial_\infty \mathbb{H}^2\times \partial_\infty \mathbb{H}^2$. 

This makes it possible to define the \emph{(weighted) Radon transform} as a continuous operator $\mathcal{R}_{s_0,s'_0}:\CT(G)\to C^\infty(\partial_\infty \mathbb{H}^2\times \partial_\infty \mathbb{H}^2)$
by
\bq \label{eq:Radontransform}
\mathcal{R}_{s_0,s'_0}(f)(b,b')\coloneq\begin{cases}\widetilde {\mathcal{R}}_{s_0,s'_0}(f)(gA),\quad &b\neq b',\;(b,b')=\psi(gA),\\
    0, &b=b'.
\end{cases}
\eq
\subsubsection{Sobolev spaces and operator norm estimates}\label{sec:Sobolev}
In order to describe the continuity properties of the Radon transform and further operators, let us briefly recall the basic theory of Sobolev spaces; see e.g.\ \cite[Chapter 10]{Hintz} for details. 
On any smooth manifold $\mathcal M$ equipped with a smooth volume density and for every $k\in \R$ one can define a Sobolev space $H^k_\mathrm{loc}(\mathcal M)\subset \D'(\mathcal M)$ which for every compact subset $\mathcal K\subset \mathcal M$ induces a subspace topology on the space
\[
    H^k_\mathcal{K}(\mathcal M):=\{u\in H^k_\mathrm{loc}(\mathcal M)\,|\,\mathrm{supp}\,u\subset \mathcal K\}
\]
that is compatible with a Hilbert space structure on $H^k_\mathcal{K}(\mathcal M)$, and for a given $\mathcal K$ all possible choices of such Hilbert Sobolev norms on $H^k_\mathcal{K}(\mathcal M)$ are equivalent. 

We will only need to consider the case where $k\in \N_0$ is a non-negative integer.
In this case, a convenient way to choose such a norm without invoking coordinates is to fix a family of smooth vector fields $X_1,\ldots,X_{l}$ ($l\in \N$) on $\mathcal M$ large enough that at every point $p\in \mathcal{K}$ the family spans the whole tangent space $T_p\mathcal M$, and put for $f\in \Cinft(\mathcal M)$
\bq \label{eq:Sobolevnormdef}
    \norm{f}_{H^k_\mathcal{K}(\mathcal M)}^2:=\max_{\alpha\in \N_0^{l}:|\alpha|\leq k}\int_{\mathcal K}|X^{\alpha_1}_1\cdots X_l^{\alpha_l}f(x)|^2 \d x, 
\eq
where $dx$ is the fixed volume density on $\M$ (whose choice does not matter on the compact set $\mathcal K$).  
We will apply this to the non-compact manifolds $\mathcal M=G\equiv S\HH^2$, $\mathcal M=G/A$, $\mathcal M=S\mathbf{X}_\Gamma$ and also to the compact manifold $\mathcal M=\partial_\infty \mathbb{H}^2\times \partial_\infty \mathbb{H}^2$, where we had already fixed Sobolev norms earlier (in general, if $\mathcal M$ is compact, one has  
$H^k_\mathrm{loc}(\mathcal M)=H^k_\mathcal{M}(\mathcal M)=:H^k(\mathcal M)$ and the Sobolev theory is much more well-known, which is why previously in this paper we used it without further comments).
Thus, in the following, we also fix choices of Sobolev spaces $H^k_\mathrm{loc}(G)$ and $H^k_\mathrm{loc}(G/A)$.  

The next lemma will be used later to estimate remainder terms.
It is analogous to \cite[Prop.~4.7]{HHS12} (see also \cite[Prop.\ 3.6, Eq.~(3.14)]{AZ07}). 
\begin{lem}\label{lem:radoncontinuity}
    For every compact set $\mathscr C\subset G$ and $k\in \N_0$, there is a constant $c_{\mathscr C,C,k}>0$, depending also on the constant $C$ fixed before \eqref{eq:resonanesform}, such that for all $s_0,s_0'$ as in  \eqref{eq:resonanesform}, all $\chi\in \CT(G)$ with $\supp \chi\subset \mathscr C$, and all $f\in \CT(G)$ one has
    \[
    \big\|\mathcal{R}_{s_0,s'_0}(\chi f)\big\|_{H^k(\partial_\infty \mathbb{H}^2\times \partial_\infty \mathbb{H}^2)}\leq c_{\mathscr C,C,k}(1+r)^{k}\|\chi f\|_{H^k_{\mathscr C}(G)}.
    \] 
\end{lem}
\begin{proof}
Let $\mathscr {K}:=\psi^{-1}(\mathrm{supp}\, \beta_{\mathscr CK})\subset G/A$.
Then the Sobolev Hilbert norm $\norm{\cdot}_{H^k(\partial_\infty \mathbb{H}^2\times \partial_\infty \mathbb{H}^2)}$ restricted to distributions supported in $\mathrm{supp}\, \beta_{\mathscr CK}\subset \partial_\infty \mathbb{H}^2\times \partial_\infty \mathbb{H}^2$ pulls back to a Sobolev Hilbert norm $\norm{\cdot}_{H^k_{\mathscr K}(G/A)}$ for distributions on $G/A$ supported in $\mathscr K$, and by definition of $\mathcal{R}_{s_0,s'_0}$ it suffices to estimate  $\Vert\widetilde {\mathcal{R}}_{s_0,s'_0}(\chi f)\Vert_{H^k_{\mathscr K}(G/A)}$.\\
Now, since the canonical projection $\pi_{G/A}:G\mapsto G/A$ is a fiber bundle, there exists a compact set $\mathscr C'\subset G$ such that $\pi_{G/A}(\mathscr C')=\mathscr K$.
Let $\mathscr C_0\subset G$ be a compact set containing $\mathscr C'$ in its interior. Since any two Sobolev norms on the compact set $\mathscr C'\subset G$ are equivalent, there is a cutoff function $\tau\in \CT(G)$, supported in $\mathscr C_0$ and equal to $1$ on $\mathscr C'$, and a   constant $c_{\mathscr C,k}>0$ such that the pullback $\pi_{G/A}^\ast: \CT(G/A)\to \Cinft(G)$ satisfies
\[
    \Vert h\Vert_{H^k_{\mathscr K}(G/A)}\leq c_{\mathscr C,k} \Vert\tau\pi_{G/A}^\ast(h)\Vert_{H^k_{\mathscr C_0}(G)}\quad \forall\;h\in \CT(G/A),\;\mathrm{supp}\,h\subset \mathscr K.
\]
Therefore, by \eqref{eq:Rtilde}, it suffices to estimate the linear map
\begin{align*}
    P_{\chi,s_0,s'_0}:\CT(G)&\to \CT(G)\\
    f&\mapsto \Big(\tau(g)\int_{A}e^{s_02\varrho(\mathscr{A}(ga))+\bar s'_02\varrho(\mathscr{A}(ga \omega_0))} \chi f(ga) \; da\Big)
\end{align*}
as an operator $H^k_{\mathscr C}(G)\to H^k_{\mathscr C_0}(G)$.

Given $f\in \CT(G)$, by \eqref{eq:Sobolevnormdef} the squared Sobolev norm $\Vert P_{\chi,s_0,s'_0}(f)\Vert_{H^k_{\mathscr C_0}}^2$ is given by the maximum of a finite number of expressions of the form
\[
    \int_{\mathscr C_0}|X_{1}\cdots X_{N}P_{\chi,s_0,s'_0}(f)(g)|^2\d g,
\]
where $0\leq N\leq k$ and $X_1,\ldots,X_N$ are vector fields on $G$.  Now, since $s_0,s_0'$ only appear in the exponential factor in the integral defining $P_{\chi,s_0,s'_0}(f)$, it follows that
\bqn
    \int_{\mathscr C_0}|X_1\cdots X_N P_{\chi,s_0,s'_0}(f)(g)|^2\d g\leq (1+r)^{2N} e^{bC}\sum_{j=0}^N C_{\tau,j,\mathscr C}\int_{\mathscr C}|X'_1\cdots X'_j (\chi f)(g)|^2\d g
\eqn
for constants $C_{\tau,j,\mathscr C}$ involving supremum norms of derivatives of $\tau$ and a constant $b>0$ and vector fields $X'_1,\ldots,X'_N$ that only depend on the vector fields $X_1,\ldots,X_N$ and the compact set $\mathscr C$ but neither on $C$, $\chi$ nor $s_0,s_0'$. 
Here the important point is that the set
\[
    \{(g,a)\in \mathscr C_0\times A\,|\,ga\in \mathscr C\}
\]
is compact, which means that the $A$-integral is over a compact domain. 

On the right-hand side of the above estimate we recognize the terms making up the Sobolev norm $\norm{\chi f}_{H^k_{\mathscr C}}$, which proves the claim.
\end{proof}
Next, we consider the \emph{intertwining operator} 
\bq
  \mathcal I_{s'_0}:\CT(G)\to \Cinft(G),\qquad \mathcal I_{s'_0}(f)(g)\coloneq\int_{N_+}  e^{-\bar s'_02\varrho(\mathscr{A}(n^{-1} \omega_0))}f(gn)\; dn.\label{eq:intw}
\eq
Let $\mathscr C\subset G$ be a compact set and let $f\in \CT(G)$ be supported in $\mathscr C$.
If we take a function $\tilde{\beta}_{\mathscr C}\in \Cinft(G)^A$  as in Lemma \ref{lem:beta}, then there is a compact subset $\mathscr K_{\mathscr C}\subset G$, independent of $f$ and $s'_0$, such that
\bq \label{eq:betaIcptsupp}
    \supp\big(\tilde{\beta}_{\mathscr C} \mathcal I_{s'_0}(f)\big)\subset \mathscr K_{\mathscr C}.
\eq
Indeed, since $G=KAN_+$ and $\mathscr C$ is compact, we have $\mathscr C\subset KC_AC_{N_+}$ for compact sets $C_A\subset A$, $C_{N_+}\subset N_+$, so the function $\mathcal I_{s'_0}(f)$  is supported in $KC_AN_+$ and by Lemma \ref{lem:beta} the set $\mathscr K_{\mathscr C}:=KC_AN_+\cap\mathrm{supp}\,\tilde{\beta}_{\mathscr C}$ has compact support.
This set contains $\supp \tilde{\beta}_{\mathscr C} \mathcal I_{s'_0}(f)$. Like the Radon transform, the intertwining operator is continuous on Sobolev spaces:
\begin{lem}\label{lem:intertwiningcontinuity}
    Let $\mathscr C\subset G$ be a compact set,  $C>0$ be the constant fixed before \eqref{eq:resonanesform}, and $k\in \N_0$. 
    Then there is a constant $c_{\mathscr C,C,k}>0$, depending on the choice of $\tilde{\beta}_{\mathscr C}$, such that for all $s_0'$ as in  \eqref{eq:resonanesform} and all $f\in \CT(G)$ supported in $\mathscr C$ one has
    \[
        \big\|\tilde{\beta}_{\mathscr C} \mathcal I_{s'_0}(f)\big\|_{H^k_{\mathscr K_{\mathscr C}}(G)}\leq c_{\mathscr C,C,k}\| f\|_{H^k_{\mathscr C}(G)}.
    \] 
\end{lem}
\begin{proof}
    The arguments are completely analogous to those in the proof of Lemma \ref{lem:radoncontinuity} with only minor differences:
    Here no $g$ appears in the exponent in the intertwining operator \eqref{eq:intw}, which means that the vector fields on $G$ defining the Sobolev norm do not act on the exponential factor and therefore produce no powers of $\bar s_0'$, which would be estimated by powers of $(1+r)$.
    Similarly as for the $A$-integral in the Radon transform, the $N_+$-integral involved in the operator $\tilde{\beta}_{\mathscr C} \mathcal I_{s'_0}$ is only over a compact domain thanks to the multiplication with $\tilde{\beta}_{\mathscr C} \mathcal I_{s'_0}$.
    Therefore the situation is essentially the same, with an analogous conclusion.
\end{proof}

Recall that  $J^\mathrm{prin}_{s_0,s'_0}(f)(b,b')=0$ if $b=b'$. For $b\neq b'$, we follow \cite[p.~632]{HHS12} to express $J^\mathrm{prin}_{s_0,s'_0}(f)(b,b')$ using Fubini's theorem as follows: 
If $\mathscr C\subset G$ is compact and $f\in \CT(G)$ is supported in $\mathscr C$, then
\begin{eqnarray} \label{eq:osci}
    J^\mathrm{prin}_{s_0,s'_0}(f)(b,b')    &\stackrel{\eqref{eq:intw_AN} - \eqref{eq:tilda_a_ref}}{=}& \beta_{\mathscr C}(b,b')\int_A e^{s_02\varrho(\mathscr{A}(ga))+\bar s'_0\varrho(\mathscr{A}(ga \omega_0))} \nonumber \\
    &&\qquad\qquad\qquad\qquad\int_{N_+}  f(gan) e^{-\bar s'_02\varrho(\mathscr{A}(n^{-1}w_0))} \; \d n \d a \nonumber \\ %\stackrel{\eqref{eq:intw_AN}, \eqref{eq:brackets_b}, \eqref{eq:brackets_b'} \& \eqref{eq:tilda_a_ref}}{=}
    &\stackrel{\eqref{eq:intw}}{=}& \int_{A} e^{s_02\varrho(\mathscr{A}(ga))+\bar s'_02\varrho(\mathscr{A}(ga \omega_0))} \tilde{\beta}_{\mathscr C}(ga)\mathcal{I}_{s'_0} (f)(ga) \; da  \nonumber \\  
    &\stackrel{\eqref{eq:Radontransform}}{= }& \mathcal{R}_{s_0,s'_0}(\tilde{\beta}_{\mathscr C}\mathcal{I}_{s'_0}(f))(g),
\end{eqnarray}
where $(b,b')=(B_+(g),B_-(g))=\psi(gA)$.
Finally, going all the way back to \eqref{eq:explicit_W} and applying the distribution $T_{\phi}\otimes \overline T_{\phi'}$ to the expression for $J^\mathrm{prin}_{s_0,s'_0}$ obtained in \eqref{eq:osci}, we get:
\begin{lem}\label{lem:Wignerleadingterm}
    Fix $C>0$, let $s_0,s_0' \in \C$ be two quantum resonances of the form $s_0=q+i r$, $s_0'=q'-i r$, where $\frac{1}{2}-C\leq q,q'\leq\frac{1}{2}$, $r\in \R_{>0}$,  
    and consider quantum resonant states $\phi\in \mathrm{Res}^1_{\Laplace}(s_0)$, $\phi'\in \mathrm{Res}^1_{\Laplace}(s_0')$. 
    Then, identifying $G=S\mathbb H^2$ via \eqref{eq:identSH}, the distribution $\widetilde W_{\phi,\phi'}\in \D'(S\mathbb H^2)^\Gamma$ defined in \eqref{eq:explicit_W} satisfies  for every compact set $\mathscr C\subset G$ and all $f\in \CT(G)$ supported in $\mathscr C$
    \bq
        \widetilde W_{\phi,\phi'}(f)
        = (T_{\phi} \otimes \overline T_{\phi'})\Big({\mathcal{R}_{s_0,s'_0}\big(\tilde{\beta}_{\mathscr C}\mathcal{I}_{s'_0}(f)\big)} + \mathrm{Rem}_{s_0,s_0'}(f)\Big),\label{eq:Wignerformula}
    \eq
    where $\mathrm{Rem}_{s_0,s_0'}(f)\in \Cinft(\partial_\infty \mathbb H^2\times\partial_\infty \mathbb H^2)$ is the $\O(r^{-\infty})$-remainder term from \eqref{eq:remainder}.  \qed
\end{lem}

%---------------------------------------------------------------------%
\subsection{Patterson-Sullivan distributions {and proof of Theorem \ref{thm:asymptotic_WPS}}} 
\label{sect:PS}
%---------------------------------------------------------------------%
As explained in the introduction, Patterson-Sullivan distributions were first introduced by Anantharaman and Zelditch \cite[Def.~3.3]{AZ07} in the setting of compact hyperbolic surfaces and generalized to compact higher rank locally symmetric spaces by Hansen, Hilgert and Schröder \cite{Sch10,HS09,HHS12}, who also introduced off-diagonal Patterson-Sullivan distributions.
In this paper we use the quantum-classical correspondence approach developed in \cite{GHWb}.

%---------------------------------------------------------------------%
\subsubsection{Description in terms of resonant and co-resonant states}
\label{sect:PS_Resonant}
%---------------------------------------------------------------------%
Guillar\-mou-Hilgert-Weich worked out in \cite[Thm.~5.2]{GHWb} that on compact rank one locally symmetric spaces Patterson-Sullivan distributions have a description as products of resonant and co-resonant states.
We will use this description in this paper to define the Patterson-Sullivan distributions in the convex-cocompact setting.

%Given a Laplace eigenfunction $\varphi \in C^\infty(S\mathbf{X}_\Gamma)$, we will prove that the Patterson-Sullivan distribution is then precisely given by the distributional product of the Ruelle resonant state and the co-resonant state.

\begin{definition}[Patterson-Sullivan distributions in terms of (co-)resonant states] \label{def:PS_Resonant}
    Given $s_0,s_0' \in \C\setminus (-\frac{1}{2}-\frac{1}{2}\N_0)$ and two quantum resonant states $\phi \in \mathrm{Res}^{1}_{\Laplace}(s_0), \phi' \in \mathrm{Res}^{1}_{\Laplace}(s_0')$, consider the classical resonant and co-resonant states 
    \begin{align*}
         v_{\phi}&\coloneq\mathbf I_-(\phi)\in \mathrm{Res}^1_{X}(s_0-1)\cap \ker U_-\subset \D'(S\mathbf{X}_\Gamma),\\
        v_{\phi'}^\ast&\coloneq\mathbf I_+(\phi')\in\mathrm{Res}^1_{X^\ast}(\bar s_0'-1)\cap \ker U_+\subset \D'(S\mathbf{X}_\Gamma),
    \end{align*}  
    with the maps $\mathbf I_\pm$ from \eqref{eq:I_minus} and \eqref{eq:I_plus}, respectively.
    Then we define the (off-diagonal) \emph{Patterson-Sullivan distribution} $ \mathrm{PS}_{\phi, \phi'} \in \D'(S\mathbf{X}_\Gamma)$  as the product
    \begin{equation} \label{eq:PS}
        \mathrm{PS}_{\phi,\phi'}\coloneq v_{\phi} \cdot \bar v_{\phi'}^\ast.
    \end{equation}
\end{definition} 
Here the product $v_{\phi} \cdot \bar v_{\phi'}^\ast$ is well-defined and when $s_0=\bar s_0'$ then it is $\varphi_t$-invariant, as already explained around \eqref{eq:distproductXzero}.
Note that when $\phi=\bar \phi'$ and $s_0=\bar s_0'$ we recover the usual ``diagonal'' Patterson-Sullivan distributions {studied} in \cite{AZ07}.

%---------------------------------------------------------------------%
\subsubsection{Description in terms of the (weighted) Radon transform}
\label{sect:PS_Radon}
%---------------------------------------------------------------------%
From the proof of \cite[Thm.~5.2]{GHWb} (with $s_0=\mu-\varrho$ and $s'_0=\mu'-\varrho$ in our setting, see Remark~\ref{rem:scaling}) we know that the Patterson-Sullivan distribution \eqref{eq:PS} can also be reformulated in terms of the (weighted) Radon transform as follows.
For $u\in C_c^\infty(S\mathbf{X}_\Gamma)$, we have  
\bq\label{eq:PS_Radon}
    \mathrm{PS}_{\phi, \phi'}(u) 
    =  (T_{\phi} \otimes \overline T_{\phi'})\big(\mathcal{R}_{s_0,s'_0}(\chi \tilde{u})\big), 
\eq
where $\mathcal{R}_{s_0,s'_0}$ is the Radon transform defined in \eqref{eq:Radontransform}, 
$T_{\phi},T_{\phi'}\in \D'(\partial_\infty \mathbb{H}^2)$ are the Helgason boundary values defined in \eqref{eq:Helgbval} and which comes from the explicit description of $\mathbf{I}_\pm$ in \eqref{eq:I_minus} and \eqref{eq:I_plus}, respectively{. 
Recall that $\chi \tilde{u}\in C^\infty_c(S\HH^2)$ is the product of the $\Gamma$-invariant lift $\tilde u\in \Cinft(S\mathbb H^2)^\Gamma$ of $u$ and a smooth fundamental domain cutoff $\chi\in \CT(S\mathbb H^2)$ near $\supp \tilde u$, as defined in Definition \ref{def:smoothfundcutoff}. 

\subsubsection{Proof of main results}
\label{sect:asymptotic_equiv}

We are finally in the position to prove Theorem \ref{thm:asymptotic_WPS}.
It will be deduced on p.~\pageref{proof:main}  from the following slightly more general result. 
 To state it, let us define that for $k\in \N_0$, a sequence  $v_j\in H^{-k}_{\mathrm{loc}}(S\mathbf{X}_\Gamma)$, and bounds $b_j>0$, the notation
\bq
v_j=\O_{H^{-k}_{\mathrm{loc}}(S\mathbf{X}_\Gamma)}(b_j)\label{eq:OHkdef}
\eq
means that for every compact set $\mathcal{C}\subset S\mathbf{X}_\Gamma$ there is a $C_{\mathcal C,k}>0$ and a $j_{\mathcal C}\in \N$ such that
\[
  \norm{v_j}_{H^{-k}_{\mathcal{C}}(S\mathbf{X}_\Gamma)}\leq C_{\mathcal C,k}\, b_j\quad \forall\; j\geq j_{\mathcal C}.
\]

\begin{thm}\label{thm:precise} 
    For $j\in \N$, let $s_j,s_j' \in \C\setminus (-\frac{1}{2}-\frac{1}{2}\N_0)$ be quantum resonances of the form $s_j=q_j+ir_j$, $s_j'=q_j' -ir_j$, where $r_j\to +\infty$ as $j\to \infty$ and $\frac{1}{2}-C\leq q_j,q'_j\leq\frac{1}{2}$ for some $C>0$.
    Let $\phi_j\in \mathrm{Res}^1_{\Laplace}(s_j)$ and $\phi_j'\in \mathrm{Res}^1_{\Laplace}(s_j')$ be quantum resonant states.\\
    Then there are sequences of operators $L_{M,s_j'},R_{s_j'}:\CT(S\mathbf{X}_\Gamma)\to \CT(S\mathbf{X}_\Gamma)$, $M,j\in \N$, such that for all $M\in \N$ and $k\geq k_C$ with $k_C>0$ as in Definition \ref{def:moderatenormalization} the following holds:

    One has $W_{\phi_j,\phi'_j}, \mathrm{PS}_{\phi_j,\phi_j'}\in H^{-k}_{\mathrm{loc}}(S\mathbf{X}_\Gamma)$, and in $H^{-k}_{\mathrm{loc}}(S\mathbf{X}_\Gamma)$ the following asymptotic relations hold (where we use the notation \eqref{eq:OHkdef}):
    \begin{align*}
        W_{\phi_j,\phi'_j} 
        &= \frac{e^{-i\frac{\pi}{4}}}{\sqrt{\pi} }\, r_j^{-1/2} {\mathrm{PS}}_{\phi_j,\phi_j'} \bigr(\cdot+r_j^{-1}R_{s_j'}\bigl)\\
        &\qquad\qquad+\O_{H^{-k}_{\mathrm{loc}}(S\mathbf{X}_\Gamma)}\Big(r_j^{-M}\Vert T_{\phi_j}\otimes \overline{T}_{\phi'_j}\Vert_{H^{-k}(\partial_\infty \mathbb{H}^2\times \partial_\infty \mathbb{H}^2)}\Big),\\
        W_{\phi_j,\phi'_j}\bigr(\cdot + r_j^{-1}L_{M,s_j'}\bigl) 
        &= \frac{e^{-i\frac{\pi}{4}}}{\sqrt{\pi} }\, r_j^{-1/2} {\mathrm{PS}}_{\phi_j,\phi_j'}\\
        &\qquad\qquad+\O_{H^{-k}_{\mathrm{loc}}(S\mathbf{X}_\Gamma)}\Big(r_j^{-M}\Vert T_{\phi_j}\otimes \overline{T}_{\phi'_j}\Vert_{H^{-k}(\partial_\infty \mathbb{H}^2\times \partial_\infty \mathbb{H}^2)}\Big).
    \end{align*}
     Moreover, for every compact set $\mathcal{C}\subset S\mathbf{X}_\Gamma$ the following holds:
     There are constants $C_{\mathcal C,k},C_{\mathcal C,M,k}>0$ and compact sets $\mathcal{K}_{\mathcal{C}}, \mathcal{K}_{\mathcal{C},M}\subset S\mathbf{X}_\Gamma$ such that, for all $j,M\in \N$ and all $u\in \CT(S\mathbf{X}_\Gamma)$ with $\supp u\subset \mathscr C$, one has
    \begin{align*}
        \supp R_{s_j'}(u)&\subset \mathcal{K}_{\mathcal{C}},& \supp L_{M,s_j'}(u)&\subset\mathcal{K}_{\mathcal{C},M},\\
          \big\Vert R_{s_j'}(u)\big\Vert_{H^k_{\mathcal{K}_{\mathcal{C}}}(S\mathbf{X}_\Gamma)}&\leq C_{\mathcal C,k} \norm{u}_{H^k_{\mathcal{C}}(S\mathbf{X}_\Gamma)},&\qquad \big\Vert L_{M,s_j'}(u)\big\Vert_{H^k_{\mathcal{K}_{\mathcal{C},M}}(S\mathbf{X}_\Gamma)}&\leq C_{\mathcal C,M,k} \norm{u}_{H^k_{\mathcal{C}}(S\mathbf{X}_\Gamma)}.
    \end{align*}
\end{thm}

\begin{proof} 
For notational simplicity, we work with $s_0=q+ir,s'_0=q'-ir$, noting that the argument holds unchanged for sequences $s_j,s'_j$ as in the statement. 

Note that by \eqref{eq:PS_Radon} and \eqref{eq:Wexplicit00} both the Patterson-Sullivan distribution and the Wigner distribution act on a function $u\in \CT(S\mathbf{X}_\Gamma)$ by evaluating a $\Gamma$-invariant distribution on $S\mathbb H=G$ at $f=\chi \tilde u\in \CT(G)$. 
Moreover, the $\Gamma$-invariant distribution $(T_{\phi} \otimes \overline T_{\phi'})\circ\mathcal{R}_{s_0,s'_0}\in \D'(G)$ featured on the right hand side of \eqref{eq:PS_Radon} looks similar to the distribution $f\mapsto (T_{\phi} \otimes \overline T_{\phi'})\big(\mathcal{R}_{s_0,s'_0}\big(\tilde{\beta}_{\mathscr C}\mathcal{I}_{s'_0}(f)\big)\big)$ in the asymptotic  formula \eqref{eq:Wignerformula} for the Wigner distribution, where $\mathscr C\subset G$ is a compact set containing $\supp f$. 
To find a precise formula for the difference between the two distributions evaluated at a function $f\in \CT(G)$ supported in a given compact set $\mathscr C\subset G$, we  compare $f$ with the function $\tilde{\beta}_{\mathscr C}\mathcal{I}_{s'_0}(f)\in \CT(G)$. 
To this end, we perform a stationary phase expansion of the oscillatory integral defining the function $\mathcal{I}_{s'_0}(f)\in \Cinft(G)$ in \eqref{eq:intw}: 
Let $g\in G$  and apply \cite[(5.15) and (5.16)]{HHS12} (where in our setting their variable $\nu'$ is given by $\nu'=2\varrho$ and the asymptotic parameter is $h=r^{-1}$, as already observed on p.~\pageref{nuHHS})  to the amplitude
\[
    f_{q',g}(n)\coloneq e^{(\frac{1}2-q')2\varrho(\mathscr{A}(n^{-1}\omega_0))}f(gn),\quad n\in N_+.
\]
Here the presence of the factor $(\frac{1}2-q')$, as opposed to just $-q'$, is due to the fact that the measure $\dj n$ used in \cite[(5.15) and (5.16)]{HHS12} involves a $\varrho$-shift. 
Then from \cite[(5.16)]{HHS12} we get that $\mathcal I_{s'_0}(f)(g)$ has the following asymptotic expansion:
\begin{equation} 
    \mathcal I_{s'_0}(f)(g)=
    \frac{e^{-i\frac{\pi}{4}}}{\sqrt{\pi}} r^{-1/2} \big(f(g) + \mathcal{O}(r^{-1})\big) \qquad \text{as }r\to +\infty,\label{eq:estiIntwOp}
\end{equation}
where we note\footnote{Using \cite[(5.13)]{HHS12} in our setting, we have (in the notation of the reference) $\kappa(-2\varrho)=(\sqrt{2}\pi)^{-1}e^{i\pi/4}$ and $C_N=(\sqrt{2}\pi)^{-1}$ (see \eqref{eq:dxdadn}), and $\|\alpha\|=1$.} that $f_{q',g}(e)=e^{(\frac{1}2-q')2\varrho(\mathscr{A}(\omega_0))}f(g)=f(g)$ since $\mathscr{A}(\omega_0)=0$. 
For any finite family $X_1,\ldots, X_N$  ($N \in \N_0$) of smooth vector fields $X_j$ on $G$ one has 
\[
    X_1\cdots X_N\tilde{\beta}_{\mathscr C} \mathcal I_{s'_0}(f) = \sum_{0\leq l\leq N} \tau_l\mathcal I_{s'_0}(X_1\cdots X_l f)
\]
with some functions $\tau_l\in \Cinft(G)$ given by derivatives of $\tilde{\beta}_{\mathscr C}$, so that the inclusion 
\[
    \supp (\tau_l \mathcal I_{s'_0}(X_1\cdots X_l f))\subset \supp (\tilde{\beta}_{\mathscr C}\mathcal I_{s'_0}(X_1\cdots X_l f)){\subset \mathscr{K}_{\mathscr C}}
\]
holds, where $\mathscr{K}_{\mathscr C}\subset G$ is the compact set from \eqref{eq:betaIcptsupp}. 
In particular, the support of each function $\tau_l\mathcal I_{s'_0}(X_1\cdots X_l f)$ is compact by \eqref{eq:betaIcptsupp}. 
Since the pointwise estimates \eqref{eq:estiIntwOp} apply to each of the functions $X_1\cdots X_l f$, recalling \eqref{eq:Sobolevnormdef} we arrive at an estimate in $H^k_{\mathscr K_{\mathscr C}}(G)$ for all $k\in \N_0$:
\begin{equation} \label{eq:intw_asymptotic}
   \tilde{\beta}_{\mathscr{C}} \mathcal I_{s'_0}(f)=\frac{e^{-i\frac{\pi}{4}}}{\sqrt{\pi}}r^{-1/2} \big(f + \mathrm{Rem}^{(1)}_{s_0'}(f)\big),
\end{equation}
where $\mathrm{Rem}^{(1)}_{s_0'}(f)\in \CT(G)$ is supported in $\mathscr{K}_{\mathscr C}$ and satisfies 
\bq \label{eq:remasymp2}
    \|\mathrm{Rem}^{(1)}_{s_0'}(f)\|_{H^k_{\mathscr K_{\mathscr C}}(G)}\leq C_{\mathscr C,k}(1+r)^{-1} \|f\|_{H^k_{\mathscr C}(G)},\qquad k\in \N_0,
\eq
with some constant $C_{\mathscr C,k}>0$ independent of $s_0'$ and $f$.

Here  to get the leading term in \eqref{eq:intw_asymptotic}  we used that, as a consequence of \eqref{eq:gAKcapCK}, the function $\tilde{\beta}_{\mathscr C}$ is equal to $1$ on $\supp f \subset\mathscr C$.

Now we proceed inductively as in \cite[proof of Thm.~7.4]{HHS12} (which is essentially a variant of the original argument in \cite[Sec.~4.2]{AZ07}): 

Define $\mathrm{Rem}^{(0)}_{s_0'}(f):=f$,  $\mathscr{K}_{\mathscr C,0}:=\mathscr C$, and define for $l\in \N$
\begin{align*}
    \mathscr{K}_{\mathscr C,l}&:=\mathscr{K}_{\mathscr{K}_{\mathscr C,l-1}}\subset G,\\
    \mathrm{Rem}^{(l)}_{s_0'}(f)&:=\mathrm{Rem}^{(1)}_{s_0'}(\mathrm{Rem}^{(l-1)}_{s_0'}(f))\in \CT(G),\quad \supp \mathrm{Rem}^{(l)}_{s_0'}(f)\subset \mathscr{K}_{\mathscr C,l},
\end{align*}
so that for all $k,l\in \N_0$ there is a constant $C_{\mathscr C,k,l}>0$ such that
\bq \label{eq:remasympCompact}
    \|\mathrm{Rem}^{(l)}_{s_0'}(f)\|_{H^k_{\mathscr{K}_{\mathscr C,l}}(G)}\leq C_{\mathscr C,k,l}(1+r)^{-l} \|f\|_{H^k_{\mathscr C}(G)},\quad k,l\in \N_0,r>0.
\eq
We then have for each $N\in \N$
\bqn
    \tilde{\beta}_{\mathscr{C}} \mathcal I_{s'_0}(f)-\sum_{l=1}^N \tilde{\beta}_{\mathscr C_l} \mathcal I_{s'_0}(\mathrm{Rem}^{(l)}_{s_0'}(f))=\frac{e^{-i\frac{\pi}{4}}}{\sqrt{\pi}}r^{-1/2} \big(f - \mathrm{Rem}^{(N+1)}_{s_0'}(f)\big).
\eqn
Applying the distribution $(T_{\phi} \otimes \overline T_{\phi'})\circ\mathcal{R}_{s_0,s'_0}$ and using Lemma \ref{lem:Wignerleadingterm} gives us
\begin{align}\begin{split}\label{eq:W_split}
    \widetilde W_{\phi,\phi'}(f)-\sum_{l=1}^N \widetilde W_{\phi,\phi'}(\mathrm{Rem}^{(l)}_{s_0'}(f))&=\frac{e^{-i\frac{\pi}{4}}}{\sqrt{\pi}}r^{-1/2} \Big((T_{\phi} \otimes \overline T_{\phi'})\big(\mathcal{R}_{s_0,s'_0}(f)\big) \\
    &\qquad \qquad\quad  -(T_{\phi} \otimes \overline T_{\phi'})\big(\mathcal{R}_{s_0,s'_0}\big(\mathrm{Rem}^{(N+1)}_{s_0'}(f)\big)\big)\Big)\\
    &\quad +(T_{\phi} \otimes \overline T_{\phi'})\big(\mathrm{Rem}^{-\infty}_{N,s_0,s_0'}(f)\big),\end{split}
\end{align}
where the remainder $\mathrm{Rem}^{-\infty}_{N,s_0,s_0'}(f)\in \Cinft(\partial_\infty \mathbb H^2\times\partial_\infty \mathbb H^2)$ is obtained by collecting the $\mathcal{O}(r^{-\infty})$-remainders from the $N$ applications of Lemma \ref{lem:Wignerleadingterm} and satisfies
\bqn
    \mathrm{Rem}^{-\infty}_{N,s_0,s_0'}(f)=\mathcal{O}_{\Cinft(\partial_\infty \mathbb H^2\times\partial_\infty \mathbb H^2)}(r^{-\infty})\qquad \text{as }r\to +\infty.
\eqn
On the other hand, applying Lemma \ref{lem:Wignerleadingterm} directly to \eqref{eq:intw_asymptotic} gives us
\bq \label{eq:W_tilde_explicit_Rem}
    \widetilde W_{\phi,\phi'}(f)=\frac{e^{-i\frac{\pi}{4}}}{\sqrt{\pi}}r^{-1/2} \Big((T_{\phi} \otimes \overline T_{\phi'})\big(\mathcal{R}_{s_0,s'_0}\big(f+\mathrm{Rem}^{(1)}_{s_0'}(f)\big)\big)\Big)
    +(T_{\phi} \otimes \overline T_{\phi'})\big(\mathrm{Rem}_{s_0,s_0'}(f)\big) 
\eq
with $\mathrm{Rem}_{s_0,s_0'}(f)=\mathcal{O}_{\Cinft(\partial_\infty \mathbb H^2\times\partial_\infty \mathbb H^2)}(r^{-\infty})$ as in Lemma \ref{lem:Wignerleadingterm}.

Now, given $k,M\in \N_0$, then by \eqref{eq:remasympCompact} and Lemma \ref{lem:radoncontinuity} the term $\mathcal{R}_{s_0,s'_0}\big(\mathrm{Rem}^{(N+1)}_{s_0'}(f)\big)$ in \eqref{eq:W_split} is of size $\mathcal{O}_{H^k(\partial_\infty \mathbb H^2\times\partial_\infty \mathbb H^2)}(r^{-M})$ for $N$ large enough. 
Fixing such an $N$, the term $\mathrm{Rem}^{-\infty}_{N,s_0,s_0'}(f)$ is also of size $\mathcal{O}_{H^k(\partial_\infty \mathbb H^2\times\partial_\infty \mathbb H^2)}(r^{-M})$ (here we use that the $\mathcal{O}_{\Cinft(\partial_\infty \mathbb H^2\times\partial_\infty \mathbb H^2)}$-estimates imply Sobolev norm estimates).
Since $H^{-k}(\partial_\infty \mathbb{H}^2\times \partial_\infty \mathbb{H}^2)$ is dual to $H^{k}(\partial_\infty \mathbb{H}^2\times \partial_\infty \mathbb{H}^2)$ (see \cite[Thm.~6.63]{Hintz}), this implies that the terms in the last two lines of \eqref{eq:W_split} are of size $\O(r^{-M}\Vert T_{\phi}\otimes \overline{T}_{\phi'}\Vert_{H^{-k}(\partial_\infty \mathbb{H}^2\times \partial_\infty \mathbb{H}^2)})$ when $k\geq k_C$. 

Similarly, the remainder in \eqref{eq:W_tilde_explicit_Rem} is  $\O(r^{-M}\Vert T_{\phi}\otimes \overline{T}_{\phi'}\Vert_{H^{-k}(\partial_\infty \mathbb{H}^2\times \partial_\infty \mathbb{H}^2)})$ when $k\geq k_C$. 
 In particular, for such $k$,   \eqref{eq:W_tilde_explicit_Rem} proves that $\widetilde W_{\phi,\phi'}\in H^{-k}_\mathrm{loc}(G)$ and thus  $W_{\phi,\phi'}\in H^{-k}_\mathrm{loc}(S\mathbf{X}_\Gamma)$.
Similarly, we infer from \eqref{eq:W_tilde_explicit_Rem} and \eqref{eq:PS_Radon} that $\mathrm{PS}_{\phi_j,\phi_j'}\in H^{-k}_{\mathrm{loc}}(S\mathbf{X}_\Gamma)$.

Finally, we rewrite the leading terms in the claimed form.
To this end, we now return to the setup at the beginning of the proof by plugging in $f=\chi \tilde u$ and fixing a compact set $\mathscr C\subset G$ with $\supp \chi\subset \mathscr C$, and then we use Lemma \ref{lem:newlemma}: 
This allows us to write the left-hand side of \eqref{eq:W_split} as $W_{\phi,\phi'}(u+L_{N,s_0'}(u)r^{-1})$ where, in the notation of Lemma \ref{lem:newlemma}, we have $L_{N,s_0'}(u):=r\,u_{f_N,s_0'}$ with
\[
    f_{N,s_0'}:=\sum_{l=1}^N \mathrm{Rem}^{(l)}_{s_0'}(\chi \tilde u).
\]
Let $\mathcal C\subset S\mathbf{X}_\Gamma$ be compact and assume that $u\in \CT(S\mathbf{X}_\Gamma)$ satisfies $\supp u\subset \mathcal C$.
Inspecting the proof of Lemma \ref{lem:newlemma} and recalling the support properties of the remainders in \eqref{eq:remasympCompact}, we see that there is a compact set $\mathcal{K}_{\mathcal C,N}\subset S\mathbf{X}_\Gamma$ depending only on $N$ and $\mathcal C$, such that $\supp L_{N,s_0'}(u)\subset \mathcal{K}_{\mathcal C,N}$, and for all $k\in \N_0$ there is a constant $C_{\mathcal C,N,k}>0$, also independent of $s_0'$, such that 
\bq
    \|L_{N,s_0'}(u)\|_{H^k_{\mathcal{K}_{\mathcal C,N}}(S\mathbf{X}_\Gamma)}\leq C_{\mathcal C,N,k} \|u\|_{H^k_{\mathcal C}(S\mathbf{X}_\Gamma)}
\eq
with the operator $L_{N,s_0'}: \CT(S\mathbf{X}_\Gamma) \rightarrow \CT(S\mathbf{X}_\Gamma)$. 

Now using \eqref{eq:PS_Radon}, we can rewrite \eqref{eq:W_split} for $f=\chi \tilde u$ in terms of Patterson-Sullivan distributions in the form
\begin{multline*}
    W_{\phi,\phi'}\big(u+L_{N,s_0'}(u)r^{-1}\big) = \frac{e^{-i\frac{\pi}{4}}}{\sqrt{\pi} }\, r^{-1/2} {\mathrm{PS}}_{\phi,\phi'}(u)\\
    +\O_ {H^{-k}_{\mathrm{loc}}(S\mathbf{X}_\Gamma)}\Big(r^{-M}\Vert T_{\phi}\otimes \overline{T}_{\phi'}\Vert_{H^{-k}(\partial_\infty \mathbb{H}^2\times \partial_\infty \mathbb{H}^2)}\Big),
\end{multline*}
which is precisely the second claimed formula in Theorem~\ref{thm:precise} when defining $L_{M,s_0'}:=L_{N,s_0'}$ for some arbitrary large enough $N$ depending on $M$. 

Analogously, we get the first claimed formula in Theorem~\ref{thm:precise} from \eqref{eq:W_tilde_explicit_Rem} by defining $R_{s_0'}(u):=r\,u_{f_{s_0'}}$, where $f_{s_0'}:=\mathrm{Rem}^{(1)}_{s_0'}(\chi \tilde u)$, and recalling \eqref{eq:remasymp2}.
\end{proof}

We can now deduce Theorem \ref{thm:asymptotic_WPS}.

\begin{proof}[Proof of Theorem \ref{thm:asymptotic_WPS}]\label{proof:main}
By Definition \ref{def:moderatenormalization},  the sequence $\{(\phi_j,\phi_j')\}_{j\in \N}$ being moderately $H^{-k}$-normalized means that for some $N'\in \N$ we have the Sobolev estimate 
\[
\Vert T_{\phi_j}\otimes \overline{T}_{\phi'_j}\Vert_{H^{-k}(\partial_\infty \mathbb{H}^2\times \partial_\infty \mathbb{H}^2)}=\O(r_j^{N'}) \quad \text{ as } j \rightarrow \infty.
\]
Hence,  the $\O(r_j^{-N})$-estimates in Theorem \ref{thm:asymptotic_WPS} are obtained by taking $M=N+N'$ in Theorem \ref{thm:precise}. 
We now write  $c:=\frac{e^{-i\frac{\pi}{4}}}{\sqrt{\pi}}$ and regroup the equations in Theorem \ref{thm:precise}:
\begin{align*}
    W_{\phi_j,\phi'_j} -  c r_j^{-1/2} {\mathrm{PS}}_{\phi_j,\phi_j'}&= c r_j^{-3/2} {\mathrm{PS}}_{\phi_j,\phi_j'}\circ R_{s_j'}+\O_{H^{-k}_{\mathrm{loc}}(S\mathbf{X}_\Gamma)}(r_j^{-N}),\\
    W_{\phi_j,\phi'_j}-c r_j^{-1/2} {\mathrm{PS}}_{\phi_j,\phi_j'}&=-r_j^{-1}W_{\phi_j,\phi'_j}\circ L_{M,s_j'}+\O_{H^{-k}_{\mathrm{loc}}(S\mathbf{X}_\Gamma)}(r_j^{-N}),
\end{align*}
with the operators $R_{s_j'},L_{M,s_j'}: \CT(S\mathbf{X}_\Gamma) \rightarrow \CT(S\mathbf{X}_\Gamma)$. 

Then we estimate for a given compact set $\mathcal C\subset  S\mathbf{X}_\Gamma$ using Sobolev duality (which involves some universal constant $C'_k>0$ independent of the operators we are considering):
\begin{align*}
     \big\Vert{\mathrm{PS}}_{\phi_j,\phi_j'}\circ R_{s_j'}\big\Vert_{H^{-k}_{\mathcal K_{\mathcal C}}(S\mathbf{X}_\Gamma)}&\leq C'_k\big\Vert{\mathrm{PS}}_{\phi_j,\phi_j'}\big\Vert_{H^{-k}_{\mathcal K_{\mathcal C}}(S\mathbf{X}_\Gamma)}\big\Vert R_{s_j'}\big\Vert_{H^k_{\mathcal{C}}(S\mathbf{X}_\Gamma)\to H^{k}_{\mathcal K_{\mathcal C}}(S\mathbf{X}_\Gamma)}\\
     &\leq C'_k C_{\mathcal C,k}\big\Vert{\mathrm{PS}}_{\phi_j,\phi_j'}\big\Vert_{H^{-k}_{\mathcal K_{\mathcal C}}(S\mathbf{X}_\Gamma)},\\
     \big\Vert W_{\phi_j,\phi'_j}\circ L_{M,s_j'}\big\Vert_{H^{-k}_{\mathcal K_{\mathcal C,M}}(S\mathbf{X}_\Gamma)}&\leq C'_k\big\Vert W_{\phi_j,\phi'_j}\big\Vert_{H^{-k}_{\mathcal K_{\mathcal C,M}}(S\mathbf{X}_\Gamma)}\big\Vert L_{M,s_j'}(u)\big\Vert_{H^k_{\mathcal{C}}(S\mathbf{X}_\Gamma)\to H^{k}_{\mathcal K_{\mathcal C,M}}(S\mathbf{X}_\Gamma)}\\
     &\leq C'_kC_{\mathcal C,M,k}\big\Vert W_{\phi_j,\phi'_j}\big\Vert_{H^{-k}_{\mathcal K_{\mathcal C,M}}(S\mathbf{X}_\Gamma)}.
\end{align*}

Recalling the meaning of the notation $\O_{H^{-k}_{\mathrm{loc}}(S\mathbf{X}_\Gamma)}(r_j^{-N})$ from \eqref{eq:OHkdef} and performing notational cosmetics by writing $\mathscr C$ and $\mathscr K$ instead of $\mathcal C$ and $\mathcal K$ and $N$ instead of $M=N+N'$, the proof of Theorem \ref{thm:asymptotic_WPS} finished. 
\end{proof}

%-------------------------------%%--------------------------------------%
\newcommand{\etalchar}[1]{$^{#1}$}
\providecommand{\bysame}{\leavevmode\hbox to3em{\hrulefill}\thinspace}
%\providecommand{\MR}{\relax\ifhmode\unskip\space\fi MR }
%\providecommand{\MRhref}[2]{%
%  \href{http://www.ams.org/mathscinet-getitem?mr=#1}{#2}
%}
%\providecommand{\href}[2]{#2}

%---------------------------------------------------------------------%

\bigskip

\end{document}